\documentclass[11pt]{amsart}
\usepackage{amsthm,amsmath,amssymb}

\usepackage{geometry}
\usepackage{graphicx}

\numberwithin{equation}{section}
\numberwithin{figure}{section}
\theoremstyle{plain}
\newtheorem{thm}{Theorem}[section]
\newtheorem{lem}[thm]{Lemma}

\theoremstyle{remark}

\newcommand{\M}{\operatorname{M}}
\newcommand{\Hf}{\operatorname{H}}

\newcommand{\Su}{\operatorname{S}}
\newcommand{\od}{\operatorname{\textbf{o}}}
\newcommand{\e}{\operatorname{\textbf{e}}}
\newcommand{\s}{\operatorname{\textbf{s}}}
\newcommand{\T}{\operatorname{T}}
\newcommand{\V}{\operatorname{V}}
\newcommand{\Pn}{\operatorname{P}}
\newcommand{\Q}{\operatorname{Q}}
\newcommand{\K}{\operatorname{K}}
\begin{document}

\title{Tiling enumeration of doubly--intruded halved hexagons}

\author{Tri Lai}
\address{Department of Mathematics, University of Nebraska -- Lincoln, Lincoln, NE 68588}
\email{tlai3@unl.edu}
\thanks{T.L. was supported in part  by Simons Foundation Collaboration Grant (\# 585923).}



\date{\today}

\dedicatory{}

\begin{abstract}
Inspired by Propp's intruded Aztec diamond regions, we consider halved hexagons in which two aligned arrays of triangular holes
 have been removed from their boundaries. Unlike the intruded Aztec diamonds (whose numbers of
  domino tilings contain some large prime factors in their factorizations), the numbers of lozenge tilings of our doubly-intruded halved hexagons are given
by simple product formulas in which all factors are linear in the parameters. In this paper, we present an extensive list of exact tiling enumerations of sixteen different types of doubly-intruded halved hexagons.
We also prove that the lozenge tilings of a symmetric hexagon with three arrays of triangles removed are always enumerated by a closed-form product formula.
Our results generalize several previous work, including Proctor's enumeration of the transposed--complementary plane partitions, related work of Ciucu,  and recent generalizations of
Rohatgi and of the author.
\end{abstract}

\maketitle
\section{Introduction}\label{sec:Intro}
MacMahon's classical theorem on plane partitions fitting in a given box is equivalent to the fact that the number of lozenge tilings of a centrally symmetric
hexagon with side-lengths $a,b,c,a,b,c$ (in a cyclic order) on the triangular lattice\footnote{ We consider the triangular lattice drawn so that one family of lattice lines is horizontal.} is given by the beautiful product formula
\begin{equation}
\prod_{i=1}^{a}\prod_{j=1}^{b}\prod_{k=1}^{c}\frac{i+j+k-1}{i+j+k-2}.
\end{equation}
Here a \emph{lozenge} is a shape made by any two unit equilateral triangles sharing an edge, and a \emph{lozenge tiling} of
a region is a covering of the region by lozenges, such that there are no gaps or overlaps.

 The beauty of MacMahon's formula motivates ones to consider more classes of boxed plane partitions. As an important topic in enumerative combinatorics,
 the study of symmetric plane partitions has got a lot of attention in the last few decades (see e.g. \cite{Stanley}, \cite{Kup},
\cite{Stem}, \cite{Andrews}, \cite{KKZ}). Each of the 10 symmetry classes of plane partitions is equivalent to a certain type of symmetric lozenge tilings of a hexagon, and is
 enumerated by  a simple product formula.
 In this paper, we focus on one of the symmetry classes, the \emph{transposed-complementary plane partitions}, that is equivalent to the lozenge tilings of a hexagon which are invariant
 under the refection over a vertical symmetry axis. The latter in turn are in bijection with the lozenge tilings of a \emph{halved hexagon},
 the region
 obtained by dividing a symmetric hexagon by a zigzag lattice path along its vertical symmetry axis.  This symmetry class was first
 enumerated by Proctor \cite{Proc}.
Proctor actually enumerated a certain class of staircase plane partitions that are in bijection with the
 lozenge tilings of a hexagon with a maximal staircase cut off (see Figure \ref{halfhex6}(a)). We refer the reader to \cite{Cutoff} for
 a number of related tiling enumerations.

Our regions are inspired by Propp's `\emph{intruded Aztec diamonds}' described in Problem 11 of his well-known survey paper \cite{Propp}.
In particular, we consider \emph{sixteen} different variations of halved hexagons in which ``intrusions''
 are made by two `\emph{ferns}' lined up along a common horizontal lattice line (see Figures \ref{fig:halvedhex1}, \ref{fig:halvedhex2}, \ref{fig:halvedhex3}, and \ref{fig:halvedhex4} for examples). Here a \emph{fern} is an array triangles with alternating orientations  (up-pointing and down-pointing).
The resulting regions turn out to have the numbers of lozenge tilings given by simple product formulas. In particular, the factor in the prime
factorizations of the tiling numbers are all linear in the parameters
(this is not the case for the intruded Aztec diamonds).

Our results generalize Proctor's enumeration and its weighted version due to Ciucu \cite{Ciucu1}. Our results also have Rohatgi's
work \cite{Ranjan}  and the author's previous work \cite{Halfhex1, Halfhex2} as special cases.


In this paper, we consider in addition symmetric hexagons with three ferns removed on the same horizontal lattice line as follows.

Let $x,y,z$ be three non-negative integers, and let $\textbf{a}=\{a_i\}_{i=1}^n$ and
$\textbf{b}=\{b_j\}_{j=1}^{m}$ be two sequences of non-negative integers.  Set
\begin{align}
\e_a:=\sum_{i\ even} a_i,  \ \ \ &\od_a=\sum_{i \ odd} a_i,\\
 \e_b:=\sum_{j\ even} b_j,  \ \ \ & \od_b=\sum_{j \ odd} b_j.
\end{align}
We consider a hexagon of side-lengths\footnote{From now on, we always list the side lengths of a hexagon in the clockwise order, starting from the north side.} $x+2\e_b+2\e_a,\ y+z+2\od_b+2\od_a-a_1,\ y+z+2\od_b+2\od_a-a_1, x+2\od_b+2\od_a-a_1,\ y+z+2\od_b+2\od_a-a_1,\ y+z+2\od_b+2\od_a-a_1$.

 We remove  two equal ferns whose  triangles are of side lengths $b_1,b_2,\dots,b_n$ at level $z$ above the west vertex of the hexagon,
 one array branches from the northeast side of the hexagon to  the right and the other branches in the opposite direction, from the northwest side to the left. In the middle of the two array,
 we remove a symmetric array of triangles with an $a_1$-triangle in the middle, then two $a_2$-triangles on both sides, and so on
  (see Figure \ref{fig:threearray}).  If the $a_1$-triangle is up-pointing, we denote the region by $S^{(1)}_{x,y,z}(a_1,a_2,\dotsc, a_m; b_1,b_2,\dotsc,b_n)$, otherwise the region is denoted by
 $S^{(2)}_{x,y,z}(a_1,a_2,\dotsc, a_m; b_1,b_2,\dotsc,b_n)$. We prove that the tilings of these two regions are both enumerated by simple product formulas.

\begin{figure}\centering
\setlength{\unitlength}{3947sp}%
\begingroup\makeatletter\ifx\SetFigFont\undefined%
\gdef\SetFigFont#1#2#3#4#5{%
  \reset@font\fontsize{#1}{#2pt}%
  \fontfamily{#3}\fontseries{#4}\fontshape{#5}%
  \selectfont}%
\fi\endgroup%
\resizebox{15cm}{!}{
\begin{picture}(0,0)
\includegraphics{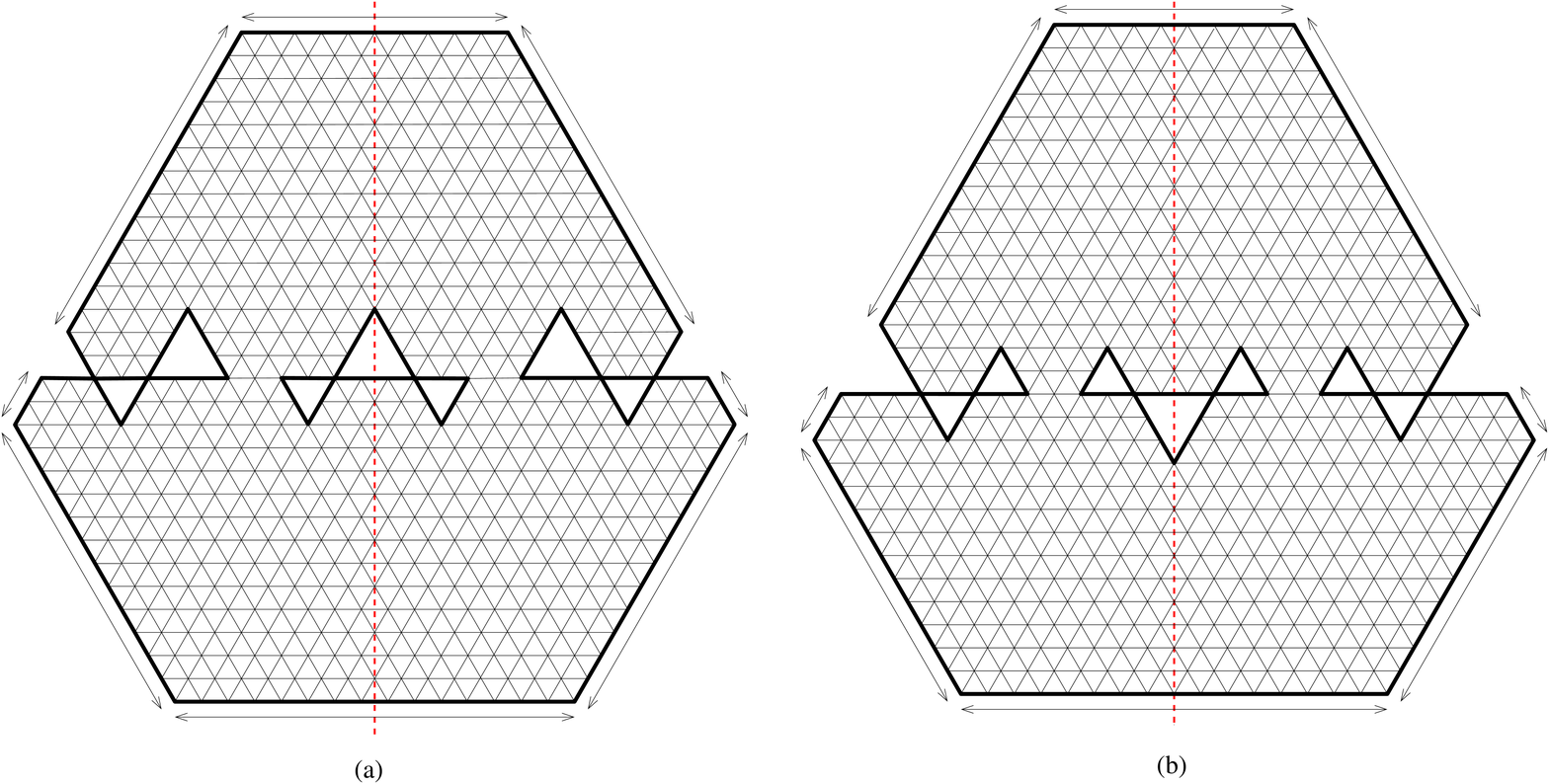}%
\end{picture}%
%
%

\begin{picture}(48039,24812)(8126,-25372)
\put(43381,-1161){\makebox(0,0)[lb]{\smash{{\SetFigFont{42}{40.8}{\rmdefault}{\mddefault}{\itdefault}{$x+a_1+2b_2$}%
}}}}
\put(50701,-4041){\rotatebox{300.0}{\makebox(0,0)[lb]{\smash{{\SetFigFont{32}{40.8}{\rmdefault}{\mddefault}{\itdefault}{$y+2a_2+b_1+2b_3$}%
}}}}}
\put(53121,-21441){\rotatebox{60.0}{\makebox(0,0)[lb]{\smash{{\SetFigFont{34}{40.8}{\rmdefault}{\mddefault}{\itdefault}{$y+z+a_1+2b_2$}%
}}}}}
\put(42101,-23861){\makebox(0,0)[lb]{\smash{{\SetFigFont{34}{40.8}{\rmdefault}{\mddefault}{\itdefault}{$x+2a_2+2b_1+2b_3$}%
}}}}
\put(33621,-17221){\rotatebox{300.0}{\makebox(0,0)[lb]{\smash{{\SetFigFont{34}{40.8}{\rmdefault}{\mddefault}{\itdefault}{$y+z+a_1+2b_2$}%
}}}}}
\put(36081,-8481){\rotatebox{60.0}{\makebox(0,0)[lb]{\smash{{\SetFigFont{34}{40.8}{\rmdefault}{\mddefault}{\itdefault}{$y+2a_2+b_1+2b_3$}%
}}}}}
\put(19801,-12421){\makebox(0,0)[lb]{\smash{{\SetFigFont{34}{40.8}{\rmdefault}{\mddefault}{\itdefault}{$a_1$}%
}}}}
\put(21801,-13481){\makebox(0,0)[lb]{\smash{{\SetFigFont{34}{40.8}{\rmdefault}{\mddefault}{\itdefault}{$a_2$}%
}}}}
\put(17721,-13441){\makebox(0,0)[lb]{\smash{{\SetFigFont{34}{40.8}{\rmdefault}{\mddefault}{\itdefault}{$a_2$}%
}}}}
\put(29181,-12581){\makebox(0,0)[lb]{\smash{{\SetFigFont{34}{40.8}{\rmdefault}{\mddefault}{\itdefault}{$b_1$}%
}}}}
\put(27521,-13521){\makebox(0,0)[lb]{\smash{{\SetFigFont{34}{40.8}{\rmdefault}{\mddefault}{\itdefault}{$b_2$}%
}}}}
\put(25341,-12381){\makebox(0,0)[lb]{\smash{{\SetFigFont{34}{40.8}{\rmdefault}{\mddefault}{\itdefault}{$b_3$}%
}}}}
\put(10434,-12507){\makebox(0,0)[lb]{\smash{{\SetFigFont{34}{40.8}{\rmdefault}{\mddefault}{\itdefault}{$b_1$}%
}}}}
\put(12071,-13452){\makebox(0,0)[lb]{\smash{{\SetFigFont{34}{40.8}{\rmdefault}{\mddefault}{\itdefault}{$b_2$}%
}}}}
\put(14117,-12271){\makebox(0,0)[lb]{\smash{{\SetFigFont{34}{40.8}{\rmdefault}{\mddefault}{\itdefault}{$b_3$}%
}}}}
\put(53151,-12991){\makebox(0,0)[lb]{\smash{{\SetFigFont{34}{40.8}{\rmdefault}{\mddefault}{\itdefault}{$b_1$}%
}}}}
\put(35091,-12991){\makebox(0,0)[lb]{\smash{{\SetFigFont{34}{40.8}{\rmdefault}{\mddefault}{\itdefault}{$b_1$}%
}}}}
\put(37291,-13991){\makebox(0,0)[lb]{\smash{{\SetFigFont{34}{40.8}{\rmdefault}{\mddefault}{\itdefault}{$b_2$}%
}}}}
\put(51271,-14031){\makebox(0,0)[lb]{\smash{{\SetFigFont{34}{40.8}{\rmdefault}{\mddefault}{\itdefault}{$b_2$}%
}}}}
\put(38831,-13131){\makebox(0,0)[lb]{\smash{{\SetFigFont{34}{40.8}{\rmdefault}{\mddefault}{\itdefault}{$b_3$}%
}}}}
\put(49491,-13171){\makebox(0,0)[lb]{\smash{{\SetFigFont{34}{40.8}{\rmdefault}{\mddefault}{\itdefault}{$b_3$}%
}}}}
\put(44200,-14251){\makebox(0,0)[lb]{\smash{{\SetFigFont{34}{40.8}{\rmdefault}{\mddefault}{\itdefault}{$a_1$}%
}}}}
\put(46221,-13091){\makebox(0,0)[lb]{\smash{{\SetFigFont{34}{40.8}{\rmdefault}{\mddefault}{\itdefault}{$a_2$}%
}}}}
\put(42201,-13111){\makebox(0,0)[lb]{\smash{{\SetFigFont{34}{40.8}{\rmdefault}{\mddefault}{\itdefault}{$a_2$}%
}}}}
\put(31241,-13161){\makebox(0,0)[lb]{\smash{{\SetFigFont{34}{40.8}{\rmdefault}{\mddefault}{\itdefault}{$z$}%
}}}}
\put(55761,-13651){\makebox(0,0)[lb]{\smash{{\SetFigFont{34}{40.8}{\rmdefault}{\mddefault}{\itdefault}{$z$}%
}}}}
\put(8141,-13251){\makebox(0,0)[lb]{\smash{{\SetFigFont{34}{40.8}{\rmdefault}{\mddefault}{\itdefault}{$z$}%
}}}}
\put(25621,-3001){\rotatebox{300.0}{\makebox(0,0)[lb]{\smash{{\SetFigFont{34}{40.8}{\rmdefault}{\mddefault}{\itdefault}{$y+a_1+b_1+2b_3$}%
}}}}}
\put(10981,-8881){\rotatebox{60.0}{\makebox(0,0)[lb]{\smash{{\SetFigFont{34}{40.8}{\rmdefault}{\mddefault}{\itdefault}{$y+a_1+b_1+2b_3$}%
}}}}}
\put(18521,-1421){\makebox(0,0)[lb]{\smash{{\SetFigFont{34}{40.8}{\rmdefault}{\mddefault}{\itdefault}{$x+2a_2+2b_2$}%
}}}}
\put(18221,-23961){\makebox(0,0)[lb]{\smash{{\SetFigFont{34}{40.8}{\rmdefault}{\mddefault}{\itdefault}{$x+a_1+2b_1+2b_3$}%
}}}}
\put(28081,-21461){\rotatebox{60.0}{\makebox(0,0)[lb]{\smash{{\SetFigFont{34}{40.8}{\rmdefault}{\mddefault}{\itdefault}{$y+z+2a_2+2b_2$}%
}}}}}
\put(9093,-17231){\rotatebox{300.0}{\makebox(0,0)[lb]{\smash{{\SetFigFont{34}{40.8}{\rmdefault}{\mddefault}{\itdefault}{$y+z+2a_2+2b_2$}%
}}}}}
\put(33081,-13731){\makebox(0,0)[lb]{\smash{{\SetFigFont{50}{50}{\rmdefault}{\mddefault}{\itdefault}{$z$}%
}}}}
\end{picture}}
\caption{The symmetric hexagons with three ferns removed. (a)
The region $S^{(1)}_{2,2,2}(3,2;\ 2,2,3)$.  (b) The region $S^{(2)}_{2,2,2}(3,2;\ 3,2,2)$}\label{fig:threearray}
\end{figure}

It is worth noticing that by the same motivation from Propp's  intruded Aztec diamonds, Ciucu and the author \cite{CL} have considered full hexagons
with two ferns removed from the boundary.  We refer the reader to \cite{Ciucu2}, \cite{Halfhex1}, and \cite{Halfhex2} for more recent discussions on
 the fern structure.



The rest of this paper is organized as follows. Due to a large number of doubly-intruded halved hexagons needed to define, we leave the precise statements of our main results to Section \ref{sec:Statement}.
 In Section \ref{sec:Prelim}, we present several fundamental results in the enumeration of tilings. For ease of reference, we also quote the particular version of Kuo condensation \cite{Kuo} and
 Ciucu's factorization theorem \cite{Ciucu3} that will be employed in our proofs. Section \ref{sec:Proof} is devoted to the proof of the main theorem. Finally, we conclude the paper by posing several open questions in Section \ref{sec:Question}.

\section{Statement of the main result}\label{sec:Statement}

Consider a hexagon of side-lengths $a,b,c,a,b,c$ on the triangular lattice. Assume that $a\leq b$ and that a maximal staircase has
been cut off from the west corner of the hexagon.
 Let $\mathcal{P}_{a,b,c}$ denote the resulting region (see Figure \ref{halfhex6}(a)).
\begin{thm}[Proctor \cite{Proc}]\label{Proctiling}For any non-negative integers $a,$ $b$, and $c$ with $a\leq b$, we have
\begin{equation}
\M(\mathcal{P}_{a,b,c})=\prod_{i=1}^{a}\left[\prod_{j=1}^{b-a+1}\frac{c+i+j-1}{i+j-1}\prod_{j=b-a+2}^{b-a+i}\frac{2c+i+j-1}{i+j-1}\right],
\end{equation}
where empty products are taken to be 1. Here we use the notation $\M(\mathcal{R})$ for the number of tilings\footnote{We only consider regions on the triangular lattice in this paper. Therefore,  from now only, we use the words  ``\emph{region(s)}" and ``\emph{tiling(s)}" to mean ``\emph{ region(s) on the triangular lattice}" and ``\emph{lozenge tiling(s)}", respectively.} of the region $\mathcal{R}$.
\end{thm}

We note that when $a=b$, the region $\mathcal{P}_{a,b,c}$ becomes a ``\emph{halved hexagon}", and
 that Proctor's theorem yields a closed-form product formula for the number of transposed-complimentary plane partitions.
In this paper, we also call $\mathcal{P}_{a,b,c}$ a  halved hexagon (with a defect).

\begin{figure}
  \centering
  \includegraphics[width=10cm]{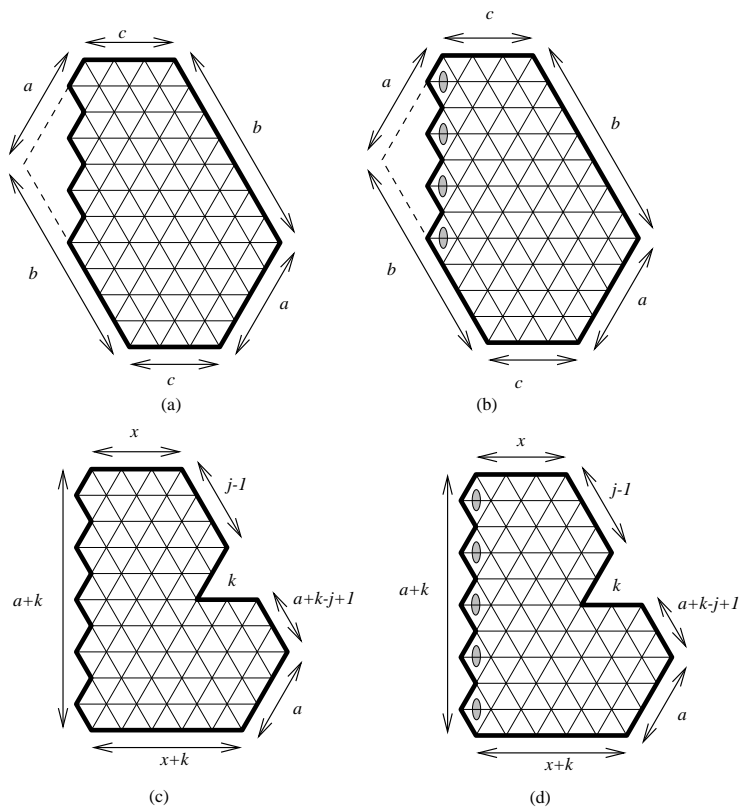}
  \caption{(a) Halved hexagon (with defect) $\mathcal{P}_{4,7,3}$. (b) The weighted halved hexagon (with defect) $\mathcal{P}'_{4,7,3}$. (c)--(d) The regions in Rohatgi's paper \cite{Ranjan}.}\label{halfhex6}
\end{figure}



Lozenges in a region can carry `weights'. In this case, we use the notation $\M(\mathcal{R})$ for the sum of weights of all tilings of $\mathcal{R}$,
where the \emph{weight} of a tiling is the weight product of its constituent lozenges. We are also interested in the weighted counterpart
$\mathcal{P}'_{a,b,c}$  of $\mathcal{P}_{a,b,c}$  where the vertical lozenges along the west side of the region are all weighted by $\frac{1}{2}$ (see the lozenges with shaded `cores' in Figure \ref{halfhex6}(b)).
 In this weight assignment, a tiling has the weight $\left(\frac{1}{2}\right)^n$, where $n$ is the number of vertical lozenges running along the west side.
 M. Ciucu \cite{Ciucu1} proved the following weighted counterpart of Theorem \ref{Proctiling}.

\begin{thm}\label{Ciuculem} For any non-negative integers $a,$ $b$, and $c$ with $a\leq b$
\begin{equation}
\M(\mathcal{P}'_{a,b,c})=2^{-a}\prod_{i=1}\frac{2c+b-a+i}{c+b-a+i}\prod_{i=1}^{a}\left[\prod_{j=1}^{b-a+1}\frac{c+i+j-1}{i+j-1}\prod_{j=b-a+2}^{b-a+i}\frac{2c+i+j-1}{i+j-1}\right].
\end{equation}
\end{thm}

\begin{figure}\centering
\setlength{\unitlength}{3947sp}%
\begingroup\makeatletter\ifx\SetFigFont\undefined%
\gdef\SetFigFont#1#2#3#4#5{%
  \reset@font\fontsize{#1}{#2pt}%
  \fontfamily{#3}\fontseries{#4}\fontshape{#5}%
  \selectfont}%
\fi\endgroup%
\resizebox{12cm}{!}{
\begin{picture}(0,0)%
\includegraphics{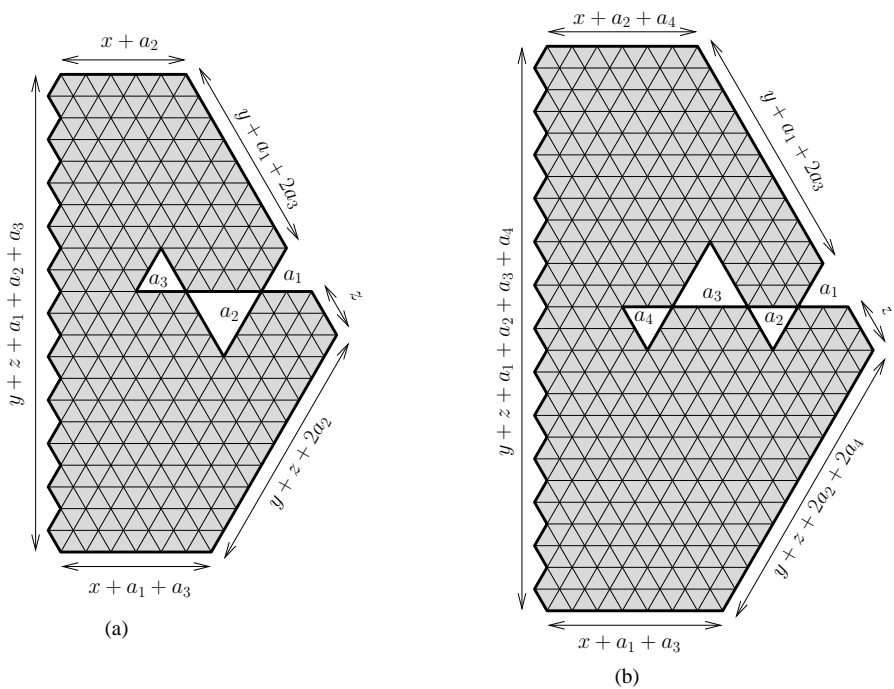}%
\end{picture}%

\begin{picture}(13959,10649)(1761,-9926)
\put(3256, 74){\makebox(0,0)[lb]{\smash{{\SetFigFont{20}{16.8}{\rmdefault}{\mddefault}{\itdefault}{$x+a_2$}%
}}}}
\put(5371,-931){\rotatebox{300.0}{\makebox(0,0)[lb]{\smash{{\SetFigFont{20}{16.8}{\rmdefault}{\mddefault}{\itdefault}{$y+a_1+2a_3$}%
}}}}}
\put(6091,-3541){\makebox(0,0)[lb]{\smash{{\SetFigFont{20}{16.8}{\rmdefault}{\mddefault}{\itdefault}{$a_1$}%
}}}}
\put(5100,-4126){\makebox(0,0)[lb]{\smash{{\SetFigFont{20}{16.8}{\rmdefault}{\mddefault}{\itdefault}{$a_2$}%
}}}}
\put(4000,-3601){\makebox(0,0)[lb]{\smash{{\SetFigFont{20}{16.8}{\rmdefault}{\mddefault}{\itdefault}{$a_3$}%
}}}}
\put(7111,-3751){\rotatebox{300.0}{\makebox(0,0)[lb]{\smash{{\SetFigFont{20}{16.8}{\rmdefault}{\mddefault}{\itdefault}{$z$}%
}}}}}
\put(6001,-7111){\rotatebox{60.0}{\makebox(0,0)[lb]{\smash{{\SetFigFont{20}{16.8}{\rmdefault}{\mddefault}{\itdefault}{$y+z+2a_2$}%
}}}}}
\put(3061,-8416){\makebox(0,0)[lb]{\smash{{\SetFigFont{20}{16.8}{\rmdefault}{\mddefault}{\itdefault}{$x+a_1+a_3$}%
}}}}
\put(1996,-5506){\rotatebox{90.0}{\makebox(0,0)[lb]{\smash{{\SetFigFont{20}{16.8}{\rmdefault}{\mddefault}{\itdefault}{$y+z+a_1+a_2+a_3$}%
}}}}}
\put(9601,-6286){\rotatebox{90.0}{\makebox(0,0)[lb]{\smash{{\SetFigFont{20}{16.8}{\rmdefault}{\mddefault}{\itdefault}{$y+z+a_1+a_2+a_3+a_4$}%
}}}}}
\put(10666,-9271){\makebox(0,0)[lb]{\smash{{\SetFigFont{20}{16.8}{\rmdefault}{\mddefault}{\itdefault}{$x+a_1+a_3$}%
}}}}
\put(13831,-8176){\rotatebox{60.0}{\makebox(0,0)[lb]{\smash{{\SetFigFont{20}{16.8}{\rmdefault}{\mddefault}{\itdefault}{$y+z+2a_2+2a_4$}%
}}}}}
\put(13533,-661){\rotatebox{300.0}{\makebox(0,0)[lb]{\smash{{\SetFigFont{20}{16.8}{\rmdefault}{\mddefault}{\itdefault}{$y+a_1+2a_3$}%
}}}}}
\put(14416,-3766){\makebox(0,0)[lb]{\smash{{\SetFigFont{20}{16.8}{\rmdefault}{\mddefault}{\itdefault}{$a_1$}%
}}}}
\put(13561,-4201){\makebox(0,0)[lb]{\smash{{\SetFigFont{20}{16.8}{\rmdefault}{\mddefault}{\itdefault}{$a_2$}%
}}}}
\put(12600,-3826){\makebox(0,0)[lb]{\smash{{\SetFigFont{20}{16.8}{\rmdefault}{\mddefault}{\itdefault}{$a_3$}%
}}}}
\put(11536,-4186){\makebox(0,0)[lb]{\smash{{\SetFigFont{20}{16.8}{\rmdefault}{\mddefault}{\itdefault}{$a_4$}%
}}}}
\put(10591,434){\makebox(0,0)[lb]{\smash{{\SetFigFont{20}{16.8}{\rmdefault}{\mddefault}{\itdefault}{$x+a_2+a_4$}%
}}}}
\end{picture}}
\caption{Halved hexagons with an array of triangles removed from their northeast sides. The regions were enumerated in \cite{Halfhex1}.} \label{halfhex13}
\end{figure}

\begin{figure}\centering
\setlength{\unitlength}{3947sp}%
\begingroup\makeatletter\ifx\SetFigFont\undefined%
\gdef\SetFigFont#1#2#3#4#5{%
  \reset@font\fontsize{#1}{#2pt}%
  \fontfamily{#3}\fontseries{#4}\fontshape{#5}%
  \selectfont}%
\fi\endgroup%
\resizebox{13cm}{!}{
\begin{picture}(0,0)%
\includegraphics{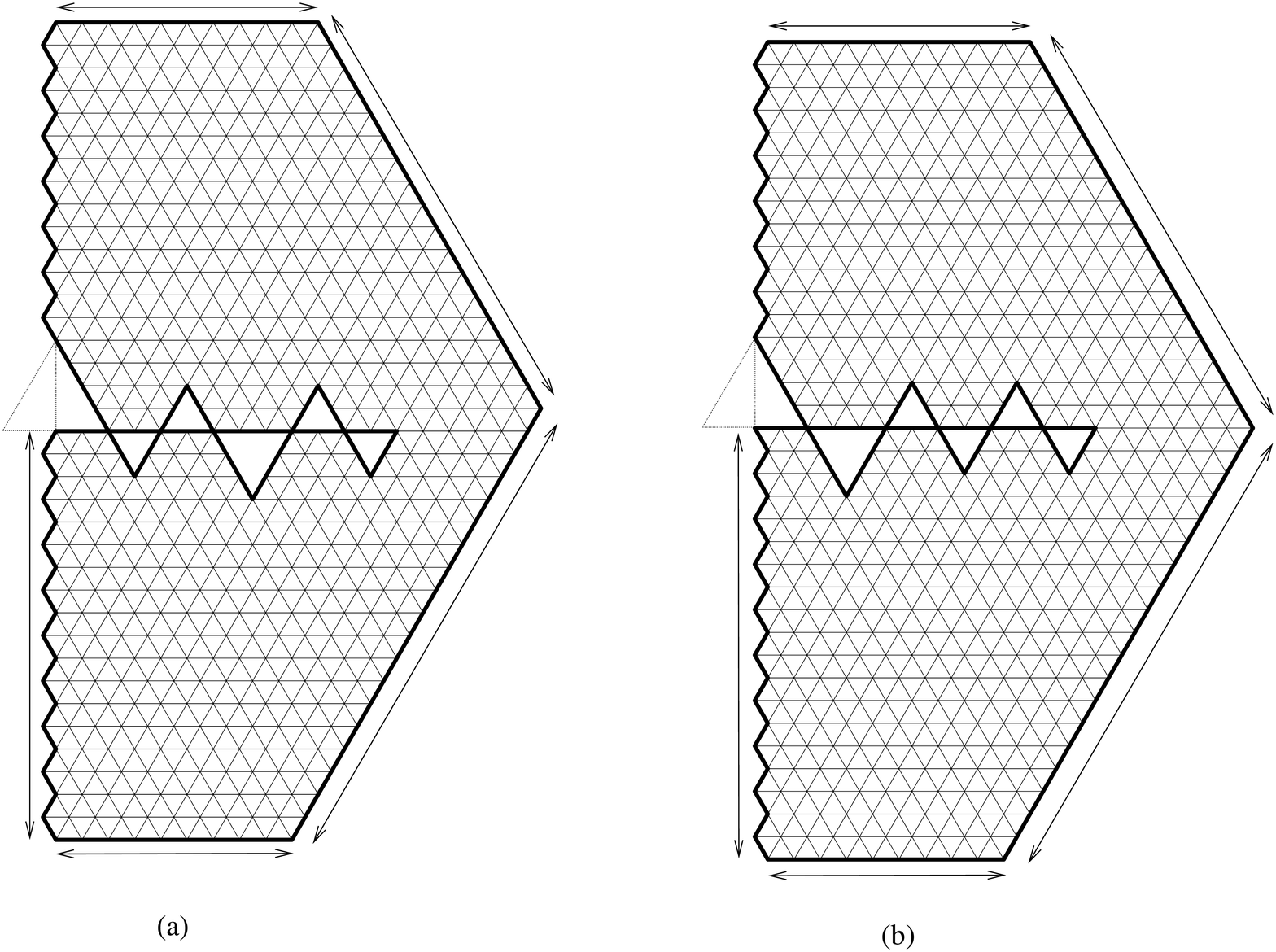}%
\end{picture}%
%
%

\begin{picture}(19122,14547)(2375,-13913)
\put(4363,-6421){\makebox(0,0)[lb]{\smash{{\SetFigFont{20}{24.0}{\rmdefault}{\mddefault}{\updefault}{$a_2$}%
}}}}
\put(5113,-6046){\makebox(0,0)[lb]{\smash{{\SetFigFont{20}{24.0}{\rmdefault}{\mddefault}{\updefault}{$a_3$}%
}}}}
\put(6013,-6519){\makebox(0,0)[lb]{\smash{{\SetFigFont{20}{24.0}{\rmdefault}{\mddefault}{\updefault}{$a_4$}%
}}}}
\put(7078,-6046){\makebox(0,0)[lb]{\smash{{\SetFigFont{20}{24.0}{\rmdefault}{\mddefault}{\updefault}{$a_5$}%
}}}}
\put(7858,-6421){\makebox(0,0)[lb]{\smash{{\SetFigFont{20}{24.0}{\rmdefault}{\mddefault}{\updefault}{$a_6$}%
}}}}
\put(3999,269){\makebox(0,0)[lb]{\smash{{\SetFigFont{20}{24.0}{\rmdefault}{\mddefault}{\updefault}{$x+a_2+a_4+a_6$}%
}}}}
\put(8484,-1336){\rotatebox{300.0}{\makebox(0,0)[lb]{\smash{{\SetFigFont{20}{24.0}{\rmdefault}{\mddefault}{\updefault}{$y+z+2a_1+2a_3+2a_5$}%
}}}}}
\put(8431,-10914){\rotatebox{60.0}{\makebox(0,0)[lb]{\smash{{\SetFigFont{20}{24.0}{\rmdefault}{\mddefault}{\updefault}{$y+z+2a_2+2a_4+2a_6$}%
}}}}}
\put(3819,-12856){\makebox(0,0)[lb]{\smash{{\SetFigFont{20}{24.0}{\rmdefault}{\mddefault}{\updefault}{$x+a_1+a_3+a_5$}%
}}}}
\put(2694,-5709){\rotatebox{60.0}{\makebox(0,0)[lb]{\smash{{\SetFigFont{20}{24.0}{\rmdefault}{\mddefault}{\updefault}{$2a_1$}%
}}}}}
\put(2821,-10891){\rotatebox{90.0}{\makebox(0,0)[lb]{\smash{{\SetFigFont{20}{24.0}{\rmdefault}{\mddefault}{\updefault}{$2z+2a_2+2a_4+2a_6$}%
}}}}}
\put(3406,-6061){\makebox(0,0)[lb]{\smash{{\SetFigFont{20}{24.0}{\rmdefault}{\mddefault}{\updefault}{$a_1$}%
}}}}
\put(13076,-5679){\rotatebox{60.0}{\makebox(0,0)[lb]{\smash{{\SetFigFont{20}{24.0}{\rmdefault}{\mddefault}{${0,0,0}2a_1$}%
}}}}}
\put(13846,-5979){\makebox(0,0)[lb]{\smash{{\SetFigFont{20}{24.0}{\rmdefault}{\mddefault}{\updefault}{$a_1$}%
}}}}
\put(14896,-6444){\makebox(0,0)[lb]{\smash{{\SetFigFont{20}{24.0}{\rmdefault}{\mddefault}{\updefault}{$a_2$}%
}}}}
\put(15886,-5994){\makebox(0,0)[lb]{\smash{{\SetFigFont{20}{24.0}{\rmdefault}{\mddefault}{\updefault}{$a_3$}%
}}}}
\put(16681,-6384){\makebox(0,0)[lb]{\smash{{\SetFigFont{20}{24.0}{\rmdefault}{\mddefault}{\updefault}{$a_4$}%
}}}}
\put(17476,-5994){\makebox(0,0)[lb]{\smash{{\SetFigFont{20}{24.0}{\rmdefault}{\mddefault}{\updefault}{$a_5$}%
}}}}
\put(18241,-6354){\makebox(0,0)[lb]{\smash{{\SetFigFont{20}{24.0}{\rmdefault}{\mddefault}{\updefault}{$a_6$}%
}}}}
\put(14409,-13179){\makebox(0,0)[lb]{\smash{{\SetFigFont{20}{24.0}{\rmdefault}{\mddefault}{\updefault}{$x+a_1+a_3+a_5$}%
}}}}
\put(14588, -9){\makebox(0,0)[lb]{\smash{{\SetFigFont{20}{24.0}{\rmdefault}{\mddefault}{\updefault}{$x+a_2+a_4+a_6$}%
}}}}
\put(18841,-1066){\rotatebox{300.0}{\makebox(0,0)[lb]{\smash{{\SetFigFont{20}{24.0}{\rmdefault}{\mddefault}{\updefault}{$y+z-1+2a_1+2a_3+2a_5$}%
}}}}}
\put(19006,-11296){\rotatebox{60.0}{\makebox(0,0)[lb]{\smash{{\SetFigFont{20}{24.0}{\rmdefault}{\mddefault}{\updefault}{$y+z-1+2a_2+2a_4+2a_6$}%
}}}}}
\put(13276,-11101){\rotatebox{90.0}{\makebox(0,0)[lb]{\smash{{\SetFigFont{20}{24.0}{\rmdefault}{\mddefault}{\updefault}{$2z-1+2a_2+2a_4+2a_6$}%
}}}}}
\end{picture}}
\caption{Two halved hexagons with a fern removed on their west sides.  The regions were enumerated in \cite{Halfhex2}.}\label{middlehole1}
\end{figure}

Recently, Rohatgi \cite{Ranjan} generalized the tiling enumeration of a halved hexagon
 to a halved hexagon with a triangle removed along the northeast side (see Figures \ref{halfhex6}(c) and (d)).
 The author \cite{Halfhex1} generalized further Rohatgi's result by extending the single triangular hole in the latter regions to a fern of
  an arbitrary number of triangular holes (see Figure \ref{halfhex13}). Moreover, it has been shown in \cite{Halfhex2} that if
  the fern is removed from the west side instead of the northeast side, we also have a simple product formula for the tiling number (see Figure \ref{middlehole1} for examples).

In this paper, we show that the instantaneous removal of two aligned ferns from two different
 sides of the halved hexagon still gives a simple product formula for the number of tilings.
 Based on the orientations of triangles in the ferns and the weight assignments of lozenges along the west side of the halved hexagons,
  there are \emph{sixteen} families of
  regions to enumerate. The extensive list of tiling enumerations of the halved hexagons will be presented in Theorems \ref{main1}--\ref{mainMR4} below.

Let us define the \emph{hyperfactorial function} by:
\begin{equation}
\Hf(n):=0!\cdot 1! \cdot 2! \cdots (n-1)!,
\end{equation}
 and its `skipping' version by
 \begin{equation}
\Hf_2(n):=\prod_{i=1}^{\lfloor n/2\rfloor} (n-2i)!.
\end{equation}
We define the \emph{Pochhammer symbol} $(x)_n$ by
\begin{equation}
(x)_n:=
\begin{cases}
x(x+1)(x+2)\cdots (x+n-1) & \text{if $n>0$;}\\
\quad\quad\quad\quad\quad\;1 &\text{if $n=0$;}\\
\dfrac{1}{(x-1)(x-2)\cdots(x+n)} &\text{if $n<0$.}
\end{cases}
\end{equation}
We also use its `skipping' variation:
\begin{equation}
[x]_n:=
\begin{cases}
x(x+2)(x+4)\cdots(x+2n-2) & \text{if $n>0$;}\\
\quad\quad\quad\quad\quad\;1 &\text{if $n=0$;}\\
\dfrac{1}{(x-2)(x-4)\cdots(x+2n)} &\text{if $n<0$.}
\end{cases}
\end{equation}
We finally define the two products for nonnegative integers $m,n$
\begin{equation}\label{defineT}
\T(x,n,m):=\prod_{i=0}^{m-1}(x+i)_{n-2i}
\end{equation}
and
\begin{equation}\label{defineV}
\V(x,n,m):=\prod_{i=0}^{m-1}[x+2i]_{n-2i},
\end{equation}
where the empty products are taken to be $1$.

Next, we quote the enumerations of four families of \emph{quartered hexagons} by the author in  \cite{Lai, Lai3}.
These enumerations were side results when we generalized  the work of W. Jockusch and J. Propp's on \emph{quartered Aztec diamonds} \cite{JP}.
 Intuitively, a quartered hexagon is  obtained from a halved hexagon by cutting off several up-pointing triangles along the base
  (see the regions in Figure \ref{fig:halfhex3c}).
It  is worth noticing that the unweighted enumeration of the quartered hexagons in \cite{Lai} is equivalent to the lattice-path
enumeration of the so-called \emph{stars} by C. Krattenthaler, A. J. Guttmann, and X. G. Viennot \cite{KGV}.


\begin{figure}\centering
\setlength{\unitlength}{3947sp}%
\begingroup\makeatletter\ifx\SetFigFont\undefined%
\gdef\SetFigFont#1#2#3#4#5{%
  \reset@font\fontsize{#1}{#2pt}%
  \fontfamily{#3}\fontseries{#4}\fontshape{#5}%
  \selectfont}%
\fi\endgroup%
\resizebox{13cm}{!}{
\begin{picture}(0,0)
\includegraphics{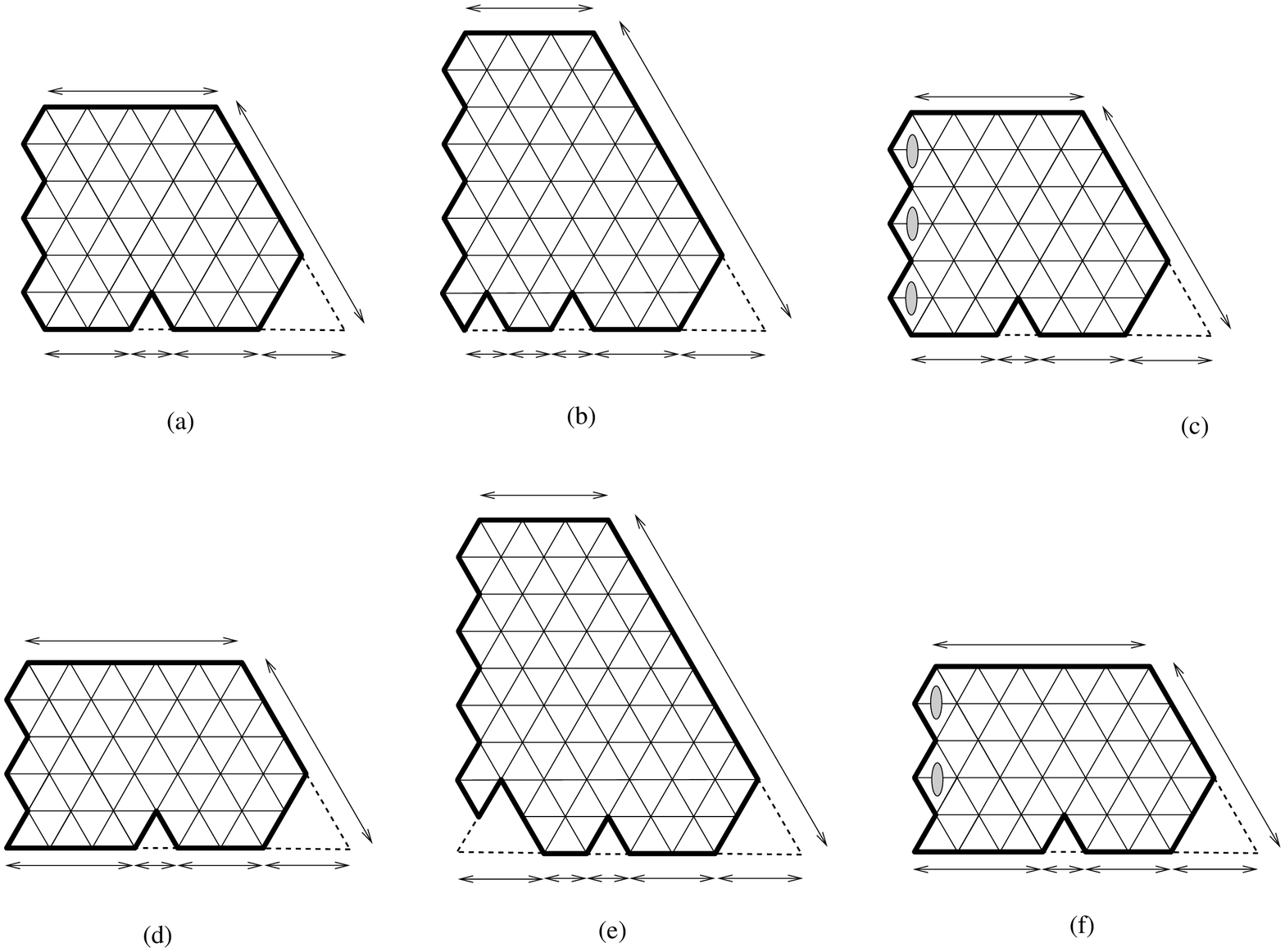}%
\end{picture}
\begin{picture}(11857,8988)(3679,-9536)
\put(4501,-1553){\makebox(0,0)[lb]{\smash{{\SetFigFont{16}{16.8}{\rmdefault}{\mddefault}{\updefault}{$t_1+t_3$}%
}}}}
\put(4148,-4338){\makebox(0,0)[lb]{\smash{{\SetFigFont{16}{16.8}{\rmdefault}{\mddefault}{\updefault}{$t_1$}%
}}}}
\put(4863,-4358){\makebox(0,0)[lb]{\smash{{\SetFigFont{16}{16.8}{\rmdefault}{\mddefault}{\updefault}{$t_2$}%
}}}}
\put(5513,-4368){\makebox(0,0)[lb]{\smash{{\SetFigFont{16}{16.8}{\rmdefault}{\mddefault}{\updefault}{$t_3$}%
}}}}
\put(8020,-4331){\makebox(0,0)[lb]{\smash{{\SetFigFont{16}{16.8}{\rmdefault}{\mddefault}{\updefault}{$t_2$}%
}}}}
\put(8387,-4331){\makebox(0,0)[lb]{\smash{{\SetFigFont{16}{16.8}{\rmdefault}{\mddefault}{\updefault}{$t_3$}%
}}}}
\put(4058,-9001){\makebox(0,0)[lb]{\smash{{\SetFigFont{16}{16.8}{\rmdefault}{\mddefault}{\updefault}{$t_1$}%
}}}}
\put(4938,-9011){\makebox(0,0)[lb]{\smash{{\SetFigFont{16}{16.8}{\rmdefault}{\mddefault}{\updefault}{$t_2$}%
}}}}
\put(5548,-9026){\makebox(0,0)[lb]{\smash{{\SetFigFont{16}{16.8}{\rmdefault}{\mddefault}{\updefault}{$t_3$}%
}}}}
\put(4583,-6566){\makebox(0,0)[lb]{\smash{{\SetFigFont{16}{16.8}{\rmdefault}{\mddefault}{\updefault}{$t_1+t_3$}%
}}}}
\put(6287,-9022){\makebox(0,0)[lb]{\smash{{\SetFigFont{16}{16.8}{\rmdefault}{\mddefault}{\updefault}{$t_4$}%
}}}}
\put(8785,-4331){\makebox(0,0)[lb]{\smash{{\SetFigFont{16}{16.8}{\rmdefault}{\mddefault}{\updefault}{$t_4$}%
}}}}
\put(9295,-4331){\makebox(0,0)[lb]{\smash{{\SetFigFont{16}{16.8}{\rmdefault}{\mddefault}{\updefault}{$t_5$}%
}}}}
\put(10045,-4331){\makebox(0,0)[lb]{\smash{{\SetFigFont{16}{16.8}{\rmdefault}{\mddefault}{\updefault}{$t_6$}%
}}}}
\put(8089,-9145){\makebox(0,0)[lb]{\smash{{\SetFigFont{16}{16.8}{\rmdefault}{\mddefault}{\updefault}{$t_2$}%
}}}}
\put(8681,-9130){\makebox(0,0)[lb]{\smash{{\SetFigFont{16}{16.8}{\rmdefault}{\mddefault}{\updefault}{$t_3$}%
}}}}
\put(9094,-9137){\makebox(0,0)[lb]{\smash{{\SetFigFont{16}{16.8}{\rmdefault}{\mddefault}{\updefault}{$t_4$}%
}}}}
\put(9611,-9130){\makebox(0,0)[lb]{\smash{{\SetFigFont{16}{16.8}{\rmdefault}{\mddefault}{\updefault}{$t_5$}%
}}}}
\put(10406,-9115){\makebox(0,0)[lb]{\smash{{\SetFigFont{16}{16.8}{\rmdefault}{\mddefault}{\updefault}{$t_6$}%
}}}}
\put(6256,-4355){\makebox(0,0)[lb]{\smash{{\SetFigFont{16}{16.8}{\rmdefault}{\mddefault}{\updefault}{$t_4$}%
}}}}
\put(6204,-2123){\rotatebox{300.0}{\makebox(0,0)[lb]{\smash{{\SetFigFont{16}{16.8}{\rmdefault}{\mddefault}{\updefault}{$2t_2+2t_4$}%
}}}}}
\put(6347,-7064){\rotatebox{300.0}{\makebox(0,0)[lb]{\smash{{\SetFigFont{16}{16.8}{\rmdefault}{\mddefault}{\updefault}{$2t_2+2t_4-1$}%
}}}}}
\put(8200,-798){\makebox(0,0)[lb]{\smash{{\SetFigFont{16}{16.8}{\rmdefault}{\mddefault}{\updefault}{$t_3+t_5$}%
}}}}
\put(8228,-5242){\makebox(0,0)[lb]{\smash{{\SetFigFont{16}{16.8}{\rmdefault}{\mddefault}{\updefault}{$t_3+t_5$}%
}}}}
\put(9685,-1481){\rotatebox{300.0}{\makebox(0,0)[lb]{\smash{{\SetFigFont{16}{16.8}{\rmdefault}{\mddefault}{\updefault}{$2t_2+2t_4+2t_6$}%
}}}}}
\put(9881,-6092){\rotatebox{300.0}{\makebox(0,0)[lb]{\smash{{\SetFigFont{16}{16.8}{\rmdefault}{\mddefault}{\updefault}{$2t_2+2t_4+2t_6-1$}%
}}}}}
\put(12392,-1604){\makebox(0,0)[lb]{\smash{{\SetFigFont{16}{16.8}{\rmdefault}{\mddefault}{\updefault}{$t_1+t_3$}%
}}}}
\put(12039,-4389){\makebox(0,0)[lb]{\smash{{\SetFigFont{16}{16.8}{\rmdefault}{\mddefault}{\updefault}{$t_1$}%
}}}}
\put(12754,-4409){\makebox(0,0)[lb]{\smash{{\SetFigFont{16}{16.8}{\rmdefault}{\mddefault}{\updefault}{$t_2$}%
}}}}
\put(13404,-4419){\makebox(0,0)[lb]{\smash{{\SetFigFont{16}{16.8}{\rmdefault}{\mddefault}{\updefault}{$t_3$}%
}}}}
\put(14147,-4406){\makebox(0,0)[lb]{\smash{{\SetFigFont{16}{16.8}{\rmdefault}{\mddefault}{\updefault}{$t_4$}%
}}}}
\put(14095,-2174){\rotatebox{300.0}{\makebox(0,0)[lb]{\smash{{\SetFigFont{16}{16.8}{\rmdefault}{\mddefault}{\updefault}{$2t_2+2t_4$}%
}}}}}
\put(12326,-9036){\makebox(0,0)[lb]{\smash{{\SetFigFont{16}{16.8}{\rmdefault}{\mddefault}{\updefault}{$t_1$}%
}}}}
\put(13206,-9046){\makebox(0,0)[lb]{\smash{{\SetFigFont{16}{16.8}{\rmdefault}{\mddefault}{\updefault}{$t_2$}%
}}}}
\put(13816,-9061){\makebox(0,0)[lb]{\smash{{\SetFigFont{16}{16.8}{\rmdefault}{\mddefault}{\updefault}{$t_3$}%
}}}}
\put(12851,-6601){\makebox(0,0)[lb]{\smash{{\SetFigFont{16}{16.8}{\rmdefault}{\mddefault}{\updefault}{$t_1+t_3$}%
}}}}
\put(14555,-9057){\makebox(0,0)[lb]{\smash{{\SetFigFont{16}{16.8}{\rmdefault}{\mddefault}{\updefault}{$t_4$}%
}}}}
\put(14615,-7099){\rotatebox{300.0}{\makebox(0,0)[lb]{\smash{{\SetFigFont{16}{16.8}{\rmdefault}{\mddefault}{\updefault}{$2t_2+2t_4-1$}%
}}}}}
\end{picture}}
\caption{ (a) The region $\mathcal{Q}(2,1,2,2)$. (b) The region $\mathcal{Q}(0,1,1,1,2,2)$.
  (c) The region $\mathcal{Q}'(2,1,2,2)$. (d) The region $\mathcal{K}(3,1,2,2)$. (e) The region $\mathcal{K}(0,2,1,1,2,2)$.
  (f) The region $\mathcal{K}'(3,1,2,2)$. The lozenges with shaded cores are weighted by $\frac{1}{2}$. The figure first appeared
  in \cite{Halfhex2}.}\label{fig:halfhex3c}
\end{figure}

Assume that $\textbf{t}=(t_1,t_2,\dots,t_{2l})$ is a sequence of non-negative integers. We define the first quartered hexagon as follows.
 Consider a trapezoidal region
whose north, northeast, and south sides have  respectively lengths
$\od_t,$ $ 2\e_t,$ and $\e_t+\od_t$, and whose west side runs along a vertical zigzag lattice path with $\e_t$ steps.
Here, and from now on, we are using the notations $\e_t$ and $\od_t$ for the sum of even terms and the sum of odd terms in the sequence $\textbf{t}$, respectively.
 We remove the triangles of side-lengths $t_{2i}$'s from the base of the latter region,  such that the distances between
 two consecutive triangles are $t_{2i-1}$'s. Denote the resulting region by
  $\mathcal{Q}(\textbf{t})=\mathcal{Q}(t_1,t_2,\dotsc,t_{2l})$ (see the regions in Figure \ref{fig:halfhex3c}(a)
   for the case when $t_{1}>0$ and Figure \ref{fig:halfhex3c}(b) for the  case when $t_{1}=0$).
   We also consider the weighted counterpart  $\mathcal{Q}'(\textbf{t})$ of the latter region,
    where the vertical lozenges on the west side are weighted by $\frac{1}{2}$
     (see Figure \ref{fig:halfhex3c}(c); the vertical lozenges with shaded cores are weighted by $\frac{1}{2}$).

We are also interested in a variation of the $\mathcal{Q}$-type regions as follows.
 Consider the  trapezoidal region whose north, northeast, and south sides have lengths $\od_t, 2\e_t-1, \e_t+\od_t,$
 respectively, and whose west side follows the vertical zigzag lattice path with $\e_t-\frac{1}{2}$ steps (i.e. the west side has
 $\e_t-1$ and a half  `bumps'). Next, we  also remove the triangles of side-lengths $t_{2i}$'s from the base, such that the distances
  between two consecutive ones are $t_{2i-1}$'s. Denote by $\mathcal{K}(\textbf{t})=\mathcal{K}(t_1,t_2,\dotsc,t_{2l})$
  the resulting regions (see the regions in Figure \ref{fig:halfhex3c}(d)  for the case when $t_1>0$ and Figure \ref{fig:halfhex3c}(e)
  for the  case when $t_1=0$). Similar to the case of $\mathcal{Q}'$-type regions, we also define the weighted version
  $\mathcal{K}'(\textbf{t})$ of the $\mathcal{K}(\textbf{t})$  by assigning to each vertical lozenge on its west side a weight
  $\frac{1}{2}$ (see Figure \ref{fig:halfhex3c}(f)).

From now on, we use respectively the notations
$\Pn_{a,b,c}$, $\Pn'_{a,b,c}$, $\Q(\textbf{t})$, $\Q'(\textbf{t})$, $\K(\textbf{t})$, and  $\K'(\textbf{t})$
for the numbers of tilings of the regions $\mathcal{P}_{a,b,c}$, $\mathcal{P}'_{a,b,c}$, $\mathcal{Q}(\textbf{t})$,
$\mathcal{Q}'(\textbf{t})$, $\mathcal{K}(\textbf{t})$, and  $\mathcal{K}'(\textbf{t})$.

It is more convenient for us to use the following form, that was first introduced in \cite{Halfhex1}, of the enumerations of the four quartered
hexagons (instead of using the original form appeared
in \cite{Lai, Lai3}).

\begin{lem}\label{QAR}
For any sequence of non-negative integers $\textbf{t}=(t_1,t_2,\dotsc,t_{2l})$
\begin{align}\label{QARa}
\Q(\textbf{t})&=\dfrac{\prod_{i=1}^{l}\frac{(\s_{2i}(\textbf{t}))!}{(\s_{2i-1}(\textbf{t}))!}}{\Hf_2(2\e_t+1)} \prod_{i=1}^{l}\big(\Hf_2(2\s_{2i}(\textbf{t})+1)\Hf_2(2\s_{2i-1}(\textbf{t})+2)\big)\notag\\
                  &\times \displaystyle {\prod_{\substack{1\leq i< j\leq 2l\\
                  \text{$j-i$ odd}}}}\dfrac{\Hf(\s_j(\textbf{t})-\s_{i}(\textbf{t}))}{\Hf(\s_j(\textbf{t})+\s_{i}(\textbf{t})+1)}\displaystyle {\prod_{\substack{1\leq i<j\leq 2l\\
                  \text{$j-i$ even }}}}\dfrac{\Hf(\s_j(\textbf{t})+\s_{i}(\textbf{t})+1)}{\Hf(\s_j(\textbf{t})-\s_{i}(\textbf{t}))},
\end{align}
\begin{align}\label{QARb}
\Q'(\textbf{t})&=\dfrac{2^{-\e_t}}{\Hf_2(2\e_t+1)}  \prod_{i=1}^{l}\big(\Hf_2(2\s_{2i}(\textbf{t})+1)\Hf_2(2\s_{2i-1}(\textbf{t}))\big)\notag\\
                  & \times \displaystyle {\prod_{\substack{1\leq i< j\leq 2l\\
                  \text{$j-i$ odd}}}}\dfrac{\Hf(\s_j(\textbf{t})-\s_{i}(\textbf{t}))}{\Hf(\s_j(\textbf{t})+\s_{i}(\textbf{t}))}\displaystyle {\prod_{\substack{1\leq i<j\leq 2l\\
                  \text{$j-i$ even }}}}\dfrac{\Hf(\s_j(\textbf{t})+\s_{i}(\textbf{t}))}{\Hf(\s_j(\textbf{t})-\s_{i}(\textbf{t}))},
\end{align}
\begin{align}\label{QARc}
\K(\textbf{t})&=\dfrac{1}{\Hf_2(2\e_t)}   \prod_{i=1}^{l}\big(\Hf_2(2\s_{2i}(\textbf{t}))\Hf_2(2\s_{2i-1}(\textbf{t})+1)\big)\notag\\
                  &\times   \displaystyle {\prod_{\substack{1\leq i< j\leq 2l\\
                  \text{$j-i$ odd}}}}\dfrac{\Hf(\s_j(\textbf{t})-\s_{i}(\textbf{t}))}{\Hf(\s_j(\textbf{t})+\s_{i}(\textbf{t}))}\displaystyle {\prod_{\substack{1\leq i<j\leq 2l\\
                  \text{$j-i$ even }}}}\dfrac{\Hf(\s_j(\textbf{t})+\s_{i}(\textbf{t}))}{\Hf(\s_j(\textbf{t})-\s_{i}(\textbf{t}))},
\end{align}
and
\begin{align}\label{QARd}
\K'(\textbf{t})&=\dfrac{1}{\Hf_2(2\e_t)} \prod_{i=1}^{l}\big(\Hf_2(2\s_{2i}(\textbf{t})-1)\Hf_2(2\s_{2i-1}(\textbf{t}))\big)\notag\\
                  &\times   \displaystyle {\prod_{\substack{1\leq i< j\leq 2l\\
                  \text{$j-i$ odd}}}}\dfrac{\Hf(\s_j(\textbf{t})-\s_{i}(\textbf{t}))}{\Hf(\s_j(\textbf{t})+\s_{i}(\textbf{t})-1)}\displaystyle {\prod_{\substack{1\leq i<j\leq 2l\\
                  \text{$j-i$ even }}}}\dfrac{\Hf(\s_j(\textbf{t})+\s_{i}(\textbf{t})-1)}{\Hf(\s_j(\textbf{t})-\s_{i}(\textbf{t}))}.
\end{align}
where $\s_k(\textbf{t})=t_1+t_2+\dots+t_k$ denotes the $k$-th partial sum of the sequence $\textbf{t}$.
\end{lem}


\begin{figure}\centering
%
%
\setlength{\unitlength}{3947sp}%
\begingroup\makeatletter\ifx\SetFigFont\undefined%
\gdef\SetFigFont#1#2#3#4#5{%
  \reset@font\fontsize{#1}{#2pt}%
  \fontfamily{#3}\fontseries{#4}\fontshape{#5}%
  \selectfont}%
\fi\endgroup%
\resizebox{13cm}{!}{
\begin{picture}(0,0)%
\includegraphics{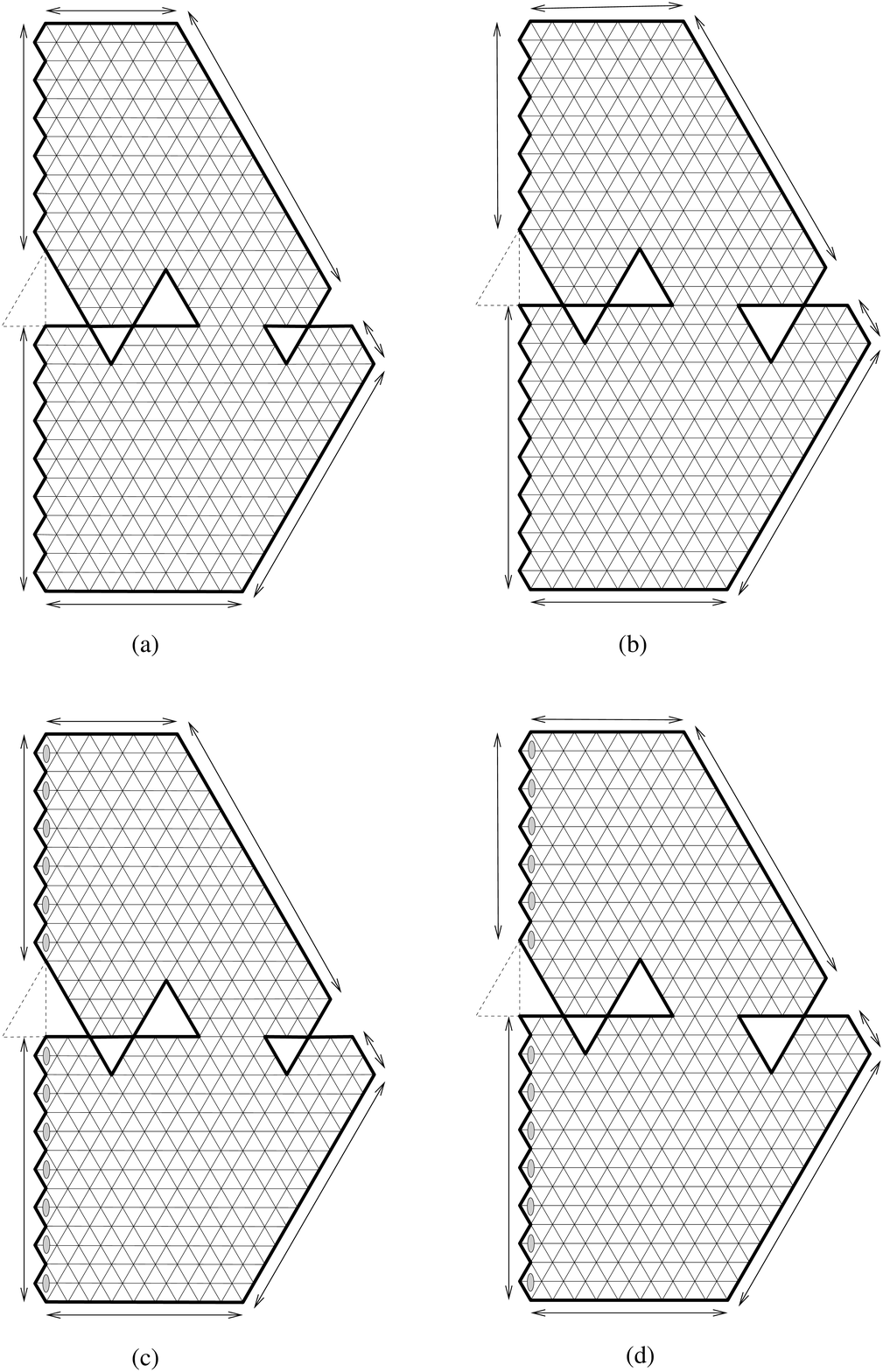}%
\end{picture}%

\begin{picture}(17150,26064)(382,-25475)
\put(14477,-14407){\rotatebox{300.0}{\makebox(0,0)[lb]{\smash{{\SetFigFont{20}{24.0}{\rmdefault}{\mddefault}{\itdefault}{$2y+2a_1+2a_3+b_1-1$}%
}}}}}
\put(17043,-19024){\makebox(0,0)[lb]{\smash{{\SetFigFont{20}{24.0}{\rmdefault}{\mddefault}{\itdefault}{$z$}%
}}}}
\put(15259,-23365){\rotatebox{60.0}{\makebox(0,0)[lb]{\smash{{\SetFigFont{20}{24.0}{\rmdefault}{\mddefault}{\itdefault}{$2y+z+2a_2+2b_2-1$}%
}}}}}
\put(14864,-19324){\makebox(0,0)[lb]{\smash{{\SetFigFont{20}{24.0}{\rmdefault}{\mddefault}{\itdefault}{$b_2$}%
}}}}
\put(15986,-18695){\makebox(0,0)[lb]{\smash{{\SetFigFont{20}{24.0}{\rmdefault}{\mddefault}{\itdefault}{$b_1$}%
}}}}
\put(9832,-22970){\rotatebox{90.0}{\makebox(0,0)[lb]{\smash{{\SetFigFont{20}{24.0}{\rmdefault}{\mddefault}{\itdefault}{$y+z+a_2+b_2-\frac{1}{2}$}%
}}}}}
\put(12371,-18552){\makebox(0,0)[lb]{\smash{{\SetFigFont{20}{24.0}{\rmdefault}{\mddefault}{\itdefault}{$a_3$}%
}}}}
\put(11358,-19159){\makebox(0,0)[lb]{\smash{{\SetFigFont{20}{24.0}{\rmdefault}{\mddefault}{\itdefault}{$a_2$}%
}}}}
\put(9686,-18292){\rotatebox{60.0}{\makebox(0,0)[lb]{\smash{{\SetFigFont{20}{24.0}{\rmdefault}{\mddefault}{\itdefault}{$2a_1$}%
}}}}}
\put(9693,-16403){\rotatebox{90.0}{\makebox(0,0)[lb]{\smash{{\SetFigFont{20}{24.0}{\rmdefault}{\mddefault}{\itdefault}{$y+a_3+b_1-\frac{1}{2}$}%
}}}}}
\put(11325,-13122){\makebox(0,0)[lb]{\smash{{\SetFigFont{20}{24.0}{\rmdefault}{\mddefault}{\itdefault}{$x+a_2+b_2$}%
}}}}
\put(729,-16278){\rotatebox{90.0}{\makebox(0,0)[lb]{\smash{{\SetFigFont{20}{24.0}{\rmdefault}{\mddefault}{\itdefault}{$y+a_3+b_1$}%
}}}}}
\put(797,-22826){\rotatebox{90.0}{\makebox(0,0)[lb]{\smash{{\SetFigFont{20}{24.0}{\rmdefault}{\mddefault}{\itdefault}{$y+z+a_2+b_2$}%
}}}}}
\put(2402,-24806){\makebox(0,0)[lb]{\smash{{\SetFigFont{20}{24.0}{\rmdefault}{\mddefault}{\itdefault}{$x + a_1+a_3+b_1$}%
}}}}
\put(6107,-23636){\rotatebox{60.0}{\makebox(0,0)[lb]{\smash{{\SetFigFont{20}{24.0}{\rmdefault}{\mddefault}{\itdefault}{$2y+z+2a_2+2b_2$}%
}}}}}
\put(5027,-14366){\rotatebox{300.0}{\makebox(0,0)[lb]{\smash{{\SetFigFont{20}{24.0}{\rmdefault}{\mddefault}{\itdefault}{$2y+2a_1+2a_3+b_1$}%
}}}}}
\put(1982,-13196){\makebox(0,0)[lb]{\smash{{\SetFigFont{20}{24.0}{\rmdefault}{\mddefault}{\itdefault}{$x+a_2+b_2$}%
}}}}
\put(7787,-19301){\makebox(0,0)[lb]{\smash{{\SetFigFont{20}{24.0}{\rmdefault}{\mddefault}{\itdefault}{$z$}%
}}}}
\put(5784,-19631){\makebox(0,0)[lb]{\smash{{\SetFigFont{20}{24.0}{\rmdefault}{\mddefault}{\itdefault}{$b_2$}%
}}}}
\put(6669,-19053){\makebox(0,0)[lb]{\smash{{\SetFigFont{20}{24.0}{\rmdefault}{\mddefault}{\itdefault}{$b_1$}%
}}}}
\put(3487,-18949){\makebox(0,0)[lb]{\smash{{\SetFigFont{20}{24.0}{\rmdefault}{\mddefault}{\itdefault}{$a_3$}%
}}}}
\put(2464,-19540){\makebox(0,0)[lb]{\smash{{\SetFigFont{20}{24.0}{\rmdefault}{\mddefault}{\itdefault}{$a_2$}%
}}}}
\put(821,-18831){\rotatebox{60.0}{\makebox(0,0)[lb]{\smash{{\SetFigFont{20}{24.0}{\rmdefault}{\mddefault}{\itdefault}{$2a_1$}%
}}}}}
\put(11276,-11470){\makebox(0,0)[lb]{\smash{{\SetFigFont{20}{24.0}{\rmdefault}{\mddefault}{\itdefault}{$x+a_1+a_3+b_1$}%
}}}}
\put(14469,-1115){\rotatebox{300.0}{\makebox(0,0)[lb]{\smash{{\SetFigFont{20}{24.0}{\rmdefault}{\mddefault}{\itdefault}{$2y+2a_1+2a_3+b_1-1$}%
}}}}}
\put(17035,-5732){\makebox(0,0)[lb]{\smash{{\SetFigFont{20}{24.0}{\rmdefault}{\mddefault}{\itdefault}{$z$}%
}}}}
\put(813,-5539){\rotatebox{60.0}{\makebox(0,0)[lb]{\smash{{\SetFigFont{20}{24.0}{\rmdefault}{\mddefault}{\itdefault}{$2a_1$}%
}}}}}
\put(2456,-6248){\makebox(0,0)[lb]{\smash{{\SetFigFont{20}{24.0}{\rmdefault}{\mddefault}{\itdefault}{$a_2$}%
}}}}
\put(3479,-5657){\makebox(0,0)[lb]{\smash{{\SetFigFont{20}{24.0}{\rmdefault}{\mddefault}{\itdefault}{$a_3$}%
}}}}
\put(6661,-5761){\makebox(0,0)[lb]{\smash{{\SetFigFont{20}{24.0}{\rmdefault}{\mddefault}{\itdefault}{$b_1$}%
}}}}
\put(5776,-6339){\makebox(0,0)[lb]{\smash{{\SetFigFont{20}{24.0}{\rmdefault}{\mddefault}{\itdefault}{$b_2$}%
}}}}
\put(7779,-6009){\makebox(0,0)[lb]{\smash{{\SetFigFont{20}{24.0}{\rmdefault}{\mddefault}{\itdefault}{$z$}%
}}}}
\put(1974, 96){\makebox(0,0)[lb]{\smash{{\SetFigFont{20}{24.0}{\rmdefault}{\mddefault}{\itdefault}{$x+a_2+b_2$}%
}}}}
\put(5019,-1074){\rotatebox{300.0}{\makebox(0,0)[lb]{\smash{{\SetFigFont{20}{24.0}{\rmdefault}{\mddefault}{\itdefault}{$2y+2a_1+2a_3+b_1$}%
}}}}}
\put(6099,-10344){\rotatebox{60.0}{\makebox(0,0)[lb]{\smash{{\SetFigFont{20}{24.0}{\rmdefault}{\mddefault}{\itdefault}{$2y+z+2a_2+2b_2$}%
}}}}}
\put(2394,-11514){\makebox(0,0)[lb]{\smash{{\SetFigFont{20}{24.0}{\rmdefault}{\mddefault}{\itdefault}{$x + a_1+a_3+b_1$}%
}}}}
\put(789,-9534){\rotatebox{90.0}{\makebox(0,0)[lb]{\smash{{\SetFigFont{20}{24.0}{\rmdefault}{\mddefault}{\itdefault}{$y+z+a_2+b_2$}%
}}}}}
\put(721,-2986){\rotatebox{90.0}{\makebox(0,0)[lb]{\smash{{\SetFigFont{20}{24.0}{\rmdefault}{\mddefault}{\itdefault}{$y+a_3+b_1$}%
}}}}}
\put(11317,170){\makebox(0,0)[lb]{\smash{{\SetFigFont{20}{24.0}{\rmdefault}{\mddefault}{\itdefault}{$x+a_2+b_2$}%
}}}}
\put(9685,-3111){\rotatebox{90.0}{\makebox(0,0)[lb]{\smash{{\SetFigFont{20}{24.0}{\rmdefault}{\mddefault}{\itdefault}{$y+a_3+b_1-\frac{1}{2}$}%
}}}}}
\put(9678,-5000){\rotatebox{60.0}{\makebox(0,0)[lb]{\smash{{\SetFigFont{20}{24.0}{\rmdefault}{\mddefault}{\itdefault}{$2a_1$}%
}}}}}
\put(11350,-5867){\makebox(0,0)[lb]{\smash{{\SetFigFont{20}{24.0}{\rmdefault}{\mddefault}{\itdefault}{$a_2$}%
}}}}
\put(12363,-5260){\makebox(0,0)[lb]{\smash{{\SetFigFont{20}{24.0}{\rmdefault}{\mddefault}{\itdefault}{$a_3$}%
}}}}
\put(9824,-9678){\rotatebox{90.0}{\makebox(0,0)[lb]{\smash{{\SetFigFont{20}{24.0}{\rmdefault}{\mddefault}{\itdefault}{$y+z+a_2+b_2-\frac{1}{2}$}%
}}}}}
\put(15978,-5403){\makebox(0,0)[lb]{\smash{{\SetFigFont{20}{24.0}{\rmdefault}{\mddefault}{\itdefault}{$b_1$}%
}}}}
\put(14856,-6032){\makebox(0,0)[lb]{\smash{{\SetFigFont{20}{24.0}{\rmdefault}{\mddefault}{\itdefault}{$b_2$}%
}}}}
\put(15251,-10073){\rotatebox{60.0}{\makebox(0,0)[lb]{\smash{{\SetFigFont{20}{24.0}{\rmdefault}{\mddefault}{\itdefault}{$2y+z+2a_2+2b_2-1$}%
}}}}}
\put(11284,-24762){\makebox(0,0)[lb]{\smash{{\SetFigFont{20}{24.0}{\rmdefault}{\mddefault}{\itdefault}{$x+a_1+a_3+b_1$}%
}}}}
\end{picture}}
\caption{(a) The region $H^{(1)}_{2,1,2}(2,2,3;\ 2,2)$. (b) The region $H^{(2)}_{2,1,2}(2,2,3;\ 2,3)$.
 (c) The weighted region $W^{(1)}_{2,1,2}(2,2,3;\ 2,2)$. (d) The weighted region $W^{(2)}_{2,1,2}(2,2,3;\ 2,3)$.}\label{fig:halvedhex1}
\end{figure}

We are now ready to define our first doubly-intruded halved hexagon.

Assume that $x,y,z$ are three non-negative integers and that $\textbf{a}=(a_1,a_2,\dotsc,a_m)$ and $\textbf{b}=(b_1,b_2,\dotsc,b_n)$ are two sequences of non-negative integers.

 We consider a halved hexagon whose north, northeast, southeast, and south sides have lengths $x+\e_a+\e_b$, $2y+z+2\od_a+
 2\od_b$, $2y+z+2\e_a+2\e_b$, $x+\od_a+\od_b$, respectively,
  and whose west side follows the vertical zigzag lattice path with $2y+z+a+b$ steps. Here, and from now on, we set
  \begin{align}
  a:=\sum_{i}a_i, \ \ \ \ \ \  b:=\sum_{j} b_j.
  \end{align}
Next, we remove two ferns at the level $z$ above
  the leftmost vertices of the halved hexagon as follows.
 The right fern starts from the northeast side with an up-pointing $b_1$-triangle and goes from right to left with the triangles of side-lengths
   $b_1,b_2,\dotsc,b_n$. The left fern starts
  with a half up-pointing triangle of side $2a_1$ on the west side, and goes from left to right with triangles of side-lengths $a_2,a_3,\dotsc, a_m$. Let $H^{(1)}_{x,y,z}(\textbf{a};\textbf{b})$ denote the
  resulting region (see Figure \ref{fig:halvedhex1}(a) for an example).
 The variation $H^{(2)}_{x,y,z}(\textbf{a}; \ \textbf{b})$ of the $H^{(1)}$-type region is obtained similarly from a halved hexagon of side-lengths\footnote{From now on, we always list the side-lengths of a halved hexagon in the clockwise order from the north side.}
   $x+\e_a+\e_b$, $2y+z+2\od_a+2\od_b-1$,
   $2y+z+2\e_a+2\e_b-1$, $x+\od_a+\od_b$, $2y+z+a+b-1$ as shown in Figure \ref{fig:halvedhex1}(b).

\begin{thm}\label{main1} Assume that $x,y,z$ are non-negative integers and that
$\textbf{a}=(a_1,a_2,\dotsc,a_m)$ and $\textbf{b}=(b_1,b_2,\dotsc,b_n)$ are two (possibly empty)
sequences of non-negative integers. Then
\begin{align}\label{main1eq}
\M(H^{(1)}_{x,y,z}(\textbf{a}; \textbf{b}))&=\frac{\M(H^{(1)}_{x+y,0,z}(\textbf{a}; \textbf{b}))
\M(H^{(1)}_{0,2y,z}(\textbf{a};\textbf{b}))}{\M(H^{(1)}_{y,0,z}(\textbf{a}; \textbf{b}))}\notag\\
&\quad\times \frac{\T(x+1,2a+b+2y+z,y)\V(2x+2a+3,b+2y+z-1,y)}{\T(1,2a+b+2y+z,y)\V(2a+3,b+2y+z-1,y)}\notag\\
&=2^{-y}\Q(0,a_1,\dotsc,a_{2\lfloor\frac{m+1}{2}\rfloor-1},a_{2\lfloor\frac{m+1}{2}\rfloor}+x+y+b_{2\lfloor\frac{n+1}{2}\rfloor},b_{2\lfloor\frac{n+1}{2}\rfloor-1},\dotsc,b_1)\notag\\
&\quad\times \Q(a_1,\dotsc,a_{\lceil \frac{m-1}{2}\rceil}, a_{\lceil \frac{m-1}{2}\rceil+1}+x+y+b_{\lceil \frac{n-1}{2}\rceil+1},b_{\lceil \frac{n-1}{2}\rceil},\dotsc, b_1,z) \notag\\
&\quad\times \frac{\Hf_2(2\od_a+2\od_b+1)\Hf_2(2\e_a+2\e_b+2z+1)}{\Hf_2(2\od_a+2\od_b+2y+1)\Hf_2(2\e_a+2\e_b+2y+2z+1)} \notag\\
&\quad\times \frac{\Hf(2a+b+2y+z+1)\Hf(b+y+z)}{\Hf(2a+b+y+z+1)\Hf(b+z)}\notag\\
&\quad\times  \frac{\T(x+1,2a+b+2y+z,y)\V(2x+2a+3,b+2y+z-1,y)}{\T(1,2a+b+2y+z,y)\V(2a+3,b+2y+z-1,y)},
\end{align}
where $a_i=0$ if $i>m$ and $b_j=0$ if $j>n$ by convention\footnote{In the rest of this paper, we always assume this convention.}.
\end{thm}

We note that the first $\mathcal{Q}$-type region in  (\ref{main1eq}), the region $\mathcal{Q}(0,a_1,\dotsc,a_{2\lfloor\frac{m+1}{2}\rfloor-1},a_{2\lfloor\frac{m+1}{2}\rfloor}+x+y+b_{2\lfloor\frac{n+1}{2}\rfloor},b_{2\lfloor\frac{n+1}{2}\rfloor}-1,\dotsc,b_1)$,
is  (1) $\mathcal{Q}(0,a_1,\dotsc,a_{m-1},a_{m}+x+y+b_{n},b_{n-1},\dotsc,b_1)$ if $m$ and $n$ are even,  (2) $\mathcal{Q}(0,a_1,\dotsc,a_{m},x+y+b_{n},b_{n-1},\dotsc,b_1)$ if $m$ is odd and $n$ is even,
(3) $\mathcal{Q}(0,a_1,\dotsc,a_{m-1},a_{m}+x+y,b_{n},\dotsc,b_1)$ if $m$ is even and $n$ is odd,  and
(4) $\mathcal{Q}(0,a_1,\dotsc,a_{m},x+y,b_{n},\dotsc,b_1)$ if $m$ and $n$ are odd. The explicit form of the second $\mathcal{Q}$-type
region can be obtained similarly.  Moreover, the two $\mathcal{Q}$-type regions in   (\ref{main1eq}) are determined by the triangles in the two ferns.

 We also note that the product of
the numbers of tilings of the above two $\mathcal{Q}$-type regions
 is exactly $\M(H^{(1)}_{x+y,0,z}(\textbf{a}; \textbf{b}))$, and this fact can be proved by using  Lemma \ref{RS}
in the next section.

\begin{thm}\label{main2} For non-negative integers $x,y,z$ and sequences of non-negative integers $\textbf{a}=(a_1,a_2,\dotsc,a_m)$ and $\textbf{b}=(b_1,b_2,\dotsc,b_n)$
\begin{align}\label{main2eq}
\M(H^{(2)}_{x,y,z}(\textbf{a};\textbf{b}))&=\frac{\M(H^{(2)}_{x+y,0,z}(\textbf{a};\textbf{b}))
\M(H^{(2)}_{0,2y,z}(\textbf{a};\textbf{b}))}{\M(H^{(2)}_{y,0,z}(\textbf{a};\textbf{b}))}\notag\\
&\quad\times \frac{\T(x+1,2a+b+2y+z-1,y)\V(2x+2a+3,b+2y+z-2,y)}{\T(1,2a+b+2y+z-1,y)\V(2a+3,b+2y+z-2,y)}\notag\\
&=\K(0,a_1,\dotsc,a_{2\lfloor\frac{m+1}{2}\rfloor-1},a_{2\lfloor\frac{m+1}{2}\rfloor}+x+y+b_{2\lfloor\frac{n+1}{2}\rfloor},b_{2\lfloor\frac{n+1}{2}\rfloor-1},\dotsc,b_1)\notag\\
&\quad\times \K(a_1,\dotsc,a_{\lceil \frac{m-1}{2}\rceil}, a_{\lceil \frac{m-1}{2}\rceil+1}+x+y+b_{\lceil \frac{n-1}{2}\rceil+1},b_{\lceil \frac{n-1}{2}\rceil},\dotsc, b_1,z) \notag\\
&\quad\times\frac{(2a-1)!!}{(2a+2y-1)!!} \frac{\Hf_2(2\od_a+2\od_b)\Hf_2(2\e_a+2\e_b+2z)}{\Hf_2(2\od_a+2\od_b+2y)
\Hf_2(2\e_a+2\e_b+2y+2z)} \notag\\
&\quad\times \frac{\Hf(2a+b+2y+z)\Hf(b+y+z)}{\Hf(2a+b+y+z)\Hf(b+z)}\notag\\
&\quad\times  \frac{\T(x+1,2a+b+2y+z-1,y)\V(2x+2a+3,b+2y+z-2,y)}{\T(1,2a+b+2y+z-1,y)\V(2a+3,b+2y+z-2,y)},
\end{align}
where the `double' factorial is defined as $(2n+1)!!=1\cdot 3\cdot 5\cdots(2n+1)$ and $(2n)!!=2\cdot4\cdots2n$.
\end{thm}

Similar to the case of the the region $\mathcal{P}_{a,b,c}$, we would like to enumerate tilings of the weighted version $W^{(1)}_{x,y,z}(\textbf{a}; \textbf{b})$ and $W^{(2)}_{x,y,z}(\textbf{a}; \textbf{b})$
 of the above $H^{(1)}$- and $H^{(2)}$-type regions that are obtained by assigning to each vertical lozenge along their west sides a weight $1/2$ (see the lozenges with shaded cores in Figures \ref{fig:halvedhex1}
 (c) and (d), respectively). The weighted numbers of tilings of these two newly defined regions are also given by closed-form products.

\begin{thm}\label{mainW1} For non-negative integers $x,y,z$ and sequences of non-negative integers $\textbf{a}=(a_1,a_2,\dotsc,a_m)$ and $\textbf{b}=(b_1,b_2,\dotsc,b_n)$, we have
\begin{align}\label{mainW1eq}
\M(W^{(1)}_{x,y,z}(\textbf{a}; \textbf{b}))&=\frac{\M(W^{(1)}_{x+y,0,z}(\textbf{a};\textbf{b}))
\M(W^{(1)}_{0,2y,z}(\textbf{a}; \textbf{b}))}{\M(W^{(1)}_{y,0,z}(\textbf{a}; \textbf{b}))}\notag\\
&\quad\times \frac{\T(x+1,2a+b+2y+z-1,y)\V(2x+2a+1,b+2y+z,y)}{\T(1,2a+b+2y+z-1,y)\V(2a+1,b+2y+z,y)}\notag\\
&=2^{-2y+a_1}\Q'(0,a_1,\dotsc,a_{2\lfloor\frac{m+1}{2}\rfloor-1},a_{2\lfloor\frac{m+1}{2}\rfloor}+x+y+b_{2\lfloor\frac{n+1}{2}\rfloor},b_{2\lfloor\frac{n+1}{2}\rfloor-1},\dotsc,b_1)\notag\\
&\quad\times \Q'(a_1,\dotsc,a_{2\lceil \frac{m-1}{2}\rceil}, a_{2\lceil \frac{m-1}{2}\rceil+1}+x+y+b_{2\lceil \frac{n-1}{2}\rceil+1},b_{2\lceil \frac{n-1}{2}\rceil},\dotsc, b_1,z) \notag\\
&\quad\times\frac{(2a+2y-1)!!}{(2a-1)!!}\frac{\Hf_2(2\od_a+2\od_b+1)\Hf_2(2\e_a+2\e_b+2z+1)}{\Hf_2(2\od_a+2\od_b+2y+1)\Hf_2(2\e_a+2\e_b+2y+2z+1)} \notag\\
&\quad\times \frac{\Hf(2a+b+2y+z)\Hf(b+y+z)}{\Hf(2a+b+y+z)\Hf(b+z)}\notag\\
&\quad\times  \frac{\T(x+1,2a+b+2y+z-1,y)\V(2x+2a+1,b+2y+z,y)}{\T(1,2a+b+2y+z-1,y)\V(2a+1,b+2y+z,y)}.
\end{align}
\end{thm}

\begin{thm}\label{mainW2} With the same notations in Theorem \ref{mainW1}, the weighted number of tilings of the $W^{(2)}$-type region is given by
\begin{align}\label{mainW2eq}
\M(W^{(2)}_{x,y,z}(\textbf{a}; \textbf{b}))&=\frac{\M(W^{(2)}_{x+y,0,z}(\textbf{a};\textbf{b}))
\M(W^{(2)}_{0,2y,z}(\textbf{a};\textbf{b}))}{\M(W^{(2)}_{y,0,z}(\textbf{a};\textbf{b}))}\notag\\
&\quad\times \frac{\T(x+1,2a+b+2y+z-2,y)\V(2x+2a+1,b+2y+z-1,y)}{\T(1,2a+b+2y+z-2,y)\V(2a+1,b+2y+z-1,2)}\notag\\
&=2^{-y+a_1-1}\K'(0,a_1,\dotsc,a_{2\lfloor\frac{m+1}{2}\rfloor-1},a_{2\lfloor\frac{m+1}{2}\rfloor}+x+y+b_{2\lfloor\frac{n+1}{2}\rfloor},b_{2\lfloor\frac{n+1}{2}\rfloor-1},\dotsc,b_1)\notag\\
&\quad\times \K'(a_1,\dotsc,a_{2\lceil \frac{m-1}{2}\rceil}, a_{2\lceil \frac{m-1}{2}\rceil+1}+x+y+b_{2\lceil \frac{n-1}{2}\rceil+1},b_{2\lceil \frac{n-1}{2}\rceil},\dotsc, b_1,z) \notag\\
&\quad\times\frac{\Hf_2(2\od_a+2\od_b)\Hf_2(2\e_a+2\e_b+2z)}{\Hf_2(2\od_a+2\od_b+2y)\Hf_2(2\e_a+2\e_b+2y+2z)} \notag\\
&\quad\times \frac{\Hf(2a+b+2y+z-1)\Hf(b+y+z)}{\Hf(2a+b+y+z-1)\Hf(b+z)}\notag\\
&\quad\times  \frac{\T(x+1,2a+b+2y+z-2,y)\V(2x+2a+1,b+2y+z-1,y)}{\T(1,2a+b+2y+z-2,y)\V(2a+1,b+2y+z-1,y)}.
\end{align}
\end{thm}

We note that, in the above four halved hexagons, the $a$-fern always starts with an up-pointing half triangle of side-length $2a_1$.
 We are also interested in the halved hexagons
in which this half triangle is \emph{down-pointing}.
The first halved hexagon of this type is defined as follows.

\begin{figure}\centering
\setlength{\unitlength}{3947sp}%
\begingroup\makeatletter\ifx\SetFigFont\undefined%
\gdef\SetFigFont#1#2#3#4#5{%
  \reset@font\fontsize{#1}{#2pt}%
  \fontfamily{#3}\fontseries{#4}\fontshape{#5}%
  \selectfont}%
\fi\endgroup%
\resizebox{13cm}{!}{
\begin{picture}(0,0)%
\includegraphics{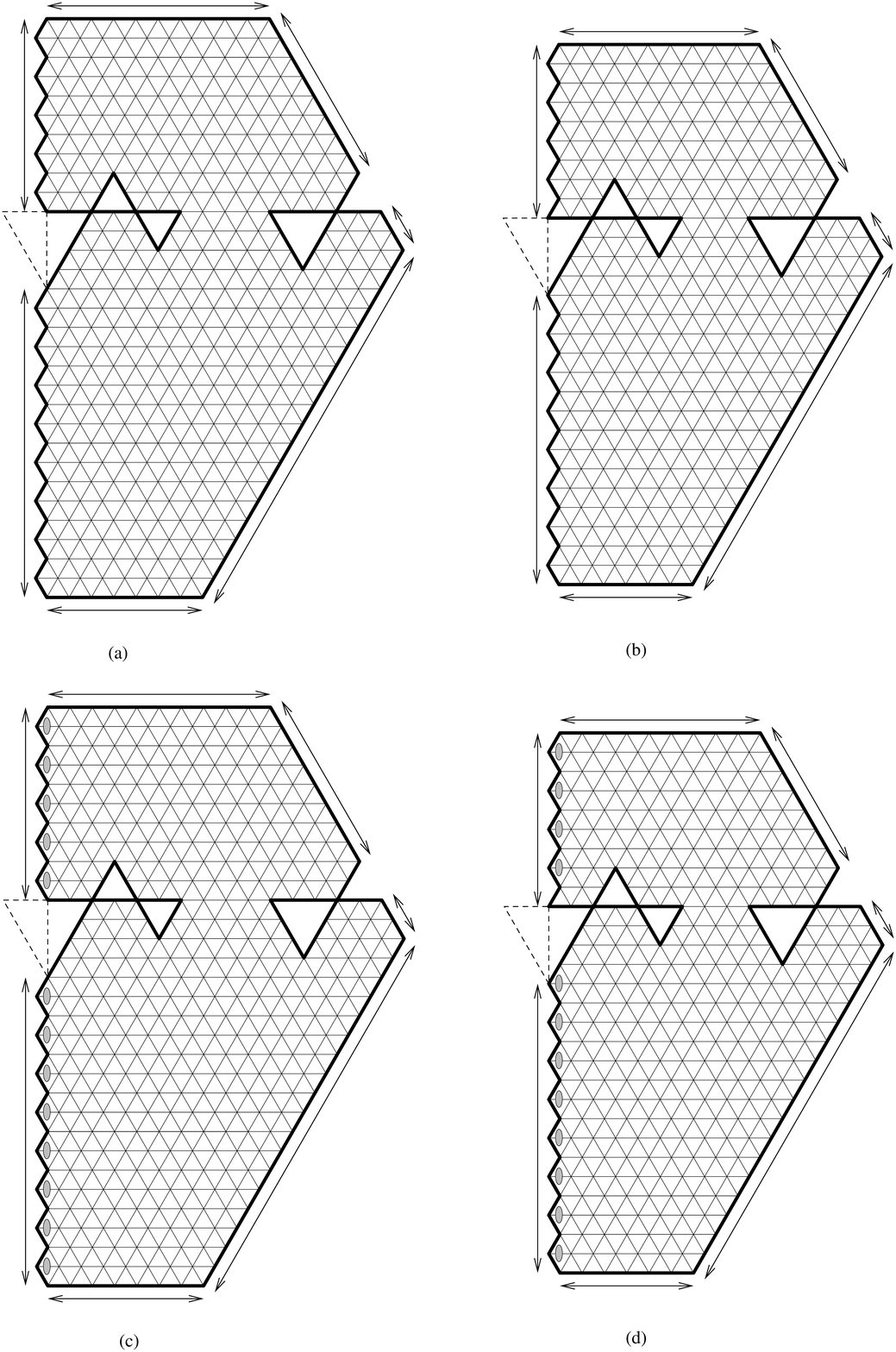}%
\end{picture}%
%
%

\begin{picture}(16644,25197)(343,-26026)
\put(16696,-5373){\makebox(0,0)[lb]{\smash{{\SetFigFont{14}{16.8}{\rmdefault}{\mddefault}{\itdefault}{$z$}%
}}}}
\put(2373,-17594){\makebox(0,0)[lb]{\smash{{\SetFigFont{14}{16.8}{\rmdefault}{\mddefault}{\itdefault}{$a_2$}%
}}}}
\put(3191,-18067){\makebox(0,0)[lb]{\smash{{\SetFigFont{14}{16.8}{\rmdefault}{\mddefault}{\itdefault}{$a_3$}%
}}}}
\put(5953,-18067){\makebox(0,0)[lb]{\smash{{\SetFigFont{14}{16.8}{\rmdefault}{\mddefault}{\itdefault}{$b_2$}%
}}}}
\put(6975,-17476){\makebox(0,0)[lb]{\smash{{\SetFigFont{14}{16.8}{\rmdefault}{\mddefault}{\itdefault}{$b_1$}%
}}}}
\put(2577,-13756){\makebox(0,0)[lb]{\smash{{\SetFigFont{14}{16.8}{\rmdefault}{\mddefault}{\itdefault}{$x+a_1+a_3+b_2$}%
}}}}
\put(6157,-14773){\rotatebox{300.0}{\makebox(0,0)[lb]{\smash{{\SetFigFont{14}{16.8}{\rmdefault}{\mddefault}{\itdefault}{$2y+2a_2+b_1$}%
}}}}}
\put(8101,-17890){\makebox(0,0)[lb]{\smash{{\SetFigFont{14}{16.8}{\rmdefault}{\mddefault}{\itdefault}{$z$}%
}}}}
\put(5543,-23264){\rotatebox{60.0}{\makebox(0,0)[lb]{\smash{{\SetFigFont{14}{16.8}{\rmdefault}{\mddefault}{\itdefault}{$2y+z+2a_1+2a_3+2b_2$}%
}}}}}
\put(2066,-25331){\makebox(0,0)[lb]{\smash{{\SetFigFont{14}{16.8}{\rmdefault}{\mddefault}{\itdefault}{$x+a_2+b_1$}%
}}}}
\put(634,-22319){\rotatebox{90.0}{\makebox(0,0)[lb]{\smash{{\SetFigFont{14}{16.8}{\rmdefault}{\mddefault}{\itdefault}{$y+z+a_3+b_2$}%
}}}}}
\put(590,-16649){\rotatebox{90.0}{\makebox(0,0)[lb]{\smash{{\SetFigFont{14}{16.8}{\rmdefault}{\mddefault}{\itdefault}{$y+a_2+b_1$}%
}}}}}
\put(11378,-14182){\makebox(0,0)[lb]{\smash{{\SetFigFont{14}{16.8}{\rmdefault}{\mddefault}{\itdefault}{$x+a_1+a_3+b_2$}%
}}}}
\put(15030,-14916){\rotatebox{300.0}{\makebox(0,0)[lb]{\smash{{\SetFigFont{14}{16.8}{\rmdefault}{\mddefault}{\itdefault}{$2y+2a_2+b_1-1$}%
}}}}}
\put(15730,-17666){\makebox(0,0)[lb]{\smash{{\SetFigFont{14}{16.8}{\rmdefault}{\mddefault}{\itdefault}{$b_1$}%
}}}}
\put(14502,-18257){\makebox(0,0)[lb]{\smash{{\SetFigFont{14}{16.8}{\rmdefault}{\mddefault}{\itdefault}{$b_2$}%
}}}}
\put(12365,-18080){\makebox(0,0)[lb]{\smash{{\SetFigFont{14}{16.8}{\rmdefault}{\mddefault}{\itdefault}{$a_3$}%
}}}}
\put(11445,-17666){\makebox(0,0)[lb]{\smash{{\SetFigFont{14}{16.8}{\rmdefault}{\mddefault}{\itdefault}{$a_2$}%
}}}}
\put(9512,-18233){\rotatebox{300.0}{\makebox(0,0)[lb]{\smash{{\SetFigFont{14}{16.8}{\rmdefault}{\mddefault}{\itdefault}{$2a_1$}%
}}}}}
\put(10116,-17190){\rotatebox{90.0}{\makebox(0,0)[lb]{\smash{{\SetFigFont{14}{16.8}{\rmdefault}{\mddefault}{\itdefault}{$y+a_2+b_1-\frac{1}{2}$}%
}}}}}
\put(9912,-23165){\rotatebox{90.0}{\makebox(0,0)[lb]{\smash{{\SetFigFont{14}{16.8}{\rmdefault}{\mddefault}{\itdefault}{$y+z+a_3+b_2-frac{1}{2}$}%
}}}}}
\put(14595,-23321){\rotatebox{60.0}{\makebox(0,0)[lb]{\smash{{\SetFigFont{14}{16.8}{\rmdefault}{\mddefault}{\itdefault}{$2y+z+2a_1+2a_3+2b_2-1$}%
}}}}}
\put(11212,-25166){\makebox(0,0)[lb]{\smash{{\SetFigFont{14}{16.8}{\rmdefault}{\mddefault}{\itdefault}{$x+a_2+b_1$}%
}}}}
\put(16798,-18070){\makebox(0,0)[lb]{\smash{{\SetFigFont{14}{16.8}{\rmdefault}{\mddefault}{\itdefault}{$z$}%
}}}}
\put(410,-5670){\rotatebox{300.0}{\makebox(0,0)[lb]{\smash{{\SetFigFont{14}{16.8}{\rmdefault}{\mddefault}{\itdefault}{$2a_1$}%
}}}}}
\put(2354,-4948){\makebox(0,0)[lb]{\smash{{\SetFigFont{14}{16.8}{\rmdefault}{\mddefault}{\itdefault}{$a_2$}%
}}}}
\put(3172,-5421){\makebox(0,0)[lb]{\smash{{\SetFigFont{14}{16.8}{\rmdefault}{\mddefault}{\itdefault}{$a_3$}%
}}}}
\put(5934,-5421){\makebox(0,0)[lb]{\smash{{\SetFigFont{14}{16.8}{\rmdefault}{\mddefault}{\itdefault}{$b_2$}%
}}}}
\put(6956,-4830){\makebox(0,0)[lb]{\smash{{\SetFigFont{14}{16.8}{\rmdefault}{\mddefault}{\itdefault}{$b_1$}%
}}}}
\put(2558,-1110){\makebox(0,0)[lb]{\smash{{\SetFigFont{14}{16.8}{\rmdefault}{\mddefault}{\itdefault}{$x+a_1+a_3+b_2$}%
}}}}
\put(6138,-2127){\rotatebox{300.0}{\makebox(0,0)[lb]{\smash{{\SetFigFont{14}{16.8}{\rmdefault}{\mddefault}{\itdefault}{$2y+2a_2+b_1$}%
}}}}}
\put(8082,-5244){\makebox(0,0)[lb]{\smash{{\SetFigFont{14}{16.8}{\rmdefault}{\mddefault}{\itdefault}{$z$}%
}}}}
\put(5524,-10618){\rotatebox{60.0}{\makebox(0,0)[lb]{\smash{{\SetFigFont{14}{16.8}{\rmdefault}{\mddefault}{\itdefault}{$2y+z+2a_1+2a_3+2b_2$}%
}}}}}
\put(2047,-12685){\makebox(0,0)[lb]{\smash{{\SetFigFont{14}{16.8}{\rmdefault}{\mddefault}{\itdefault}{$x+a_2+b_1$}%
}}}}
\put(615,-9673){\rotatebox{90.0}{\makebox(0,0)[lb]{\smash{{\SetFigFont{14}{16.8}{\rmdefault}{\mddefault}{\itdefault}{$y+z+a_3+b_2$}%
}}}}}
\put(571,-4003){\rotatebox{90.0}{\makebox(0,0)[lb]{\smash{{\SetFigFont{14}{16.8}{\rmdefault}{\mddefault}{\itdefault}{$y+a_2+b_1$}%
}}}}}
\put(11359,-1536){\makebox(0,0)[lb]{\smash{{\SetFigFont{14}{16.8}{\rmdefault}{\mddefault}{\itdefault}{$x+a_1+a_3+b_2$}%
}}}}
\put(15011,-2270){\rotatebox{300.0}{\makebox(0,0)[lb]{\smash{{\SetFigFont{14}{16.8}{\rmdefault}{\mddefault}{\itdefault}{$2y+2a_2+b_1-1$}%
}}}}}
\put(15711,-5020){\makebox(0,0)[lb]{\smash{{\SetFigFont{14}{16.8}{\rmdefault}{\mddefault}{\itdefault}{$b_1$}%
}}}}
\put(14483,-5611){\makebox(0,0)[lb]{\smash{{\SetFigFont{14}{16.8}{\rmdefault}{\mddefault}{\itdefault}{$b_2$}%
}}}}
\put(12346,-5434){\makebox(0,0)[lb]{\smash{{\SetFigFont{14}{16.8}{\rmdefault}{\mddefault}{\itdefault}{$a_3$}%
}}}}
\put(11426,-5020){\makebox(0,0)[lb]{\smash{{\SetFigFont{14}{16.8}{\rmdefault}{\mddefault}{\itdefault}{$a_2$}%
}}}}
\put(9493,-5587){\rotatebox{300.0}{\makebox(0,0)[lb]{\smash{{\SetFigFont{14}{16.8}{\rmdefault}{\mddefault}{\itdefault}{$2a_1$}%
}}}}}
\put(10097,-4544){\rotatebox{90.0}{\makebox(0,0)[lb]{\smash{{\SetFigFont{14}{16.8}{\rmdefault}{\mddefault}{\itdefault}{$y+a_2+b_1-\frac{1}{2}$}%
}}}}}
\put(9893,-10519){\rotatebox{90.0}{\makebox(0,0)[lb]{\smash{{\SetFigFont{14}{16.8}{\rmdefault}{\mddefault}{\itdefault}{$y+z+a_3+b_2-\frac{1}{2}$}%
}}}}}
\put(14576,-10675){\rotatebox{60.0}{\makebox(0,0)[lb]{\smash{{\SetFigFont{14}{16.8}{\rmdefault}{\mddefault}{\itdefault}{$2y+z+2a_1+2a_3+2b_2-1$}%
}}}}}
\put(11193,-12520){\makebox(0,0)[lb]{\smash{{\SetFigFont{14}{16.8}{\rmdefault}{\mddefault}{\itdefault}{$x+a_2+b_1$}%
}}}}
\put(429,-18316){\rotatebox{300.0}{\makebox(0,0)[lb]{\smash{{\SetFigFont{14}{16.8}{\rmdefault}{\mddefault}{\itdefault}{$2a_1$}%
}}}}}
\end{picture}}
\caption{(a) The region $R^{(1)}_{3,1,2}(2,2,2;\ 2,3)$. (b) The region $R^{(2)}_{2,1,2}(2,2,2; \ 2,3)$.  (c) The weighted region $RW^{(1)}_{3,1,2}(2,2,2;\ 2,3)$. (d) The weighted region $RW^{(2)}_{2,1,2}(2,2,2; \ 2,3)$.}\label{fig:halvedhex2}
\end{figure}

Start with a halved hexagon  of side-lengths $x+\od_a+\e_b$, $2y+z+2\e_a+2\od_b$, $2y+z+2\od_a+2\e_b$, $x+\e_a+\od_b$,
 $2y+z+a+b$.
Remove the `\emph{upside down}' $a$-fern from the west side of the hexagon and remove
the normal $b$-fern from the northeast side at the level $z$ above the rightmost vertex of the halved hexagon as shown
in Figure \ref{fig:halvedhex2} (a). Let $R^{(1)}_{x,y,z}(\textbf{a};\ \textbf{b})$ denote
the resulting region.

\begin{thm}\label{mainR1} Assume that $x,y,z$ are non-negative integers and that $\textbf{a}=(a_1,a_2,\dotsc,a_m)$ and $\textbf{b}=(b_1,b_2,\dotsc,b_n)$ are two sequences of non-negative integers. Then
\begin{align}\label{mainR1eq}
\M(R^{(1)}_{x,y,z}(\textbf{a};\textbf{b}))&=\frac{\M(R^{(1)}_{x+y,0,z}(\textbf{a};\textbf{b}))
\M(R^{(1)}_{0,2y,z}(\textbf{a};\textbf{b}))}{\M(R^{(1)}_{y,0,z}(\textbf{a};\textbf{b}))}\notag\\
&\quad\times \frac{\T(x+1,2a+b+2y+z,y)\V(2x+2a+3,b+2y+z-1,y)}{\T(1,2a+b+2y+z,y)\V(2a+3,b+2y+z-1,y)}\notag\\
&=2^{-y}\Q(a_1,\dotsc,a_{2\lceil \frac{m-1}{2}\rceil}, a_{2\lceil \frac{m-1}{2}\rceil+1}+x+y+b_{2\lfloor\frac{n+1}{2}\rfloor},b_{2\lfloor\frac{n+1}{2}\rfloor-1},\dotsc,b_1)\notag\\
&\quad\times \Q(0,a_1,\dotsc,a_{2\lfloor\frac{m+1}{2}\rfloor-1},a_{2\lfloor\frac{m+1}{2}\rfloor}+x+y+b_{2\lceil \frac{n-1}{2}\rceil+1},b_{2\lceil \frac{n-1}{2}\rceil},\dotsc, b_1,z) \notag\\
&\quad\times \frac{\Hf_2(2\e_a+2\od_b+1)\Hf_2(2\od_a+2\e_b+2z+1)}{\Hf_2(2\e_a+2\od_b+2y+1)\Hf_2(2\od_a+2\e_b+2y+2z+1)} \notag\\
&\quad\times \frac{\Hf(2a+b+2y+z+1)\Hf(b+y+z)}{\Hf(2a+b+y+z+1)\Hf(b+z)}\notag\\
&\quad\times  \frac{\T(x+1,2a+b+2y+z,y)\V(2x+2a+3,b+2y+z-1,y)}{\T(1,2a+b+2y+z,y)\V(2a+3,b+2y+z-1,y)}.
\end{align}
\end{thm}

We consider next a variation of  the above $R^{(1)}$-type region, where the initial halved hexagon has side-lengths $x+\od_a+\e_b$,
$2y+z+2\e_a+2\od_b-1$, $2y+z+2\od_a+2\e_b-1$, $x+\e_a+\od_b$,
 $2y+z+a+b-1$, and where the $a$- and $b$-ferns are removed in the same way as in the case of $R^{(1)}$-type regions.
 We denote by $R^{(2)}_{x,y,z}(\textbf{a};\textbf{b})$ the resulting region (illustrated in Figure
 \ref{fig:halvedhex2}(b)).

\begin{thm}\label{mainR2} With the same notations in Theorem \ref{mainR1}, the number of tilings of the $R^{(2)}$-type region is given by
\begin{align}\label{mainR2eq}
\M(R^{(2)}_{x,y,z}(\textbf{a};\textbf{b}))&=\frac{\M(R^{(2)}_{x+y,0,z}(\textbf{a};\textbf{b}))
\M(R^{(2)}_{0,2y,z}(\textbf{a};\textbf{b}))}{\M(R^{(2)}_{y,0,z}(\textbf{a};\textbf{b}))}\notag\\
&\quad\times \frac{\T(x+1,2a+b+2y+z-1,y)\V(2x+2a+3,b+2y+z-2,y)}{\T(1,2a+b+2y+z-1,y)\V(2a+3,b+2y+z-2,y)}\notag\\
&=\K(a_1,\dotsc,a_{2\lceil \frac{m-1}{2}\rceil}, a_{2\lceil \frac{m-1}{2}\rceil+1}+x+y+b_{2\lfloor\frac{n+1}{2}\rfloor},b_{2\lfloor\frac{n+1}{2}\rfloor-1},\dotsc,b_1)\notag\\
&\quad\times \K(0,a_1,\dotsc,a_{2\lfloor\frac{m+1}{2}\rfloor-1},a_{2\lfloor\frac{m+1}{2}\rfloor}+x+y+b_{2\lceil \frac{n-1}{2}\rceil+1},b_{2\lceil \frac{n-1}{2}\rceil},\dotsc, b_1,z) \notag\\
&\quad\times \frac{(2a-1)!!}{(2a+2y-1)!!} \frac{\Hf_2(2\e_a+2\od_b)\Hf_2(2\od_a+2\e_b+2z)}{\Hf_2(2\e_a+2\od_b+2y)
\Hf_2(2\od_a+2\e_b+2y+2z)} \notag\\
&\quad\times \frac{\Hf(2a+b+2y+z)\Hf(b+y+z)}{\Hf(2a+b+y+z)\Hf(b+z)}\notag\\
&\quad\times  \frac{\T(x+1,2a+b+2y+z-1,y)\V(2x+2a+3,b+2y+z-2,y)}{\T(1,2a+b+2y+z-1,y)\V(2a+3,b+2y+z-2,y)}.
\end{align}
\end{thm}

We are also investigate the weighted versions of the above `reversing' regions with the vertical lozenges along the west side weighted by
$1/2$. Denote the weighted version of $R^{(i)}_{x,y,z}(\textbf{a};\textbf{b})$ by
$RW^{(i)}_{x,y,z}(\textbf{a};\textbf{b})$, for $i=1,2$. These newly defined regions are illustrated in
 Figures \ref{fig:halvedhex2}(c) and (d), respectively, and their tiling numbers are also given by simple product formulas as below.

\begin{thm}\label{mainRW1} For non-negative integers $x,y,z$ and sequences of non-negative integers $\textbf{a}=(a_1,a_2,\dotsc,a_m)$ and $\textbf{b}=(b_1,b_2,\dotsc,b_n)$, we have
\begin{align}\label{mainRW1eq}
\M(RW^{(1)}_{x,y,z}(\textbf{a}; \textbf{b}))&=\frac{\M(RW^{(1)}_{x+y,0,z}(\textbf{a};\textbf{b}))
\M(RW^{(1)}_{0,2y,z}(\textbf{a};\textbf{b}))}{\M(RW^{(1)}_{y,0,z}(\textbf{a};\textbf{b}))}\notag\\
&\quad\times \frac{\T(x+1,2a+b+2y+z-1,y)\V(2x+2a+1,b+2y+z,y)}{\T(1,2a+b+2y+z-1,y)\V(2a+1,b+2y+z,y)}\notag\\
&=2^{a_1-2y}\Q'(a_1,\dotsc,a_{2\lceil \frac{m-1}{2}\rceil}, a_{2\lceil \frac{m-1}{2}\rceil+1}+x+y+b_{2\lfloor\frac{n+1}{2}\rfloor},b_{2\lfloor\frac{n+1}{2}\rfloor-1},\dotsc,b_1)\notag\\
&\quad\times \Q'(0,a_1,\dotsc,a_{2\lfloor\frac{m+1}{2}\rfloor-1},a_{2\lfloor\frac{m+1}{2}\rfloor}+x+y+b_{2\lceil \frac{n-1}{2}\rceil+1},b_{2\lceil \frac{n-1}{2}\rceil},\dotsc, b_1,z) \notag\\
&\quad\times \frac{(2a+2y-1)!!}{(2a-1)!!} \frac{\Hf_2(2\e_a+2\od_b+1)\Hf_2(2\od_a+2\e_b+2z+1)}
{\Hf_2(2\e_a+2\od_b+2y+1)\Hf_2(2\od_a+2\e_b+2y+2z+1)} \notag\\
&\quad\times \frac{\Hf(2a+b+2y+z)\Hf(b+y+z)}{\Hf(2a+b+y+z)\Hf(b+z)}\notag\\
&\quad\times  \frac{\T(x+1,2a+b+2y+z-1,y)\V(2x+2a+1,b+2y+z,y)}{\T(1,2a+b+2y+z-1,y)\V(2a+1,b+2y+z,y)}.
\end{align}
\end{thm}

\begin{thm}\label{mainRW2} With the same notations in Theorem \ref{mainRW1}, the weighted number of tilings of the $RW^{(2)}$-type region is given by
\begin{align}\label{mainRW2eq}
\M(RW^{(2)}_{x,y,z}(\textbf{a};\textbf{b}))&=\frac{\M(RW^{(2)}_{x+y,0,z}(\textbf{a};\textbf{b}))
\M(RW^{(2)}_{0,2y,z}(\textbf{a};\textbf{b}))}{\M(RW^{(2)}_{y,0,z}(\textbf{a};\textbf{b}))}\notag\\
&\quad\times \frac{\T(x+1,2a+b+2y+z-2,y)\V(2x+2a+1,b+2y+z-1,y)}{\T(1,2a+b+2y+z-2,y)\V(2a+1,b+2y+z-1,y)}\\
&=2^{a_1-y-1}\K'(a_1,\dotsc,a_{2\lceil \frac{m-1}{2}\rceil}, a_{2\lceil \frac{m-1}{2}\rceil+1}+x+y+b_{2\lfloor\frac{n+1}{2}\rfloor},b_{2\lfloor\frac{n+1}{2}\rfloor-1},\dotsc,b_1)\notag\\
&\quad\times \K'(0,a_1,\dotsc,a_{2\lfloor\frac{m+1}{2}\rfloor-1},a_{2\lfloor\frac{m+1}{2}\rfloor}+x+y+b_{2\lceil \frac{n-1}{2}\rceil+1},b_{2\lceil \frac{n-1}{2}\rceil},\dotsc, b_1,z) \notag\\
&\quad\times  \frac{\Hf_2(2\e_a+2\od_b)\Hf_2(2\od_a+2\e_b+2z)}{\Hf_2(2\e_a+2\od_b+2y)\Hf_2(2\od_a+2\e_b+2y+2z)} \notag\\
&\quad\times \frac{\Hf(2a+b+2y+z-1)\Hf(b+y+z)}{\Hf(2a+b+y+z-1)\Hf(b+z)}\notag\\
&\quad\times  \frac{\T(x+1,2a+b+2y+z-2,y)\V(2x+2a+1,b+2y+z-1,y)}{\T(1,2a+b+2y+z-2,y)\V(2a+1,b+2y+z-1,y)}.
\end{align}
\end{thm}

In the next part of this section, we consider the situations when only half  of the lozenges
along the west side of our halved hexagons are weighted.
In particular, either the portion of above or the portion below the lattice line, on which the two ferns are lying down,
 has the adjacent  lozenges weighted by $1/2$.
 Our first `mixed-boundary'  region is defined in the next paragraph.

\begin{figure}\centering
\setlength{\unitlength}{3947sp}%
\begingroup\makeatletter\ifx\SetFigFont\undefined%
\gdef\SetFigFont#1#2#3#4#5{%
  \reset@font\fontsize{#1}{#2pt}%
  \fontfamily{#3}\fontseries{#4}\fontshape{#5}%
  \selectfont}%
\fi\endgroup%
\resizebox{13cm}{!}{
\begin{picture}(0,0)%
\includegraphics{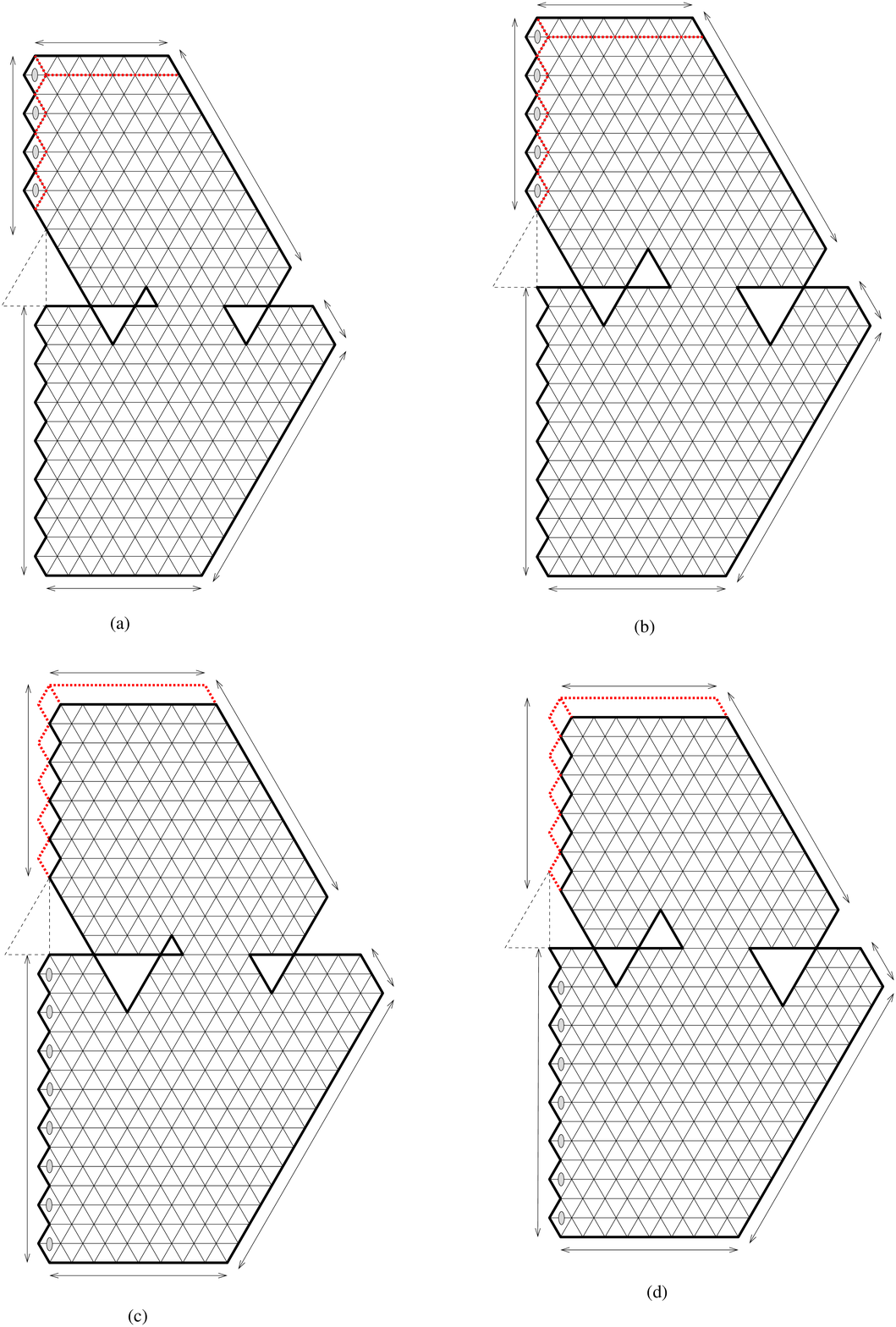}%
\end{picture}%
%
%

\begin{picture}(16686,24702)(1807,-24455)
\put(14011,-17309){\makebox(0,0)[lb]{\smash{{\SetFigFont{14}{16.8}{\rmdefault}{\mddefault}{\itdefault}{$a_3$}%
}}}}
\put(13201,-17774){\makebox(0,0)[lb]{\smash{{\SetFigFont{14}{16.8}{\rmdefault}{\mddefault}{\itdefault}{$a_2$}%
}}}}
\put(17026,-21494){\rotatebox{60.0}{\makebox(0,0)[lb]{\smash{{\SetFigFont{14}{16.8}{\rmdefault}{\mddefault}{\itdefault}{$2y+z+2a_2+2b_2-1$}%
}}}}}
\put(13445,-23451){\makebox(0,0)[lb]{\smash{{\SetFigFont{14}{16.8}{\rmdefault}{\mddefault}{\itdefault}{$x+a_1+b_1$}%
}}}}
\put(11633,-20624){\rotatebox{90.0}{\makebox(0,0)[lb]{\smash{{\SetFigFont{14}{16.8}{\rmdefault}{\mddefault}{\itdefault}{$y+z+a_2+b_2-\frac{1}{2}$}%
}}}}}
\put(16272,-17931){\makebox(0,0)[lb]{\smash{{\SetFigFont{14}{16.8}{\rmdefault}{\mddefault}{\itdefault}{$b_2$}%
}}}}
\put(17375,-17301){\makebox(0,0)[lb]{\smash{{\SetFigFont{14}{16.8}{\rmdefault}{\mddefault}{\itdefault}{$b_1$}%
}}}}
\put(18313,-17691){\makebox(0,0)[lb]{\smash{{\SetFigFont{14}{16.8}{\rmdefault}{\mddefault}{\itdefault}{$z$}%
}}}}
\put(16246,-13859){\rotatebox{300.0}{\makebox(0,0)[lb]{\smash{{\SetFigFont{14}{16.8}{\rmdefault}{\mddefault}{\itdefault}{$2y+2a_1+2a_3+b_1-1$}%
}}}}}
\put(13344,-12487){\makebox(0,0)[lb]{\smash{{\SetFigFont{14}{16.8}{\rmdefault}{\mddefault}{\itdefault}{$x+a_2+b_2$}%
}}}}
\put(11506,-15239){\rotatebox{90.0}{\makebox(0,0)[lb]{\smash{{\SetFigFont{14}{16.8}{\rmdefault}{\mddefault}{\itdefault}{$y+a_3+b_1$}%
}}}}}
\put(11445,-17009){\rotatebox{60.0}{\makebox(0,0)[lb]{\smash{{\SetFigFont{14}{16.8}{\rmdefault}{\mddefault}{\itdefault}{$2a_1$}%
}}}}}
\put(7563,-21905){\rotatebox{60.0}{\makebox(0,0)[lb]{\smash{{\SetFigFont{14}{16.8}{\rmdefault}{\mddefault}{\itdefault}{$2y+z+2a_2+2b_2$}%
}}}}}
\put(4017,-23879){\makebox(0,0)[lb]{\smash{{\SetFigFont{14}{16.8}{\rmdefault}{\mddefault}{\itdefault}{$x+a_1+b_1$}%
}}}}
\put(2296,-21509){\rotatebox{90.0}{\makebox(0,0)[lb]{\smash{{\SetFigFont{14}{16.8}{\rmdefault}{\mddefault}{\itdefault}{$y+z+a_2+b_2$}%
}}}}}
\put(6841,-13672){\rotatebox{300.0}{\makebox(0,0)[lb]{\smash{{\SetFigFont{14}{16.8}{\rmdefault}{\mddefault}{\itdefault}{$2y+2a_1+2a_3+b_1$}%
}}}}}
\put(2341,-14909){\rotatebox{90.0}{\makebox(0,0)[lb]{\smash{{\SetFigFont{14}{16.8}{\rmdefault}{\mddefault}{\itdefault}{$y+a_3+b_1$}%
}}}}}
\put(3879,-12314){\makebox(0,0)[lb]{\smash{{\SetFigFont{14}{16.8}{\rmdefault}{\mddefault}{\itdefault}{$x+a_2+b_2$}%
}}}}
\put(9125,-17744){\makebox(0,0)[lb]{\smash{{\SetFigFont{14}{16.8}{\rmdefault}{\mddefault}{\itdefault}{$z$}%
}}}}
\put(7917,-17294){\makebox(0,0)[lb]{\smash{{\SetFigFont{14}{16.8}{\rmdefault}{\mddefault}{\itdefault}{$b_1$}%
}}}}
\put(6882,-17924){\makebox(0,0)[lb]{\smash{{\SetFigFont{14}{16.8}{\rmdefault}{\mddefault}{\itdefault}{$b_2$}%
}}}}
\put(5026,-17841){\makebox(0,0)[lb]{\smash{{\SetFigFont{14}{16.8}{\rmdefault}{\mddefault}{\itdefault}{$a_3$}%
}}}}
\put(4209,-18044){\makebox(0,0)[lb]{\smash{{\SetFigFont{14}{16.8}{\rmdefault}{\mddefault}{\itdefault}{$a_2$}%
}}}}
\put(2257,-17129){\rotatebox{60.0}{\makebox(0,0)[lb]{\smash{{\SetFigFont{14}{16.8}{\rmdefault}{\mddefault}{\itdefault}{$2a_1$}%
}}}}}
\put(16674,-9421){\rotatebox{60.0}{\makebox(0,0)[lb]{\smash{{\SetFigFont{14}{16.8}{\rmdefault}{\mddefault}{\itdefault}{$2y+z+2a_2+2b_2-1$}%
}}}}}
\put(13182,-11198){\makebox(0,0)[lb]{\smash{{\SetFigFont{14}{16.8}{\rmdefault}{\mddefault}{\itdefault}{$x+a_1+b_1$}%
}}}}
\put(11423,-8506){\rotatebox{90.0}{\makebox(0,0)[lb]{\smash{{\SetFigFont{14}{16.8}{\rmdefault}{\mddefault}{\itdefault}{$y+z+a_2+b_2-\frac{1}{2}$}%
}}}}}
\put(11311,-2874){\rotatebox{90.0}{\makebox(0,0)[lb]{\smash{{\SetFigFont{14}{16.8}{\rmdefault}{\mddefault}{\itdefault}{$y+a_3+b_1$}%
}}}}}
\put(12867,-31){\makebox(0,0)[lb]{\smash{{\SetFigFont{14}{16.8}{\rmdefault}{\mddefault}{\itdefault}{$x+a_2+b_2$}%
}}}}
\put(8303,-5880){\makebox(0,0)[lb]{\smash{{\SetFigFont{14}{16.8}{\rmdefault}{\mddefault}{\itdefault}{$z$}%
}}}}
\put(2198,-8895){\rotatebox{90.0}{\makebox(0,0)[lb]{\smash{{\SetFigFont{14}{16.8}{\rmdefault}{\mddefault}{\itdefault}{$y+z+a_2+b_2$}%
}}}}}
\put(7343,-5415){\makebox(0,0)[lb]{\smash{{\SetFigFont{14}{16.8}{\rmdefault}{\mddefault}{\itdefault}{$b_1$}%
}}}}
\put(6398,-6000){\makebox(0,0)[lb]{\smash{{\SetFigFont{14}{16.8}{\rmdefault}{\mddefault}{\itdefault}{$b_2$}%
}}}}
\put(2273,-5085){\rotatebox{60.0}{\makebox(0,0)[lb]{\smash{{\SetFigFont{14}{16.8}{\rmdefault}{\mddefault}{\itdefault}{$2a_1$}%
}}}}}
\put(3953,-5977){\makebox(0,0)[lb]{\smash{{\SetFigFont{14}{16.8}{\rmdefault}{\mddefault}{\itdefault}{$a_2$}%
}}}}
\put(4560,-5887){\makebox(0,0)[lb]{\smash{{\SetFigFont{14}{16.8}{\rmdefault}{\mddefault}{\itdefault}{$a_3$}%
}}}}
\put(2033,-3540){\rotatebox{90.0}{\makebox(0,0)[lb]{\smash{{\SetFigFont{14}{16.8}{\rmdefault}{\mddefault}{\itdefault}{$y+a_3+b_1+\frac{1}{2}$}%
}}}}}
\put(3375,-727){\makebox(0,0)[lb]{\smash{{\SetFigFont{14}{16.8}{\rmdefault}{\mddefault}{\itdefault}{$x+a_2+b_2$}%
}}}}
\put(6045,-1950){\rotatebox{300.0}{\makebox(0,0)[lb]{\smash{{\SetFigFont{14}{16.8}{\rmdefault}{\mddefault}{\itdefault}{$2y+2a_1+2a_3+b_1+1$}%
}}}}}
\put(6863,-9637){\rotatebox{60.0}{\makebox(0,0)[lb]{\smash{{\SetFigFont{14}{16.8}{\rmdefault}{\mddefault}{\itdefault}{$2y+z+2a_2+2b_2$}%
}}}}}
\put(3653,-11130){\makebox(0,0)[lb]{\smash{{\SetFigFont{14}{16.8}{\rmdefault}{\mddefault}{\itdefault}{$x+a_1+b_1$}%
}}}}
\put(11311,-4771){\rotatebox{60.0}{\makebox(0,0)[lb]{\smash{{\SetFigFont{14}{16.8}{\rmdefault}{\mddefault}{\itdefault}{$2a_1$}%
}}}}}
\put(12976,-5656){\makebox(0,0)[lb]{\smash{{\SetFigFont{14}{16.8}{\rmdefault}{\mddefault}{\itdefault}{$a_2$}%
}}}}
\put(13786,-5198){\makebox(0,0)[lb]{\smash{{\SetFigFont{14}{16.8}{\rmdefault}{\mddefault}{\itdefault}{$a_3$}%
}}}}
\put(17180,-5146){\makebox(0,0)[lb]{\smash{{\SetFigFont{14}{16.8}{\rmdefault}{\mddefault}{\itdefault}{$b_1$}%
}}}}
\put(16077,-5791){\makebox(0,0)[lb]{\smash{{\SetFigFont{14}{16.8}{\rmdefault}{\mddefault}{\itdefault}{$b_2$}%
}}}}
\put(18110,-5536){\makebox(0,0)[lb]{\smash{{\SetFigFont{14}{16.8}{\rmdefault}{\mddefault}{\itdefault}{$z$}%
}}}}
\put(15856,-1486){\rotatebox{300.0}{\makebox(0,0)[lb]{\smash{{\SetFigFont{14}{16.8}{\rmdefault}{\mddefault}{\itdefault}{$2y+2a_1+2a_3+b_1$}%
}}}}}
\end{picture}%
}
\caption{The four mixed-boundary regions: (a) The region $N^{(1)}_{2,1,2}(2,2,1;\ 2,2)$. (b) The region $N^{(2)}_{2,2,2}(2,2,2;\ 2,3)$.  (c) The region $N^{(3)}_{2,1,2}(2,3,1;\ 3,2)$.
(d) The region $N^{(4)}_{2,1,2}(2,2,2;\ 2,3)$.}\label{fig:halvedhex3}
\end{figure}

The first mixed-boundary region $N^{(1)}_{x,y,z}(\textbf{a};\ \textbf{b})$ is obtained
 from the region $H^{(1)}_{x,y,z}(\textbf{a};\ \textbf{b})$
  by adding a layer of unit triangles running along the north side, as well as a layer running along the portion of the west side
 above the two ferns as shown in Figure \ref{fig:halvedhex3}(a);
  the added unit triangles are restricted between the bold and the dotted contours. In the resulting region, we assign to each vertical lozenge
  above the ferns and running along the west side a weight $1/2$.
  The second region with mixed boundary is defined similarly, the only difference is that it is now obtained
  by applying the same lozenge-adding procedure to the region $H^{(2)}_{x,y,z}(\textbf{a};\ \textbf{b})$
   (see Figure \ref{fig:halvedhex3}(b)).

\begin{thm}\label{mainM1} Assume that $x,y,z$ are non-negative integers and that $\textbf{a}=(a_1,a_2,\dotsc,a_m)$ and $\textbf{b}=(b_1,b_2,\dotsc,b_n)$ are two sequences of non-negative integers. Then
\begin{align}\label{mainM1eq}
\M(N^{(1)}_{x,y,z}(\textbf{a}; \textbf{b}))&=\frac{\M(N^{(1)}_{x+y,0,z}(\textbf{a}; \textbf{b}))
\M(N^{(1)}_{0,2y,z}(\textbf{a};\textbf{b}))}{\M(N^{(1)}_{y,0,z}(\textbf{a};\textbf{b}))}\notag\\
&\quad\times \frac{\T(x+1,2a+b+2y+z,y)\T(x+a+1,b+2y+z,y)}{\T(1,2a+b+2y+z,y)\T(a+1,b+2y+z,y)}\notag\\
&=2^{a_1-y}\K'(0,a_1+1,a_2,\dotsc,a_{2\lfloor\frac{m+1}{2}\rfloor-1},a_{2\lfloor\frac{m+1}{2}\rfloor}+x+y+b_{2\lfloor\frac{n+1}{2}\rfloor},b_{2\lfloor\frac{n+1}{2}\rfloor-1},\dotsc,b_1)\notag\\
&\quad\times \Q(a_1,\dotsc,a_{\lceil \frac{m-1}{2}\rceil}, a_{\lceil \frac{m-1}{2}\rceil+1}+x+y+b_{\lceil \frac{n-1}{2}\rceil+1},b_{\lceil \frac{n-1}{2}\rceil},\dotsc, b_1,z) \notag\\
&\quad\times \frac{(a+y)!}{a!} \frac{\Hf_2(2\od_a+2\od_b+2)\Hf_2(2\e_a+2\e_b+2z+1)}{\Hf_2(2\od_a+2\od_b+2y+2)
\Hf_2(2\e_a+2\e_b+2y+2z+1)} \notag\\
&\quad\times \frac{\Hf(2a+b+2y+z+1)\Hf(b+y+z)}{\Hf(2a+b+y+z+1)\Hf(b+z)}\notag\\
&\quad\times   \frac{\T(x+1,2a+b+2y+z,y)\T(x+a+1,b+2y+z,y)}{\T(1,2a+b+2y+z,y)\T(a+1,b+2y+z,y)}.
\end{align}
\end{thm}

\begin{thm}\label{mainM2} With the same notations in Theorem \ref{mainM1}, the weighted number of tilings of the $N^{(2)}$-type region is given by
\begin{align}\label{mainM2eq}
\M(N^{(2)}_{x,y,z}(\textbf{a}; \textbf{b}))&=\frac{\M(N^{(2)}_{x+y,0,z}(\textbf{a};\textbf{b}))
\M(N^{(2)}_{0,2y,z}(\textbf{a};\textbf{b}))}{\M(N^{(2)}_{y,0,z}(\textbf{a};\textbf{b}))}\notag\\
&\quad\times \frac{\T(x+1,2a+b+2y+z-1,y)\T(x+a+1,b+2y+z-1,y)}{\T(1,2a+b+2y+z-1,y)\T(a+1,b+2y+z-1,y)}\notag\\
&=2^{a_1-y}\Q'(0,a_1,\dotsc,a_{2\lfloor\frac{m+1}{2}\rfloor-1},a_{2\lfloor\frac{m+1}{2}\rfloor}+x+y+b_{2\lfloor\frac{n+1}{2}\rfloor},b_{2\lfloor\frac{n+1}{2}\rfloor-1},\dotsc,b_1)\notag\\
&\quad\times \K(a_1,\dotsc,a_{\lceil \frac{m-1}{2}\rceil}, a_{\lceil \frac{m-1}{2}\rceil+1}+x+y+b_{\lceil \frac{n-1}{2}\rceil+1},b_{\lceil \frac{n-1}{2}\rceil},\dotsc, b_1,z) \notag\\
&\quad\times  \frac{\Hf_2(2\od_a+2\od_b+1)\Hf_2(2\e_a+2\e_b+2z)}{\Hf_2(2\od_a+2\od_b+2y+1)
\Hf_2(2\e_a+2\e_b+2y+2z)} \notag\\
&\quad\times \frac{\Hf(2a+b+2y+z)\Hf(b+y+z)}{\Hf(2a+b+y+z)\Hf(b+z)}\notag\\
&\quad\times    \frac{\T(x+1,2a+b+2y+z-1,y)\T(x+a+1,b+2y+z-1,y)}{\T(1,2a+b+2y+z-1,y)\T(a+1,b+2y+z-1,y)}.
\end{align}
\end{thm}

In the above $N^{(1)}$- and $N^{(2)}$-type regions, only the portion above the ferns of the west side is weighted.
In contrast, our next two regions have the portion \emph{below} the ferns weighted.
In particular, we remove from the weighted region $W^{(i)}_{x,y,z}(\textbf{a}; \textbf{b})$, $i=1,2$, all unit triangles running along its north side and the ones running along
the portion of the west side above the fern (indicated by the part between the dotted lines and the bold contour in Figures \ref{fig:halvedhex3}(c) and (d)).
The removal of the unit triangles from the region $W^{(1)}_{x,y,z}(\textbf{a}; \textbf{b})$ gives the
new region $N^{(3)}_{x,y,z}(\textbf{a}; \textbf{b})$ (shown in Figure \ref{fig:halvedhex3}(c)),
 and the removal from the region $W^{(2)}_{x,y,z}(\textbf{a}; \textbf{b})$ gives the region $N^{(4)}_{x,y,z}(\textbf{a}; \textbf{b})$   (illustrated in Figure \ref{fig:halvedhex3}(d)).

\begin{thm}\label{mainM3} For non-negative integers $x,y,z$ and sequences of non-negative integers $\textbf{a}=(a_1,a_2,\dotsc,a_m)$ and $\textbf{b}=(b_1,b_2,\dotsc,b_n)$, we have
\begin{align}\label{mainM3eq}
\M(N^{(3)}_{x,y,z}(\textbf{a};\textbf{b}))&=\frac{\M(N^{(3)}_{x+y,0,z}(\textbf{a}; \textbf{b}))
\M(N^{(3)}_{0,2y,z}(\textbf{a};\textbf{b}))}{\M(N^{(3)}_{y,0,z}(\textbf{a};\textbf{b}))}\notag\\
&\quad\times \frac{\T(x+1,2a+b+2y+z-1,y)\T(x+a+1,b+2y+z-1,y)}{\T(1,2a+b+2y+z-1,y)\T(a+1,b+2y+z-1,y)}\notag\\
&=2^{-y}\K(0,a_1,\dotsc,a_{2\lfloor\frac{m+1}{2}\rfloor-1},a_{2\lfloor\frac{m+1}{2}\rfloor}+x+y+b_{2\lfloor\frac{n+1}{2}\rfloor},b_{2\lfloor\frac{n+1}{2}\rfloor-1},\dotsc,b_1)\notag\\
&\quad\times \Q'(a_1,\dotsc,a_{\lceil \frac{m-1}{2}\rceil}, a_{\lceil \frac{m-1}{2}\rceil+1}+x+y+b_{\lceil \frac{n-1}{2}\rceil+1},b_{\lceil \frac{n-1}{2}\rceil},\dotsc, b_1,z) \notag\\
&\quad\times  \frac{\Hf_2(2\od_a+2\od_b)\Hf_2(2\e_a+2\e_b+2z+1)}{\Hf_2(2\od_a+2\od_b+2y)
\Hf_2(2\e_a+2\e_b+2y+2z+1)} \notag\\
&\quad\times \frac{\Hf(2a+b+2y+z)\Hf(b+y+z)}{\Hf(2a+b+y+z)\Hf(b+z)}\notag\\
&\quad\times    \frac{\T(x+1,2a+b+2y+z-1,y)\T(x+a+1,b+2y+z-1,y)}{\T(1,2a+b+2y+z-1,y)\T(a+1,b+2y+z-1,y)}.
\end{align}
\end{thm}

\begin{thm}\label{mainM4} Assume that $x,y,z$ are non-negative integers and that $\textbf{a}=(a_1,a_2,\dotsc,a_m)$ and $\textbf{b}=(b_1,b_2,\dotsc,b_n)$ are two sequences of non-negative integers. Then
\begin{align}\label{mainM4eq}
\M(N^{(4)}_{x,y,z}(\textbf{a};\textbf{b}))&=\frac{\M(N^{(4)}_{x+y,0,z}(\textbf{a};\textbf{b}))
\M(N^{(4)}_{0,2y,z}(\textbf{a};\textbf{b}))}{\M(N^{(4)}_{y,0,z}(\textbf{a};\textbf{b}))}\notag\\
&\quad\times \frac{\T(x+1,2a+b+2y+z-2,y)\T(x+a+1,b+2y+z-2,y)}{\T(1,2a+b+2y+z-2,y)\T(a+1,b+2y+z-2,y)}\notag\\
&=2^{-y}\Q(0,a_1-1, a_2,\dotsc,a_{2\lfloor\frac{m+1}{2}\rfloor-1},a_{2\lfloor\frac{m+1}{2}\rfloor}+x+y+b_{2\lfloor\frac{n+1}{2}\rfloor},b_{2\lfloor\frac{n+1}{2}\rfloor-1},\dotsc,b_1)\notag\\
&\quad\times \K'(a_1,\dotsc,a_{\lceil \frac{m-1}{2}\rceil}, a_{\lceil \frac{m-1}{2}\rceil+1}+x+y+b_{\lceil \frac{n-1}{2}\rceil+1},b_{\lceil \frac{n-1}{2}\rceil},\dotsc, b_1,z) \notag\\
&\quad\times \frac{(a-1)!}{(a+y-1)!}  \frac{\Hf_2(2\od_a+2\od_b-1)\Hf_2(2\e_a+2\e_b+2z)}{\Hf_2(2\od_a+2\od_b+2y-1)
\Hf_2(2\e_a+2\e_b+2y+2z)} \notag\\
&\quad\times \frac{\Hf(2a+b+2y+z-1)\Hf(b+y+z)}{\Hf(2a+b+y+z-1)\Hf(b+z)}\notag\\
&\quad\times    \frac{\T(x+1,2a+b+2y+z-2,y)\T(x+a+1,b+2y+z-2,y)}{\T(1,2a+b+2y+z-2,y)\T(a+1,b+2y+z-2,y)}.
\end{align}
\end{thm}

\begin{figure}\centering
%
%
\setlength{\unitlength}{3947sp}%
\begingroup\makeatletter\ifx\SetFigFont\undefined%
\gdef\SetFigFont#1#2#3#4#5{%
  \reset@font\fontsize{#1}{#2pt}%
  \fontfamily{#3}\fontseries{#4}\fontshape{#5}%
  \selectfont}%
\fi\endgroup%
\resizebox{13cm}{!}{
\begin{picture}(0,0)%
\includegraphics{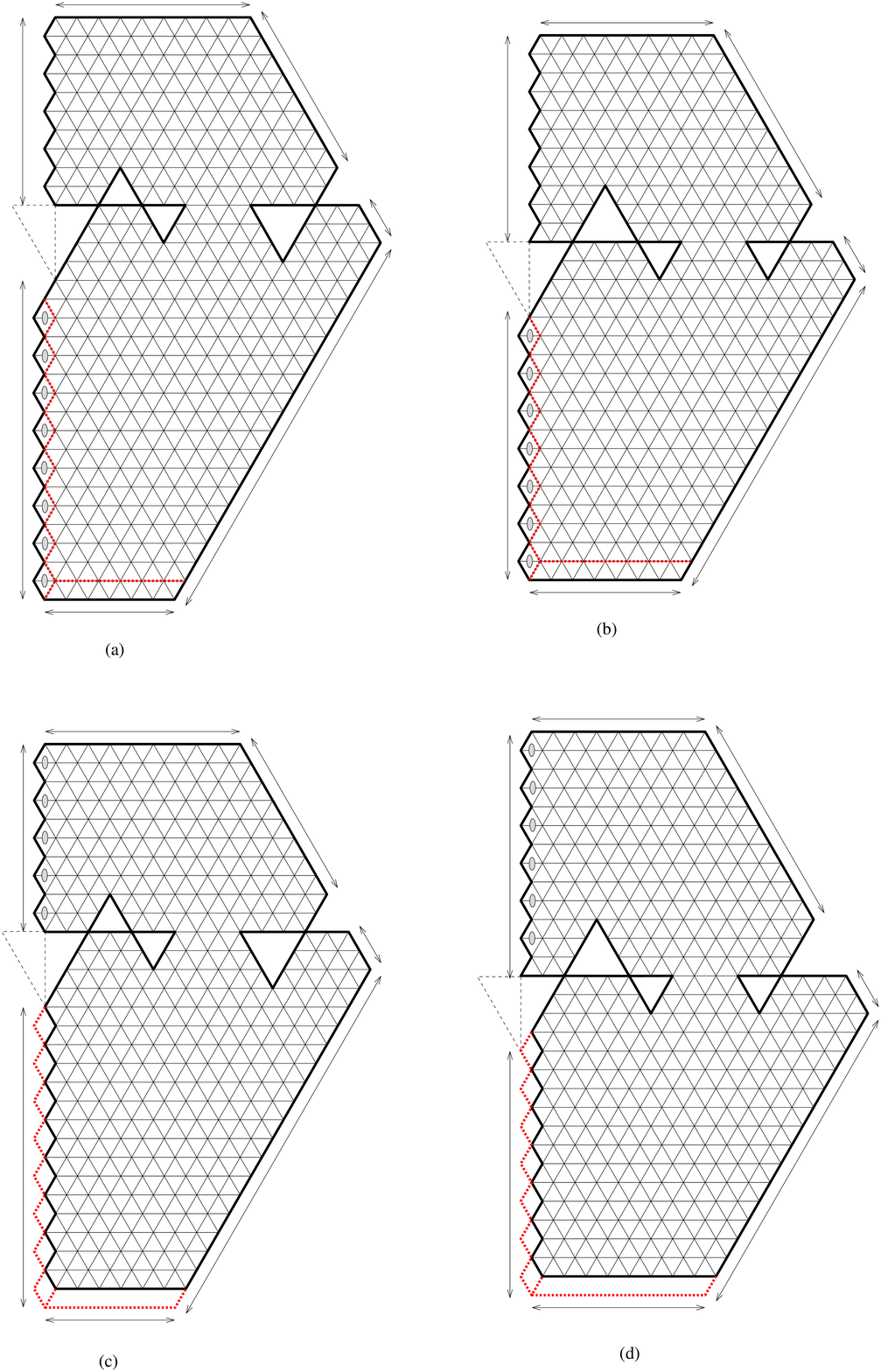}%
\end{picture}%

\begin{picture}(16711,26234)(371,-30281)
\put(15142,-27500){\rotatebox{60.0}{\makebox(0,0)[lb]{\smash{{\SetFigFont{14}{16.8}{\rmdefault}{\mddefault}{\itdefault}{$2y+z+2a_1+2a_3+2b_2$}%
}}}}}
\put(5359,-27412){\rotatebox{60.0}{\makebox(0,0)[lb]{\smash{{\SetFigFont{14}{16.8}{\rmdefault}{\mddefault}{\itdefault}{$2y+z+2a_1+2a_3+2b_2+1$}%
}}}}}
\put(9819,-27446){\rotatebox{90.0}{\makebox(0,0)[lb]{\smash{{\SetFigFont{14}{16.8}{\rmdefault}{\mddefault}{\itdefault}{$y+z+a_3+b_2-\frac{1}{2}$}%
}}}}}
\put(606,-26886){\rotatebox{90.0}{\makebox(0,0)[lb]{\smash{{\SetFigFont{14}{16.8}{\rmdefault}{\mddefault}{\itdefault}{$y+z+a_3+b_2$}%
}}}}}
\put(14648,-18858){\rotatebox{300.0}{\makebox(0,0)[lb]{\smash{{\SetFigFont{14}{16.8}{\rmdefault}{\mddefault}{\itdefault}{$2y+2a_2+b_1-1$}%
}}}}}
\put(5870,-19132){\rotatebox{300.0}{\makebox(0,0)[lb]{\smash{{\SetFigFont{14}{16.8}{\rmdefault}{\mddefault}{\itdefault}{$2y+2a_2+b_1$}%
}}}}}
\put(9759,-21116){\rotatebox{90.0}{\makebox(0,0)[lb]{\smash{{\SetFigFont{14}{16.8}{\rmdefault}{\mddefault}{\itdefault}{$y+a_2+b_1-\frac{1}{2}$}%
}}}}}
\put(606,-20743){\rotatebox{90.0}{\makebox(0,0)[lb]{\smash{{\SetFigFont{14}{16.8}{\rmdefault}{\mddefault}{\itdefault}{$y+a_2+b_1$}%
}}}}}
\put(11244,-17786){\makebox(0,0)[lb]{\smash{{\SetFigFont{14}{16.8}{\rmdefault}{\mddefault}{\itdefault}{$x+a_1+a_3+b_2$}%
}}}}
\put(2316,-18013){\makebox(0,0)[lb]{\smash{{\SetFigFont{14}{16.8}{\rmdefault}{\mddefault}{\itdefault}{$x+a_1+a_3+b_2$}%
}}}}
\put(7491,-22168){\makebox(0,0)[lb]{\smash{{\SetFigFont{14}{16.8}{\rmdefault}{\mddefault}{\itdefault}{$z$}%
}}}}
\put(12564,-23171){\makebox(0,0)[lb]{\smash{{\SetFigFont{14}{16.8}{\rmdefault}{\mddefault}{\itdefault}{$a_3$}%
}}}}
\put(9398,-23645){\makebox(0,0)[lb]{\smash{{\SetFigFont{14}{16.8}{\rmdefault}{\mddefault}{\itdefault}{$2a_1$}%
}}}}
\put(425,-22792){\makebox(0,0)[lb]{\smash{{\SetFigFont{14}{16.8}{\rmdefault}{\mddefault}{\itdefault}{$2a_1$}%
}}}}
\put(5421,-22438){\makebox(0,0)[lb]{\smash{{\SetFigFont{14}{16.8}{\rmdefault}{\mddefault}{\itdefault}{$b_2$}%
}}}}
\put(14634,-23186){\makebox(0,0)[lb]{\smash{{\SetFigFont{14}{16.8}{\rmdefault}{\mddefault}{\itdefault}{$b_2$}%
}}}}
\put(6546,-21778){\makebox(0,0)[lb]{\smash{{\SetFigFont{14}{16.8}{\rmdefault}{\mddefault}{\itdefault}{$b_1$}%
}}}}
\put(15735,-22573){\makebox(0,0)[lb]{\smash{{\SetFigFont{14}{16.8}{\rmdefault}{\mddefault}{\itdefault}{$b_1$}%
}}}}
\put(11554,-29600){\makebox(0,0)[lb]{\smash{{\SetFigFont{14}{16.8}{\rmdefault}{\mddefault}{\itdefault}{$x+a_2+b_1$}%
}}}}
\put(1760,-29716){\makebox(0,0)[lb]{\smash{{\SetFigFont{14}{16.8}{\rmdefault}{\mddefault}{\itdefault}{$x+a_2+b_1$}%
}}}}
\put(9816,-7770){\rotatebox{90.0}{\makebox(0,0)[lb]{\smash{{\SetFigFont{14}{16.8}{\rmdefault}{\mddefault}{\itdefault}{$y+a_2+b_1-\frac{1}{2}$}%
}}}}}
\put(11298,-15952){\makebox(0,0)[lb]{\smash{{\SetFigFont{14}{16.8}{\rmdefault}{\mddefault}{\itdefault}{$x+a_2+b_1$}%
}}}}
\put(9771,-14195){\rotatebox{90.0}{\makebox(0,0)[lb]{\smash{{\SetFigFont{14}{16.8}{\rmdefault}{\mddefault}{\itdefault}{$y+z+a_3+b_2$}%
}}}}}
\put(14656,-14185){\rotatebox{60.0}{\makebox(0,0)[lb]{\smash{{\SetFigFont{14}{16.8}{\rmdefault}{\mddefault}{\itdefault}{$2y+z+2a_1+2a_3+2b_2$}%
}}}}}
\put(9438,-9802){\makebox(0,0)[lb]{\smash{{\SetFigFont{14}{16.8}{\rmdefault}{\mddefault}{\itdefault}{$2a_1$}%
}}}}
\put(12734,-9300){\makebox(0,0)[lb]{\smash{{\SetFigFont{14}{16.8}{\rmdefault}{\mddefault}{\itdefault}{$a_3$}%
}}}}
\put(11632,-8700){\makebox(0,0)[lb]{\smash{{\SetFigFont{14}{16.8}{\rmdefault}{\mddefault}{\itdefault}{$a_2$}%
}}}}
\put(14800,-9307){\makebox(0,0)[lb]{\smash{{\SetFigFont{14}{16.8}{\rmdefault}{\mddefault}{\itdefault}{$b_2$}%
}}}}
\put(15694,-8830){\makebox(0,0)[lb]{\smash{{\SetFigFont{14}{16.8}{\rmdefault}{\mddefault}{\itdefault}{$b_1$}%
}}}}
\put(16646,-9148){\makebox(0,0)[lb]{\smash{{\SetFigFont{14}{16.8}{\rmdefault}{\mddefault}{\itdefault}{$z$}%
}}}}
\put(11448,-4778){\makebox(0,0)[lb]{\smash{{\SetFigFont{14}{16.8}{\rmdefault}{\mddefault}{\itdefault}{$x+a_1+a_3+b_2$}%
}}}}
\put(14723,-5776){\rotatebox{300.0}{\makebox(0,0)[lb]{\smash{{\SetFigFont{14}{16.8}{\rmdefault}{\mddefault}{\itdefault}{$2y+2a_2+b_1-1$}%
}}}}}
\put(1854,-16299){\makebox(0,0)[lb]{\smash{{\SetFigFont{14}{16.8}{\rmdefault}{\mddefault}{\itdefault}{$x+a_2+b_1$}%
}}}}
\put(654,-13546){\rotatebox{90.0}{\makebox(0,0)[lb]{\smash{{\SetFigFont{14}{16.8}{\rmdefault}{\mddefault}{\itdefault}{$y+z+a_3+b_2+\frac{1}{2}$}%
}}}}}
\put(676,-6856){\rotatebox{90.0}{\makebox(0,0)[lb]{\smash{{\SetFigFont{14}{16.8}{\rmdefault}{\mddefault}{\itdefault}{$y+a_2+b_1$}%
}}}}}
\put(5896,-12969){\rotatebox{60.0}{\makebox(0,0)[lb]{\smash{{\SetFigFont{14}{16.8}{\rmdefault}{\mddefault}{\itdefault}{$2y+z+2a_1+2a_3+2b_2+1$}%
}}}}}
\put(7741,-8499){\makebox(0,0)[lb]{\smash{{\SetFigFont{14}{16.8}{\rmdefault}{\mddefault}{\itdefault}{$z$}%
}}}}
\put(5866,-5311){\rotatebox{300.0}{\makebox(0,0)[lb]{\smash{{\SetFigFont{14}{16.8}{\rmdefault}{\mddefault}{\itdefault}{$2y+2a_2+b_1$}%
}}}}}
\put(2701,-4336){\makebox(0,0)[lb]{\smash{{\SetFigFont{14}{16.8}{\rmdefault}{\mddefault}{\itdefault}{$x+a_1+a_3+b_2$}%
}}}}
\put(5649,-8806){\makebox(0,0)[lb]{\smash{{\SetFigFont{14}{16.8}{\rmdefault}{\mddefault}{\itdefault}{$b_2$}%
}}}}
\put(6691,-8139){\makebox(0,0)[lb]{\smash{{\SetFigFont{14}{16.8}{\rmdefault}{\mddefault}{\itdefault}{$b_1$}%
}}}}
\put(3384,-8604){\makebox(0,0)[lb]{\smash{{\SetFigFont{14}{16.8}{\rmdefault}{\mddefault}{\itdefault}{$a_3$}%
}}}}
\put(2574,-8131){\makebox(0,0)[lb]{\smash{{\SetFigFont{14}{16.8}{\rmdefault}{\mddefault}{\itdefault}{$a_2$}%
}}}}
\put(616,-9091){\makebox(0,0)[lb]{\smash{{\SetFigFont{14}{16.8}{\rmdefault}{\mddefault}{\itdefault}{$2a_1$}%
}}}}
\put(2331,-21853){\makebox(0,0)[lb]{\smash{{\SetFigFont{14}{16.8}{\rmdefault}{\mddefault}{\itdefault}{$a_2$}%
}}}}
\put(11514,-22601){\makebox(0,0)[lb]{\smash{{\SetFigFont{14}{16.8}{\rmdefault}{\mddefault}{\itdefault}{$a_2$}%
}}}}
\put(16914,-23006){\makebox(0,0)[lb]{\smash{{\SetFigFont{14}{16.8}{\rmdefault}{\mddefault}{\itdefault}{$z$}%
}}}}
\put(3156,-22318){\makebox(0,0)[lb]{\smash{{\SetFigFont{14}{16.8}{\rmdefault}{\mddefault}{\itdefault}{$a_3$}%
}}}}
\end{picture}%
}
\caption{The four mixed-boundary regions:
(a) The region $NR^{(1)}_{2,1,2}(2,2,2;\ 2,3)$. (b) The region $NR^{(2)}_{2,1,2}(2,3,2;\ 2,2)$.
 (c) The region $NR^{(3)}_{2,1,2}(2,2,2;\ 2,3)$. (d) The region $NR^{(4)}_{2,1,2}(2,3,2;\ 3,2)$.}\label{fig:halvedhex4}
\end{figure}

The final quadruple of halved hexagons considered in this section   are  reversing versions of the above mixed-boundary regions.
The region $NR^{(1)}_{x,y,z}(\textbf{a};\  \textbf{b})$ is obtained from the region $R^{(1)}_{x,y,z}(\textbf{a};\ \textbf{b})$
by adding a layer of unit triangles running along the south side and a layer along the portion below the two ferns of the west side.
 We next assign a weight $1/2$ to
each newly added vertical lozenges along the west side of the region (see Figure \ref{fig:halvedhex4}(a)).
The region $NR^{(2)}_{x,y,z}(\textbf{a}; \textbf{b})$ is obtained similarly from  the region
$R^{(2)}_{x,y,z}(\textbf{a}; \textbf{b})$ as shown in Figure \ref{fig:halvedhex4}(b).
 If we \emph{remove} the unit triangles running along the south side and the portion below the ferns of the west side from the regions
 $RW^{(1)}_{x,y,z}(\textbf{a}; \textbf{b})$ and $RW^{(2)}_{x,y,z}(\textbf{a}; \textbf{b})$,
 we get respectively the new regions $NR^{(3)}_{x,y,z}(\textbf{a}; \textbf{b})$ and
 $NR^{(4)}_{x,y,z}(\textbf{a}; \textbf{b})$ (see examples in Figures \ref{fig:halvedhex4}(c) and (d), respectively).
 The tilings of these four new regions are also enumerated by closed-form product formulas.

\begin{thm}\label{mainMR1} Assume that $x,y,z$ are non-negative integers and that $\textbf{a}=(a_1,a_2,\dotsc,a_m)$ and $\textbf{b}=(b_1,b_2,\dotsc,b_n)$ are two sequences of non-negative integers. Then
\begin{align}\label{mainMR1eq}
\M(NR^{(1)}_{x,y,z}(\textbf{a};\textbf{b}))&=\frac{\M(NR^{(1)}_{x+y,0,z}(\textbf{a};\textbf{b}))
\M(NR^{(1)}_{0,2y,z}(\textbf{a};\textbf{b}))}{\M(NR^{(1)}_{y,0,z}(\textbf{a};\textbf{b}))}\notag\\
&\quad\times \frac{\T(x+1,2a+b+2y+z,y)\T(x+a+1,b+2y+z,y)}{\T(1,2a+b+2y+z,y)\T(a+1,\Su_b+2y+z,y)}\notag\\
&=2^{a_1-y}\Q(a_1,\dotsc,a_{2\lceil \frac{m-1}{2}\rceil}, a_{2\lceil \frac{m-1}{2}\rceil+1}+x+y+b_{2\lfloor\frac{n+1}{2}\rfloor},b_{2\lfloor\frac{n+1}{2}\rfloor-1},\dotsc,b_1)\notag\\
&\quad\times \K'(0,a_1+1,a_2,\dotsc,a_{2\lfloor\frac{m+1}{2}\rfloor-1},a_{2\lfloor\frac{m+1}{2}\rfloor}+x+y+b_{2\lceil \frac{n-1}{2}\rceil+1},b_{2\lceil \frac{n-1}{2}\rceil},\dotsc, b_1,z) \notag\\
&\quad\times\frac{(a+y)!}{a!} \frac{\Hf_2(2\e_a+2\od_b+1)\Hf_2(2\od_a+2\e_b+2z+2)}{\Hf_2(2\e_a+2\od_b+2y+1)
\Hf_2(2\od_a+2\e_b+2y+2z+2)} \notag\\
&\quad\times \frac{\Hf(2a+b+2y+z+1)\Hf(b+y+z)}{\Hf(2a+b+y+z+1)\Hf(b+z)}\notag\\
&\quad\times   \frac{\T(x+1,2a+b+2y+z,y)\T(x+a+1,b+2y+z,y)}{\T(1,2a+b+2y+z,y)\T(a+1,b+2y+z,y)}.
\end{align}
\end{thm}

\begin{thm}\label{mainMR2} With the same notations in Theorem \ref{mainMR1}, the number of tilings of the $NR^{(2)}$-type region is given by
\begin{align}\label{mainMR2eq}
\M(NR^{(2)}_{x,y,z}(\textbf{a};\textbf{b}))&=\frac{\M(NR^{(2)}_{x+y,0,z}(\textbf{a};\textbf{b}))
\M(NR^{(2)}_{0,2y,z}(\textbf{a};\textbf{b}))}{\M(NR^{(2)}_{y,0,z}(\textbf{a};\textbf{b}))}\notag\\
&\quad\times \frac{\T(x+1,2a+b+2y+z-1,y)\T(x+a+1,b+2y+z-1,y)}{\T(1,2a+b+2y+z-1,y)\T(a+1,b+2y+z-1,y)}\\
&=2^{a_1-y}\K(a_1,\dotsc,a_{2\lceil \frac{m-1}{2}\rceil}, a_{2\lceil \frac{m-1}{2}\rceil+1}+x+y+b_{2\lfloor\frac{n+1}{2}\rfloor},b_{2\lfloor\frac{n+1}{2}\rfloor-1},\dotsc,b_1)\notag\\
&\quad\times \Q'(0,a_1,\dotsc,a_{2\lfloor\frac{m+1}{2}\rfloor-1},a_{2\lfloor\frac{m+1}{2}\rfloor}+x+y+b_{2\lceil \frac{n-1}{2}\rceil+1},b_{2\lceil \frac{n-1}{2}\rceil},\dotsc, b_1,z) \notag\\
&\quad\times \frac{\Hf_2(2\e_a+2\od_b)\Hf_2(2\od_a+2\e_b+2z+1)}{\Hf_2(2\e_a+2\od_b+2y)
\Hf_2(2\od_a+2\e_b+2y+2z+1)} \notag\\
&\quad\times \frac{\Hf(2a+b+2y+z)\Hf(b+y+z)}{\Hf(2a+b+y+z)\Hf(b+z)}\notag\\
&\quad\times   \frac{\T(x+1,2a+b+2y+z-1,y)\T(x+a+1,b+2y+z-1,y)}{\T(1,2a+b+2y+z-1,y)\T(a+1,b+2y+z-1,y)}.
\end{align}
\end{thm}

\begin{thm}\label{mainMR3} With the same notations in Theorem \ref{mainMR1}, the number of tilings of the $NR^{(3)}$-type region is given by
\begin{align}\label{mainMR3eq}
\M(MR^{(3)}_{x,y,z}(\textbf{a};\textbf{b}))&=\frac{\M(NR^{(3)}_{x+y,0,z}(\textbf{a};\textbf{b}))
\M(NR^{(3)}_{0,2y,z}(\textbf{a};\textbf{b}))}{\M(NR^{(3)}_{y,0,z}(\textbf{a};\textbf{b}))}\notag\\
&\quad\times \frac{\T(x+1,2a+b+2y+z-1,y)\T(x+a+1,b+2y+z-1,y)}{\T(1,2a+b+2y+z-1,y)\T(a+1,b+2y+z-1,y)}\notag\\
&=2^{-y}\Q'(a_1,\dotsc,a_{2\lceil \frac{m-1}{2}\rceil}, a_{2\lceil \frac{m-1}{2}\rceil+1}+x+y+b_{2\lfloor\frac{n+1}{2}\rfloor},b_{2\lfloor\frac{n+1}{2}\rfloor-1},\dotsc,b_1)\notag\\
&\quad\times \K(0,a_1,\dotsc,a_{2\lfloor\frac{m+1}{2}\rfloor-1},a_{2\lfloor\frac{m+1}{2}\rfloor}+x+y+b_{2\lceil \frac{n-1}{2}\rceil+1},b_{2\lceil \frac{n-1}{2}\rceil},\dotsc, b_1,z) \notag\\
&\quad\times \frac{\Hf_2(2\e_a+2\od_b+1)\Hf_2(2\od_a+2\e_b+2z)}{\Hf_2(2\e_a+2\od_b+2y+1)
\Hf_2(2\od_a+2\e_b+2y+2z)} \notag\\
&\quad\times \frac{\Hf(2a+b+2y+z)\Hf(b+y+z)}{\Hf(2a+b+y+z)\Hf(b+z)}\notag\\
&\quad\times   \frac{\T(x+1,2a+b+2y+z-1,y)\T(x+a+1,b+2y+z-1,y)}{\T(1,2a+b+2y+z-1,y)\T(a+1,b+2y+z-1,y)}.
\end{align}
\end{thm}

\begin{thm}\label{mainMR4} With the same notations in Theorem \ref{mainMR1}, the number of tilings of the $NR^{(4)}$-type region is given by
\begin{align}\label{mainMR4eq}
\M(NR^{(4)}_{x,y,z}(\textbf{a};\textbf{b}))&=\frac{\M(NR^{(4)}_{x+y,0,z}(\textbf{a};\textbf{b}))
\M(NR^{(4)}_{0,2y,z}(\textbf{a};\textbf{b}))}{\M(MR^{(4)}_{y,0,z}(\textbf{a};\textbf{b}))}\notag\\
&\quad\times \frac{\T(x+1,2a+b+2y+z-2,y)\T(x+a+1,b+2y+z-2,y)}{\T(1,2a+b+2y+z-2,y)\T(a+1,b+2y+z-2,y)}\notag\\
&=2^{-y}\K'(a_1,\dotsc,a_{2\lceil \frac{m-1}{2}\rceil}, a_{2\lceil \frac{m-1}{2}\rceil+1}+x+y+b_{2\lfloor\frac{n+1}{2}\rfloor},b_{2\lfloor\frac{n+1}{2}\rfloor-1},\dotsc,b_1)\notag\\
&\quad\times \Q(0,a_1-1,a_2,\dotsc,a_{2\lfloor\frac{m+1}{2}\rfloor-1},a_{2\lfloor\frac{m+1}{2}\rfloor}+x+y+b_{2\lceil \frac{n-1}{2}\rceil+1},b_{2\lceil \frac{n-1}{2}\rceil},\dotsc, b_1,z) \notag\\
&\quad\times \frac{(a-1)!}{(a+y-1)!} \frac{\Hf_2(2\e_a+2\od_b)\Hf_2(2\od_a+2\e_b+2z-1)}{\Hf_2(2\e_a+2\od_b+2y)
\Hf_2(2\od_a+2\e_b+2y+2z-1)} \notag\\
&\quad\times \frac{\Hf(2a+b+2y+z-1)\Hf(b+y+z)}{\Hf(2a+b+y+z-1)\Hf(b+z)}\notag\\
&\quad\times   \frac{\T(x+1,2a+b+2y+z-2,y)\T(x+a+1,b+2y+z-2,y)}{\T(1,2a+b+2y+z-2,y)\T(a+1,b+2y+z-2,y)}.
\end{align}
\end{thm}

We conclude this section by giving exact tiling formulas for the symmetric hexagons in which three aligned ferns have been removed
 (denoted by $S^{(i)}_{x,y,z} (\textbf{a}; \textbf{b})$, for $i=1,2$; illustrated in Figure \ref{fig:threearray}).
We have a small notice that 
the $x$- and $y$-parameters ofn the region $S^{(i)}_{x,y,z} (\textbf{a}; \textbf{b})$ always have the same parity.
\begin{thm}\label{mainthm1}
Assume that $x,y,z$ are non-negative integers and that
$\textbf{a}=(a_1,\dotsc,a_m)$ and $\textbf{b}=(b_1,\dotsc,b_n)$ are two sequences of non-negative integers as usual.
The number of lozenge tilings of the symmetric hexagon with three ferns removed
 $S^{(1)}_{x,y,z} (\textbf{a}; \textbf{b})$ is always given by a simple product formula as follows.

(a) If $a_1$ is even, then
\begin{align}\label{eqm1}
\M(S^{(1)}_{x,y,z} (\textbf{a}; \textbf{b}))=2^{y+z+a+b-a_1}&\M\left(H^{(2)}_{\left\lfloor\frac{x}{2}\right\rfloor,\left\lceil\frac{y}{2}\right\rceil,z}\left(\frac{a_1}{2},a_2,\dotsc,a_m; \textbf{b}\right)\right)\notag\\
&\times\M\left(W^{(1)}_{\left\lceil\frac{x}{2}\right\rceil,\left\lfloor\frac{y}{2}\right\rfloor,z}\left(\frac{a_1}{2},a_2,\dotsc,a_m; \textbf{b}\right)\right).
\end{align}
(d) If $a_1$ is odd, then
\begin{align}\label{eqm4}
\M(S^{(1)}_{x,y,z} (\textbf{a}; \textbf{b}))=2^{y+z+a+b-a_1}&\M\left(N^{(4)}_{\left\lfloor\frac{x}{2}\right\rfloor,\left\lceil\frac{y}{2}\right\rceil,z}\left(\frac{a_1+1}{2},a_2,\dotsc,a_m; \textbf{b}\right)\right)\notag\\
&\times\M\left(N^{(1)}_{\left\lceil\frac{x}{2}\right\rceil,\left\lfloor\frac{y}{2}\right\rfloor,z}\left(\frac{a_1-1}{2},a_2,\dotsc,a_m; \textbf{b}\right)\right).
\end{align}
\end{thm}

\begin{thm}\label{mainthm2}
Assume that $x,y,z$ are non-negative integers and that $\textbf{a}=(a_1,\dotsc,a_m)$ and $\textbf{b}=(b_1,\dotsc,b_n)$ are two sequences of non-negative integers. The number of lozenge tilings of the symmetric hexagon with three ferns removed $S^{(2)}_{x,y,z} (\textbf{a}; \textbf{b})$ is always given by a simple product formula as follows.

(a) If $a_1$ is even, then
\begin{align}\label{eqm1b}
\M(S^{(2)}_{x,y,z} (\textbf{a}; \textbf{b}))=2^{y+z+a+b-a_1}&\M\left(R^{(2)}_{\left\lfloor\frac{x}{2}\right\rfloor,\left\lceil\frac{y}{2}\right\rceil,z}\left(\frac{a_1}{2},a_2,\dotsc,a_m; \textbf{b}\right)\right)\notag\\
&\times\M\left(RW^{(1)}_{\left\lceil\frac{x}{2}\right\rceil,\left\lfloor\frac{y}{2}\right\rfloor,z}\left(\frac{a_1}{2},a_2,\dotsc,a_m; \textbf{b}\right)\right).
\end{align}
(b) If $a_1$ is odd, then
\begin{align}\label{eqm3b}
\M(S^{(2)}_{x,y,z} (\textbf{a}; \textbf{b}))=2^{y+z+a+b-a_1}&\M\left(NR^{(1)}_{\left\lceil\frac{x}{2}\right\rceil,\left\lfloor\frac{y}{2}\right\rfloor,z}\left(\frac{a_1-1}{2},a_2,\dotsc,a_m; \textbf{b}\right)\right)\notag\\
&\times\M\left(NR^{(4)}_{\left\lfloor\frac{x}{2}\right\rfloor,\left\lceil\frac{y}{2}\right\rceil,z}\left(\frac{a_1+1}{2},a_2,\dotsc,a_m; \textbf{b}\right)\right).
\end{align}
\end{thm}

\section{Preliminaries}\label{sec:Prelim}

Let $G$ be a finite graph with no loop, however, multiple edges are allowed. A \emph{perfect matching} of $G$ (or simply \emph{matching} in this paper) is a
collection of vertex-disjoint edges that covers all vertices of the graph. Lozenge tilings of a region on the triangular
 lattice can be identified with matchings of its \emph{(planar) dual graph}
 (the graph whose vertices are the unit triangles of the regions and whose edges connect precisely two unit triangles sharing an edge).
  In this point of view, we let $\M(G)$ denote the sum of weights of all matchings of $G$, where the \emph{weight} of a matching is the product
  of weights of its constituent edges. In the unweighted case, $\M(G)$ counts the matchings of the graph $G$.

A \emph{forced lozenge} is a lozenge contained in any tilings of the region. By removing a forced lozenge,
the weighted tiling number of a region is reduced by a factor equal to the weight of the removed lozenge. More generally, we have the following lemma that first appeared in \cite{Lai1q, Lai2q}.

\begin{lem}[Region-splitting Lemma]\label{RS}
Let $R$ be a \emph{balanced} region on the triangular lattice (i.e. $R$ has the same number of up-pointing and down-pointing unit triangles). Assume that $S$ is a subregion of $R$ satisfying the following conditions:

(a) There is exactly one type of unit triangles (up-pointing or down-pointing) running along each side of the border separating $S$ and its complement $R-S$.

(b) $S$ is balanced.

Then $\M(R)= \M(S) \cdot \M(R-S)$.
\end{lem}

One of the main ingredients of our proofs is the following powerful theorem by Kuo \cite{Kuo}, that is usually mentioned as \emph{Kuo condensation}.
\begin{thm}[Theorem 5.1 in \cite{Kuo}]\label{kuothm}
Assume that $G=(V_1,V_2,E)$ is a weighted bipartite planar graph with the two vertex classes $V_1$ and $V_2$ of the same cardinality. Assume in addition that $u,v,w,s$ are four vertices appearing on a cyclic order on a face of $G$, such that $u,w \in V_1$ and $v,s\in V_2$. Then
\begin{align}\label{Kuoeq}
\M(G)\M(G-\{u,v,w,s\})=\M(G-\{u,v\})\M(G-\{w,s\})+\M(G-\{u,s\})\M(G-\{v,w\}).
\end{align}
\end{thm}

Next, we quote here a factorization theorem by Ciucu (Theorem 1.2 in \cite{Ciucu3}), that allows us write the number of matchings of  a symmetric graph as the product of the matching numbers of two disjoint subgraphs.

\begin{lem}[Ciucu's Factorization Theorem]\label{ciucufactor}
Let $G=(V_1,V_2,E)$ be a weighted bipartite planar graph with a vertical symmetry axis $\ell$. Assume that $a_1,b_1,a_2,b_2,\dots,a_k,b_k$ are all the vertices of $G$ on $\ell$ appearing in this order from top to bottom\footnote{It is easy to see that if $G$ admits a perfect matching, then $G$ has an even number of vertices on $\ell$.}. Assume in addition that the vertices of $G$ on $\ell$ form a cut set of $G$ (i.e. the removal of those vertices separates $G$ into two vertex-disjoint graphs). We reduce the weights of all edges of $G$ lying on $\ell$ by half and keep the other edge-weights unchanged. Next, we color the two vertex classes $V_1$ and $V_2$ of $G$ by black and white, without loss of generality, assume that $a_1$ is black. Finally, we remove all edges on the left of $\ell$ which are adjacent to a black $a_i$ or a white $b_j$; we also remove the edges  on the right of $\ell$ which are adjacent to a white $a_i$ or a black $b_j$. This way, $G$ is divided into two disjoint weighted graphs $G^+$ and $G^-$ (on the left and on the right of $\ell$, respectively).  See Figure \ref{Figurefactor} for an example.  Then
\begin{equation}
\M(G)=2^{k}\M(G^+)\M(G^-).
\end{equation}
\end{lem}

\begin{figure}\centering
\setlength{\unitlength}{3947sp}%
\begingroup\makeatletter\ifx\SetFigFont\undefined%
\gdef\SetFigFont#1#2#3#4#5{%
  \reset@font\fontsize{#1}{#2pt}%
  \fontfamily{#3}\fontseries{#4}\fontshape{#5}%
  \selectfont}%
\fi\endgroup%
\resizebox{12cm}{!}{
\begin{picture}(0,0)%
\includegraphics{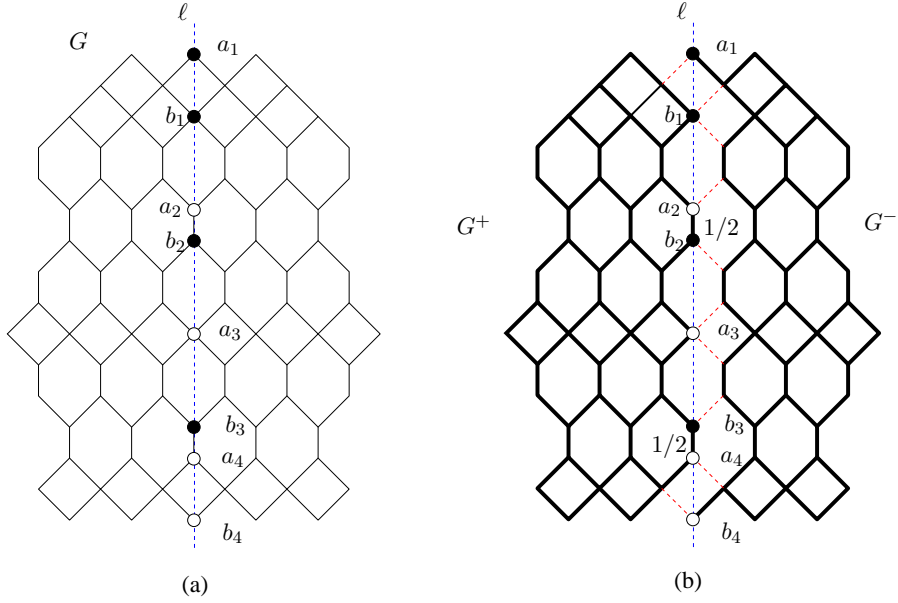}%
\end{picture}%
%
%

\begin{picture}(6803,4626)(1415,-4264)
\put(2042,-61){\makebox(0,0)[rb]{\smash{{\SetFigFont{11}{13.2}{\familydefault}{\mddefault}{\updefault}{$G$}%
}}}}
\put(6992,-1478){\makebox(0,0)[rb]{\smash{{\SetFigFont{11}{13.2}{\familydefault}{\mddefault}{\updefault}{$1/2$}%
}}}}
\put(6593,162){\makebox(0,0)[rb]{\smash{{\SetFigFont{11}{13.2}{\familydefault}{\mddefault}{\updefault}{$\ell$}%
}}}}
\put(6992,-75){\makebox(0,0)[rb]{\smash{{\SetFigFont{11}{13.2}{\familydefault}{\mddefault}{\updefault}{$a_1$}%
}}}}
\put(6581,-618){\makebox(0,0)[rb]{\smash{{\SetFigFont{11}{13.2}{\familydefault}{\mddefault}{\updefault}{$b_1$}%
}}}}
\put(6551,-1309){\makebox(0,0)[rb]{\smash{{\SetFigFont{11}{13.2}{\familydefault}{\mddefault}{\updefault}{$a_2$}%
}}}}
\put(6581,-1571){\makebox(0,0)[rb]{\smash{{\SetFigFont{11}{13.2}{\familydefault}{\mddefault}{\updefault}{$b_2$}%
}}}}
\put(7004,-2241){\makebox(0,0)[rb]{\smash{{\SetFigFont{11}{13.2}{\familydefault}{\mddefault}{\updefault}{$a_3$}%
}}}}
\put(7038,-2970){\makebox(0,0)[rb]{\smash{{\SetFigFont{11}{13.2}{\familydefault}{\mddefault}{\updefault}{$b_3$}%
}}}}
\put(7025,-3207){\makebox(0,0)[rb]{\smash{{\SetFigFont{11}{13.2}{\familydefault}{\mddefault}{\updefault}{$a_4$}%
}}}}
\put(7013,-3792){\makebox(0,0)[rb]{\smash{{\SetFigFont{11}{13.2}{\familydefault}{\mddefault}{\updefault}{$b_4$}%
}}}}
\put(5106,-1474){\makebox(0,0)[rb]{\smash{{\SetFigFont{11}{13.2}{\familydefault}{\mddefault}{\updefault}{$G^+$}%
}}}}
\put(8203,-1440){\makebox(0,0)[rb]{\smash{{\SetFigFont{11}{13.2}{\familydefault}{\mddefault}{\updefault}{$G^-$}%
}}}}
\put(6597,-3131){\makebox(0,0)[rb]{\smash{{\SetFigFont{11}{13.2}{\familydefault}{\mddefault}{\updefault}{$1/2$}%
}}}}
\put(2799,156){\makebox(0,0)[rb]{\smash{{\SetFigFont{11}{13.2}{\familydefault}{\mddefault}{\updefault}{$\ell$}%
}}}}
\put(3198,-81){\makebox(0,0)[rb]{\smash{{\SetFigFont{11}{13.2}{\familydefault}{\mddefault}{\updefault}{$a_1$}%
}}}}
\put(2787,-624){\makebox(0,0)[rb]{\smash{{\SetFigFont{11}{13.2}{\familydefault}{\mddefault}{\updefault}{$b_1$}%
}}}}
\put(2757,-1315){\makebox(0,0)[rb]{\smash{{\SetFigFont{11}{13.2}{\familydefault}{\mddefault}{\updefault}{$a_2$}%
}}}}
\put(2787,-1577){\makebox(0,0)[rb]{\smash{{\SetFigFont{11}{13.2}{\familydefault}{\mddefault}{\updefault}{$b_2$}%
}}}}
\put(3210,-2247){\makebox(0,0)[rb]{\smash{{\SetFigFont{11}{13.2}{\familydefault}{\mddefault}{\updefault}{$a_3$}%
}}}}
\put(3244,-2976){\makebox(0,0)[rb]{\smash{{\SetFigFont{11}{13.2}{\familydefault}{\mddefault}{\updefault}{$b_3$}%
}}}}
\put(3231,-3213){\makebox(0,0)[rb]{\smash{{\SetFigFont{11}{13.2}{\familydefault}{\mddefault}{\updefault}{$a_4$}%
}}}}
\put(3219,-3798){\makebox(0,0)[rb]{\smash{{\SetFigFont{11}{13.2}{\familydefault}{\mddefault}{\updefault}{$b_4$}%
}}}}
\end{picture}}
\caption{Ciucu's Factorization Theorem. The edges cut off are illustrated by dotted edges.}\label{Figurefactor}
\end{figure}

\medskip

We have the several identities related to the  products  $\T$ and $\V$ as follows:
\begin{lem}\label{TV}
\begin{align}
\frac{\T(x,n,m)}{\T(x-1,n,m)}=\frac{(x+n-m)_m}{(x-1)_m},
\end{align}
\begin{equation}
\frac{\T(x,n,m)}{\T(x+1,n-2,m-1)}=(x)_{n},
\end{equation}
\begin{equation}
\frac{\V(x,n,m)}{\V(x-2,n,m))}=\frac{[x+2n-2m]_{m}}{[x-2]_m}.
\end{equation}
\begin{equation}
\frac{\V(x,n,m)}{\V(x+2,n-2,m-1)}=[x]_{n},
\end{equation}
\end{lem}

We also have the following immediate  consequence of  Lemma \ref{QAR}:
\begin{lem}\label{QK}
For any sequence $\textbf{t}=(t_1,t_2,\dotsc,t_{2l})$, we have
\begin{align}
\frac{\Q(t_1,\dots,t_{2l}+1)}{\Q(t_1,\dots,t_{2l})}=&\frac{(\s_{2l}(\textbf{t})+1)(2\s_{2l}(\textbf{t})+1)!}{(2\e_t+1)!}\notag\\
&\times\prod_{i=1}^{l}\frac{(\s_{2l}(\textbf{t})-\s_{2i-1}(\textbf{t}))!}{(\s_{2l}(\textbf{t})+\s_{2i-1}(\textbf{t})+1)!}\prod_{i=1}^{l-1}\frac{(\s_{2l}(\textbf{t})+\s_{2i}(\textbf{t})+1)!}{(\s_{2l}(\textbf{t})-\s_{2i}(\textbf{t}))!}
\end{align}
and
\begin{align}
\frac{\K'(t_1,\dots,t_{2l}+1)}{\K'(t_1,\dots,t_{2l})}=&\frac{(2\s_{2l}(\textbf{t})-1)!}{(2\e_t)!}\notag\\
&\times\prod_{i=1}^{l}\frac{(\s_{2l}(\textbf{t})-\s_{2i-1}(\textbf{t}))!}{(\s_{2l}(\textbf{t})+\s_{2i-1}(\textbf{t})-1)!}\prod_{i=1}^{l-1}\frac{(\s_{2l}(\textbf{t})+\s_{2i}(\textbf{t})-1)!}{(\s_{2l}(\textbf{t})-\s_{2i}(\textbf{t}))!}.
\end{align}
\end{lem}

\section{Proofs of the main results}\label{sec:Proof}

\begin{figure}\centering
\setlength{\unitlength}{3947sp}%
\begingroup\makeatletter\ifx\SetFigFont\undefined%
\gdef\SetFigFont#1#2#3#4#5{%
  \reset@font\fontsize{#1}{#2pt}%
  \fontfamily{#3}\fontseries{#4}\fontshape{#5}%
  \selectfont}%
\fi\endgroup%
\resizebox{12cm}{!}{
\begin{picture}(0,0)%
\includegraphics{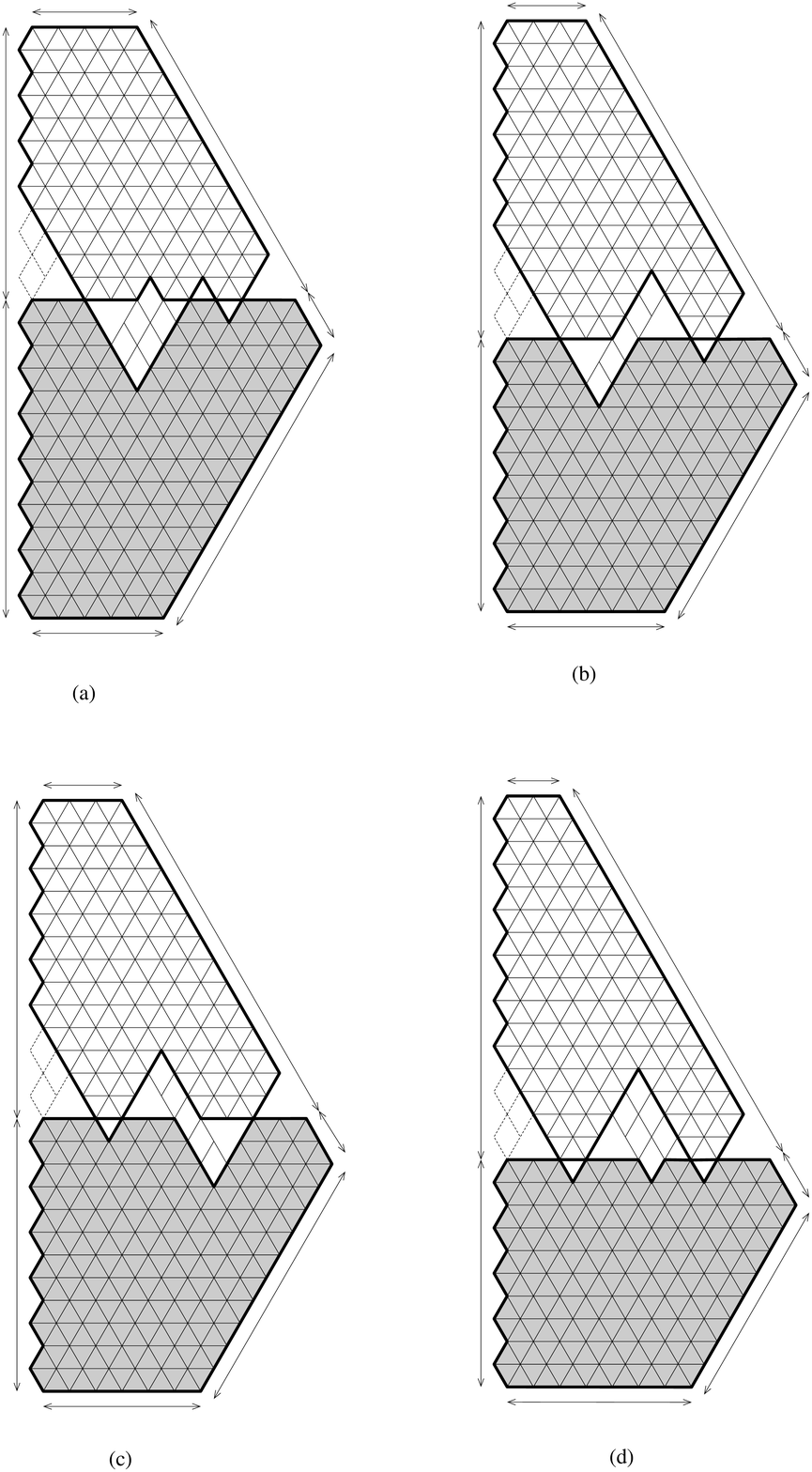}%
\end{picture}%
%
%

\begin{picture}(13184,23204)(946,-23801)
\put(12181,-18886){\makebox(0,0)[lb]{\smash{{\SetFigFont{14}{16.8}{\rmdefault}{\mddefault}{\itdefault}{$b_2$}%
}}}}
\put(5656,-18121){\makebox(0,0)[lb]{\smash{{\SetFigFont{14}{16.8}{\rmdefault}{\mddefault}{\itdefault}{$b_1$}%
}}}}
\put(12151,-6076){\makebox(0,0)[lb]{\smash{{\SetFigFont{14}{16.8}{\rmdefault}{\mddefault}{\itdefault}{$b_2$}%
}}}}
\put(11551,-6016){\makebox(0,0)[lb]{\smash{{\SetFigFont{14}{16.8}{\rmdefault}{\mddefault}{\itdefault}{$b_3$}%
}}}}
\put(3946,-5476){\makebox(0,0)[lb]{\smash{{\SetFigFont{14}{16.8}{\rmdefault}{\mddefault}{\itdefault}{$b_4$}%
}}}}
\put(4336,-5761){\makebox(0,0)[lb]{\smash{{\SetFigFont{14}{16.8}{\rmdefault}{\mddefault}{\itdefault}{$b_3$}%
}}}}
\put(4786,-5476){\makebox(0,0)[lb]{\smash{{\SetFigFont{14}{16.8}{\rmdefault}{\mddefault}{\itdefault}{$b_2$}%
}}}}
\put(12781,-6001){\makebox(0,0)[lb]{\smash{{\SetFigFont{14}{16.8}{\rmdefault}{\mddefault}{\itdefault}{$b_1$}%
}}}}
\put(5446,-5401){\makebox(0,0)[lb]{\smash{{\SetFigFont{14}{16.8}{\rmdefault}{\mddefault}{\itdefault}{$b_1$}%
}}}}
\put(10711,-18811){\makebox(0,0)[lb]{\smash{{\SetFigFont{14}{16.8}{\rmdefault}{\mddefault}{\itdefault}{$a_3$}%
}}}}
\put(3481,-18166){\makebox(0,0)[lb]{\smash{{\SetFigFont{14}{16.8}{\rmdefault}{\mddefault}{\itdefault}{$a_3$}%
}}}}
\put(10111,-18871){\makebox(0,0)[lb]{\smash{{\SetFigFont{14}{16.8}{\rmdefault}{\mddefault}{\itdefault}{$a_2$}%
}}}}
\put(2896,-18256){\makebox(0,0)[lb]{\smash{{\SetFigFont{14}{16.8}{\rmdefault}{\mddefault}{\itdefault}{$a_2$}%
}}}}
\put(10321,-6481){\makebox(0,0)[lb]{\smash{{\SetFigFont{14}{16.8}{\rmdefault}{\mddefault}{\itdefault}{$a_2$}%
}}}}
\put(2941,-5881){\makebox(0,0)[lb]{\smash{{\SetFigFont{14}{16.8}{\rmdefault}{\mddefault}{\itdefault}{$a_2$}%
}}}}
\put(9511,-18811){\makebox(0,0)[lb]{\smash{{\SetFigFont{14}{16.8}{\rmdefault}{\mddefault}{\itdefault}{$a_1$}%
}}}}
\put(2191,-18181){\makebox(0,0)[lb]{\smash{{\SetFigFont{14}{16.8}{\rmdefault}{\mddefault}{\itdefault}{$a_1$}%
}}}}
\put(9481,-6001){\makebox(0,0)[lb]{\smash{{\SetFigFont{14}{16.8}{\rmdefault}{\mddefault}{\itdefault}{$a_1$}%
}}}}
\put(2101,-5401){\makebox(0,0)[lb]{\smash{{\SetFigFont{14}{16.8}{\rmdefault}{\mddefault}{\itdefault}{$a_1$}%
}}}}
\put(9920,-23188){\makebox(0,0)[lb]{\smash{{\SetFigFont{14}{16.8}{\rmdefault}{\mddefault}{\itdefault}{$\od_a+\od_b$}%
}}}}
\put(2486,-23255){\makebox(0,0)[lb]{\smash{{\SetFigFont{14}{16.8}{\rmdefault}{\mddefault}{\itdefault}{$\od_a+\od_b$}%
}}}}
\put(8556,-17145){\rotatebox{90.0}{\makebox(0,0)[lb]{\smash{{\SetFigFont{14}{16.8}{\rmdefault}{\mddefault}{\itdefault}{$y+\od_a+\od_b$}%
}}}}}
\put(1327,-16622){\rotatebox{90.0}{\makebox(0,0)[lb]{\smash{{\SetFigFont{14}{16.8}{\rmdefault}{\mddefault}{\itdefault}{$y+\od_a+\od_b$}%
}}}}}
\put(1327,-21349){\rotatebox{90.0}{\makebox(0,0)[lb]{\smash{{\SetFigFont{14}{16.8}{\rmdefault}{\mddefault}{\itdefault}{$y+z+\e_a+\e_b$}%
}}}}}
\put(8556,-21517){\rotatebox{90.0}{\makebox(0,0)[lb]{\smash{{\SetFigFont{14}{16.8}{\rmdefault}{\mddefault}{\itdefault}{$y+z+\e_a+\e_b$}%
}}}}}
\put(8557,-8593){\rotatebox{90.0}{\makebox(0,0)[lb]{\smash{{\SetFigFont{14}{16.8}{\rmdefault}{\mddefault}{\itdefault}{$y+z+\e_a+\e_b$}%
}}}}}
\put(9716,-10972){\makebox(0,0)[lb]{\smash{{\SetFigFont{14}{16.8}{\rmdefault}{\mddefault}{\itdefault}{$\od_a+\od_b$}%
}}}}
\put(8556,-4350){\rotatebox{90.0}{\makebox(0,0)[lb]{\smash{{\SetFigFont{14}{16.8}{\rmdefault}{\mddefault}{\itdefault}{$y+\od_a+\od_b$}%
}}}}}
\put(13729,-19179){\makebox(0,0)[lb]{\smash{{\SetFigFont{14}{16.8}{\rmdefault}{\mddefault}{\itdefault}{$z$}%
}}}}
\put(6602,-18479){\makebox(0,0)[lb]{\smash{{\SetFigFont{14}{16.8}{\rmdefault}{\mddefault}{\itdefault}{$z$}%
}}}}
\put(13729,-6266){\makebox(0,0)[lb]{\smash{{\SetFigFont{14}{16.8}{\rmdefault}{\mddefault}{\itdefault}{$z$}%
}}}}
\put(5374,-21993){\rotatebox{60.0}{\makebox(0,0)[lb]{\smash{{\SetFigFont{14}{16.8}{\rmdefault}{\mddefault}{\itdefault}{$2y+z+2\e_a+2\e_b$}%
}}}}}
\put(12706,-22339){\rotatebox{60.0}{\makebox(0,0)[lb]{\smash{{\SetFigFont{14}{16.8}{\rmdefault}{\mddefault}{\itdefault}{$2y+z+2\e_a+2\e_b$}%
}}}}}
\put(12706,-9721){\rotatebox{60.0}{\makebox(0,0)[lb]{\smash{{\SetFigFont{14}{16.8}{\rmdefault}{\mddefault}{\itdefault}{$2y+z+2\e_a+2\e_b$}%
}}}}}
\put(11608,-15109){\rotatebox{300.0}{\makebox(0,0)[lb]{\smash{{\SetFigFont{14}{16.8}{\rmdefault}{\mddefault}{\itdefault}{$2y+2\e_a+2\e_b$}%
}}}}}
\put(4583,-14941){\rotatebox{300.0}{\makebox(0,0)[lb]{\smash{{\SetFigFont{14}{16.8}{\rmdefault}{\mddefault}{\itdefault}{$2y+2\od_a+2\od_b$}%
}}}}}
\put(11710,-2609){\rotatebox{300.0}{\makebox(0,0)[lb]{\smash{{\SetFigFont{14}{16.8}{\rmdefault}{\mddefault}{\itdefault}{$2y+2\od_a+2\od_b$}%
}}}}}
\put(2002,-12987){\makebox(0,0)[lb]{\smash{{\SetFigFont{14}{16.8}{\rmdefault}{\mddefault}{\itdefault}{$\e_a+\e_b$}%
}}}}
\put(9095,-12778){\makebox(0,0)[lb]{\smash{{\SetFigFont{14}{16.8}{\rmdefault}{\mddefault}{\itdefault}{$\e_a+\e_b$}%
}}}}
\put(9231,-832){\makebox(0,0)[lb]{\smash{{\SetFigFont{14}{16.8}{\rmdefault}{\mddefault}{\itdefault}{$\e_a+\e_b$}%
}}}}
\put(1181,-3981){\rotatebox{90.0}{\makebox(0,0)[lb]{\smash{{\SetFigFont{14}{16.8}{\rmdefault}{\mddefault}{\itdefault}{$y+\od_a+\od_b$}%
}}}}}
\put(1181,-8469){\rotatebox{90.0}{\makebox(0,0)[lb]{\smash{{\SetFigFont{14}{16.8}{\rmdefault}{\mddefault}{\itdefault}{$y+z+\e_a+\e_b$}%
}}}}}
\put(2251,-11208){\makebox(0,0)[lb]{\smash{{\SetFigFont{14}{16.8}{\rmdefault}{\mddefault}{\itdefault}{$\od_a+\od_b$}%
}}}}
\put(5097,-9296){\rotatebox{60.0}{\makebox(0,0)[lb]{\smash{{\SetFigFont{14}{16.8}{\rmdefault}{\mddefault}{\itdefault}{$2y+z+2\e_a+2\e_b$}%
}}}}}
\put(6397,-5575){\makebox(0,0)[lb]{\smash{{\SetFigFont{14}{16.8}{\rmdefault}{\mddefault}{\itdefault}{$z$}%
}}}}
\put(4659,-2463){\rotatebox{300.0}{\makebox(0,0)[lb]{\smash{{\SetFigFont{14}{16.8}{\rmdefault}{\mddefault}{\itdefault}{$2y+2\od_a+2\od_b$}%
}}}}}
\put(2101,-910){\makebox(0,0)[lb]{\smash{{\SetFigFont{14}{16.8}{\rmdefault}{\mddefault}{\itdefault}{$\e_a+\e_b$}%
}}}}
\put(4726,-18616){\makebox(0,0)[lb]{\smash{{\SetFigFont{14}{16.8}{\rmdefault}{\mddefault}{\itdefault}{$b_2$}%
}}}}
\put(12811,-18781){\makebox(0,0)[lb]{\smash{{\SetFigFont{14}{16.8}{\rmdefault}{\mddefault}{\itdefault}{$b_1$}%
}}}}
\put(11761,-19186){\makebox(0,0)[lb]{\smash{{\SetFigFont{14}{16.8}{\rmdefault}{\mddefault}{\itdefault}{$b_3$}%
}}}}
\end{picture}%
}
\caption{Splitting an $H^{(1)}$-type region into two $Q$-type regions when $x=0$.}\label{halvedhexbase2}
\end{figure}

\begin{figure}\centering
\setlength{\unitlength}{3947sp}%
\begingroup\makeatletter\ifx\SetFigFont\undefined%
\gdef\SetFigFont#1#2#3#4#5{%
  \reset@font\fontsize{#1}{#2pt}%
  \fontfamily{#3}\fontseries{#4}\fontshape{#5}%
  \selectfont}%
\fi\endgroup%
\resizebox{14cm}{!}{
\begin{picture}(0,0)%
\includegraphics{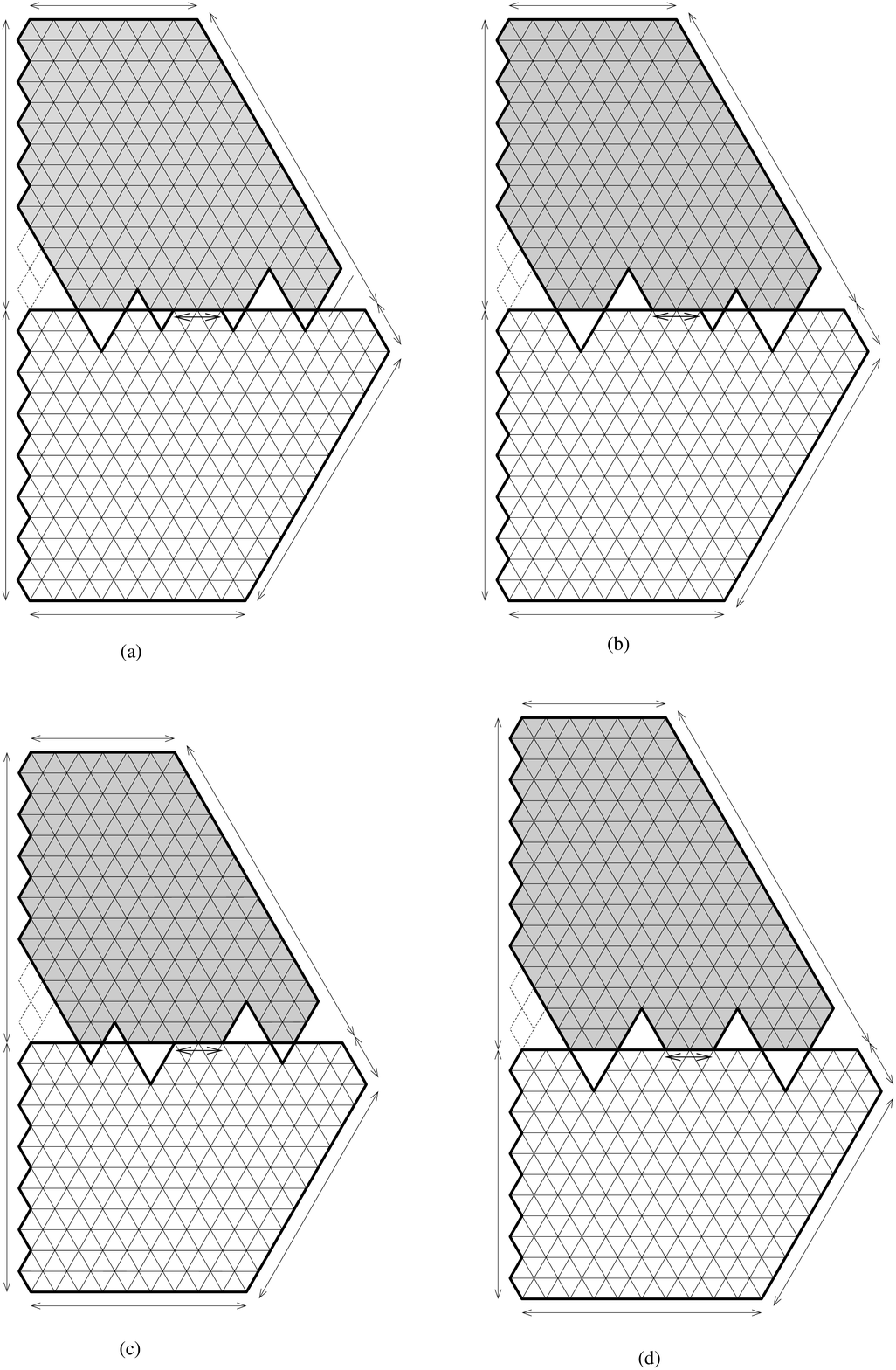}%
\end{picture}%
%
%

\begin{picture}(15593,23552)(1072,-22817)
\put(6046,-17437){\makebox(0,0)[lb]{\smash{{\SetFigFont{14}{16.8}{\rmdefault}{\mddefault}{\itdefault}{$b_2$}%
}}}}
\put(5436,-17147){\makebox(0,0)[lb]{\smash{{\SetFigFont{14}{16.8}{\rmdefault}{\mddefault}{\itdefault}{$b_3$}%
}}}}
\put(13826,-17237){\makebox(0,0)[lb]{\smash{{\SetFigFont{14}{16.8}{\rmdefault}{\mddefault}{\itdefault}{$b_3$}%
}}}}
\put(14656,-17677){\makebox(0,0)[lb]{\smash{{\SetFigFont{14}{16.8}{\rmdefault}{\mddefault}{\itdefault}{$b_2$}%
}}}}
\put(10561,-17177){\makebox(0,0)[lb]{\smash{{\SetFigFont{14}{16.8}{\rmdefault}{\mddefault}{\itdefault}{$a_1$}%
}}}}
\put(12181,-17187){\makebox(0,0)[lb]{\smash{{\SetFigFont{14}{16.8}{\rmdefault}{\mddefault}{\itdefault}{$a_3$}%
}}}}
\put(11371,-17667){\makebox(0,0)[lb]{\smash{{\SetFigFont{14}{16.8}{\rmdefault}{\mddefault}{\itdefault}{$a_2$}%
}}}}
\put(15516,-17187){\makebox(0,0)[lb]{\smash{{\SetFigFont{14}{16.8}{\rmdefault}{\mddefault}{\itdefault}{$b_1$}%
}}}}
\put(6686,-17087){\makebox(0,0)[lb]{\smash{{\SetFigFont{14}{16.8}{\rmdefault}{\mddefault}{\itdefault}{$b_1$}%
}}}}
\put(3821,-17552){\makebox(0,0)[lb]{\smash{{\SetFigFont{14}{16.8}{\rmdefault}{\mddefault}{\itdefault}{$a_4$}%
}}}}
\put(3171,-17197){\makebox(0,0)[lb]{\smash{{\SetFigFont{14}{16.8}{\rmdefault}{\mddefault}{\itdefault}{$a_3$}%
}}}}
\put(2771,-17417){\makebox(0,0)[lb]{\smash{{\SetFigFont{14}{16.8}{\rmdefault}{\mddefault}{\itdefault}{$a_2$}%
}}}}
\put(2141,-17107){\makebox(0,0)[lb]{\smash{{\SetFigFont{14}{16.8}{\rmdefault}{\mddefault}{\itdefault}{$a_1$}%
}}}}
\put(13396,-4926){\makebox(0,0)[lb]{\smash{{\SetFigFont{14}{16.8}{\rmdefault}{\mddefault}{\itdefault}{$b_4$}%
}}}}
\put(13811,-4691){\makebox(0,0)[lb]{\smash{{\SetFigFont{14}{16.8}{\rmdefault}{\mddefault}{\itdefault}{$b_3$}%
}}}}
\put(14426,-5056){\makebox(0,0)[lb]{\smash{{\SetFigFont{14}{16.8}{\rmdefault}{\mddefault}{\itdefault}{$b_2$}%
}}}}
\put(15256,-4566){\makebox(0,0)[lb]{\smash{{\SetFigFont{14}{16.8}{\rmdefault}{\mddefault}{\itdefault}{$b_1$}%
}}}}
\put(11941,-4576){\makebox(0,0)[lb]{\smash{{\SetFigFont{14}{16.8}{\rmdefault}{\mddefault}{\itdefault}{$a_3$}%
}}}}
\put(11131,-5036){\makebox(0,0)[lb]{\smash{{\SetFigFont{14}{16.8}{\rmdefault}{\mddefault}{\itdefault}{$a_2$}%
}}}}
\put(10301,-4576){\makebox(0,0)[lb]{\smash{{\SetFigFont{14}{16.8}{\rmdefault}{\mddefault}{\itdefault}{$a_1$}%
}}}}
\put(5211,-4931){\makebox(0,0)[lb]{\smash{{\SetFigFont{14}{16.8}{\rmdefault}{\mddefault}{\itdefault}{$b_4$}%
}}}}
\put(5831,-4621){\makebox(0,0)[lb]{\smash{{\SetFigFont{14}{16.8}{\rmdefault}{\mddefault}{\itdefault}{$b_3$}%
}}}}
\put(6441,-4941){\makebox(0,0)[lb]{\smash{{\SetFigFont{14}{16.8}{\rmdefault}{\mddefault}{\itdefault}{$b_2$}%
}}}}
\put(7131,-4591){\makebox(0,0)[lb]{\smash{{\SetFigFont{14}{16.8}{\rmdefault}{\mddefault}{\itdefault}{$b_1$}%
}}}}
\put(3981,-4911){\makebox(0,0)[lb]{\smash{{\SetFigFont{14}{16.8}{\rmdefault}{\mddefault}{\itdefault}{$a_4$}%
}}}}
\put(3561,-4691){\makebox(0,0)[lb]{\smash{{\SetFigFont{14}{16.8}{\rmdefault}{\mddefault}{\itdefault}{$a_3$}%
}}}}
\put(4651,-5061){\makebox(0,0)[lb]{\smash{{\SetFigFont{14}{16.8}{\rmdefault}{\mddefault}{\itdefault}{$x$}%
}}}}
\put(2971,-5041){\makebox(0,0)[lb]{\smash{{\SetFigFont{14}{16.8}{\rmdefault}{\mddefault}{\itdefault}{$a_2$}%
}}}}
\put(2141,-4621){\makebox(0,0)[lb]{\smash{{\SetFigFont{14}{16.8}{\rmdefault}{\mddefault}{\itdefault}{$a_1$}%
}}}}
\put(9761,-19952){\rotatebox{90.0}{\makebox(0,0)[lb]{\smash{{\SetFigFont{12}{14.4}{\rmdefault}{\mddefault}{\itdefault}{$z+\e_a+\e_b$}%
}}}}}
\put(1361,-19672){\rotatebox{90.0}{\makebox(0,0)[lb]{\smash{{\SetFigFont{12}{14.4}{\rmdefault}{\mddefault}{\itdefault}{$z+\e_a+\e_b$}%
}}}}}
\put(9521,-7941){\rotatebox{90.0}{\makebox(0,0)[lb]{\smash{{\SetFigFont{12}{14.4}{\rmdefault}{\mddefault}{\itdefault}{$z+\e_a+\e_b$}%
}}}}}
\put(1321,-7861){\rotatebox{90.0}{\makebox(0,0)[lb]{\smash{{\SetFigFont{12}{14.4}{\rmdefault}{\mddefault}{\itdefault}{$z+\e_a+\e_b$}%
}}}}}
\put(1351,-15657){\rotatebox{90.0}{\makebox(0,0)[lb]{\smash{{\SetFigFont{12}{14.4}{\rmdefault}{\mddefault}{\itdefault}{$\od_a+\od_b$}%
}}}}}
\put(9751,-15167){\rotatebox{90.0}{\makebox(0,0)[lb]{\smash{{\SetFigFont{12}{14.4}{\rmdefault}{\mddefault}{\itdefault}{$\od_a+\od_b$}%
}}}}}
\put(9531,-3096){\rotatebox{90.0}{\makebox(0,0)[lb]{\smash{{\SetFigFont{12}{14.4}{\rmdefault}{\mddefault}{\itdefault}{$\od_a+\od_b$}%
}}}}}
\put(1251,-3341){\rotatebox{90.0}{\makebox(0,0)[lb]{\smash{{\SetFigFont{12}{14.4}{\rmdefault}{\mddefault}{\itdefault}{$\od_a+\od_b$}%
}}}}}
\put(11968,-22254){\rotatebox{1.0}{\makebox(0,0)[lb]{\smash{{\SetFigFont{12}{14.4}{\rmdefault}{\mddefault}{\itdefault}{$x+\od_a+\od_b$}%
}}}}}
\put(3078,-22124){\rotatebox{1.0}{\makebox(0,0)[lb]{\smash{{\SetFigFont{12}{14.4}{\rmdefault}{\mddefault}{\itdefault}{$x+\od_a+\od_b$}%
}}}}}
\put(11138,-10203){\rotatebox{1.0}{\makebox(0,0)[lb]{\smash{{\SetFigFont{12}{14.4}{\rmdefault}{\mddefault}{\itdefault}{$x+\od_a+\od_b$}%
}}}}}
\put(3011,-10241){\rotatebox{1.0}{\makebox(0,0)[lb]{\smash{{\SetFigFont{12}{14.4}{\rmdefault}{\mddefault}{\itdefault}{$x+\od_a+\od_b$}%
}}}}}
\put(16527,-17553){\rotatebox{1.0}{\makebox(0,0)[lb]{\smash{{\SetFigFont{12}{14.4}{\rmdefault}{\mddefault}{\itdefault}{$z$}%
}}}}}
\put(7697,-17353){\rotatebox{1.0}{\makebox(0,0)[lb]{\smash{{\SetFigFont{12}{14.4}{\rmdefault}{\mddefault}{\itdefault}{$z$}%
}}}}}
\put(6668,-20659){\rotatebox{60.0}{\makebox(0,0)[lb]{\smash{{\SetFigFont{12}{14.4}{\rmdefault}{\mddefault}{\itdefault}{$z+2\e_a+2\e_b$}%
}}}}}
\put(15428,-20839){\rotatebox{60.0}{\makebox(0,0)[lb]{\smash{{\SetFigFont{12}{14.4}{\rmdefault}{\mddefault}{\itdefault}{$z+2\e_a+2\e_b$}%
}}}}}
\put(15068,-8458){\rotatebox{60.0}{\makebox(0,0)[lb]{\smash{{\SetFigFont{12}{14.4}{\rmdefault}{\mddefault}{\itdefault}{$z+2\e_a+2\e_b$}%
}}}}}
\put(6741,-8611){\rotatebox{60.0}{\makebox(0,0)[lb]{\smash{{\SetFigFont{12}{14.4}{\rmdefault}{\mddefault}{\itdefault}{$z+2\e_a+2\e_b$}%
}}}}}
\put(16327,-4842){\rotatebox{1.0}{\makebox(0,0)[lb]{\smash{{\SetFigFont{12}{14.4}{\rmdefault}{\mddefault}{\itdefault}{$z$}%
}}}}}
\put(8171,-4821){\rotatebox{1.0}{\makebox(0,0)[lb]{\smash{{\SetFigFont{12}{14.4}{\rmdefault}{\mddefault}{\itdefault}{$z$}%
}}}}}
\put(14407,-14026){\rotatebox{300.0}{\makebox(0,0)[lb]{\smash{{\SetFigFont{12}{14.4}{\rmdefault}{\mddefault}{\itdefault}{$2\od_a+2\od_b$}%
}}}}}
\put(5877,-14246){\rotatebox{300.0}{\makebox(0,0)[lb]{\smash{{\SetFigFont{12}{14.4}{\rmdefault}{\mddefault}{\itdefault}{$2\od_a+2\od_b$}%
}}}}}
\put(14367,-1635){\rotatebox{300.0}{\makebox(0,0)[lb]{\smash{{\SetFigFont{12}{14.4}{\rmdefault}{\mddefault}{\itdefault}{$2\od_a+2\od_b$}%
}}}}}
\put(6138,-1623){\rotatebox{300.0}{\makebox(0,0)[lb]{\smash{{\SetFigFont{12}{14.4}{\rmdefault}{\mddefault}{\itdefault}{$2\od_a+2\od_b$}%
}}}}}
\put(10946,-11332){\makebox(0,0)[lb]{\smash{{\SetFigFont{12}{14.4}{\rmdefault}{\mddefault}{\itdefault}{$x+\e_a+\e_b$}%
}}}}
\put(2496,-11952){\makebox(0,0)[lb]{\smash{{\SetFigFont{12}{14.4}{\rmdefault}{\mddefault}{\itdefault}{$x+\e_a+\e_b$}%
}}}}
\put(10926,529){\makebox(0,0)[lb]{\smash{{\SetFigFont{12}{14.4}{\rmdefault}{\mddefault}{\itdefault}{$x+\e_a+\e_b$}%
}}}}
\put(2761,499){\makebox(0,0)[lb]{\smash{{\SetFigFont{14}{16.8}{\rmdefault}{\mddefault}{\itdefault}{$x+\e_a+\e_b$}%
}}}}
\put(4661,-17552){\makebox(0,0)[lb]{\smash{{\SetFigFont{14}{16.8}{\rmdefault}{\mddefault}{\itdefault}{$x$}%
}}}}
\put(13031,-17672){\makebox(0,0)[lb]{\smash{{\SetFigFont{14}{16.8}{\rmdefault}{\mddefault}{\itdefault}{$x$}%
}}}}
\put(12821,-5051){\makebox(0,0)[lb]{\smash{{\SetFigFont{14}{16.8}{\rmdefault}{\mddefault}{\itdefault}{$x$}%
}}}}
\end{picture}%
}
\caption{Splitting an $H^{(1)}$-type region into two $Q$-type regions when $y=0$.}\label{halvedhexbase}
\end{figure}

\pagebreak

\begin{proof}[Combined Proof of Theorems \ref{main1} and \ref{mainR1}]

We first prove that
\begin{align}\label{main1eqn}
\M(H^{(1)}_{x,y,z}(\textbf{a},\textbf{b}))&=2^{-y}\Q(0,a_1,\dotsc,a_{2\lfloor\frac{m+1}{2}\rfloor-1},a_{2\lfloor\frac{m+1}{2}\rfloor}+x+y+b_{2\lfloor\frac{n+1}{2}\rfloor},b_{2\lfloor\frac{n+1}{2}\rfloor-1},\dotsc,b_1)\notag\\
&\quad\times \Q(a_1,\dotsc,a_{\lceil \frac{m-1}{2}\rceil}, a_{\lceil \frac{m-1}{2}\rceil+1}+x+y+b_{\lceil \frac{n-1}{2}\rceil+1},b_{\lceil \frac{n-1}{2}\rceil},\dotsc, b_1,z) \notag\\
&\quad\times \frac{\Hf_2(2\od_a+2\od_b+1)\Hf_2(2\e_a+2\e_b+2z+1)}{\Hf_2(2\od_a+2\od_b+2y+1)\Hf_2(2\e_a+2\e_b+2y+2z+1)} \notag\\
&\quad\times \frac{\Hf(2a+b+2y+z+1)\Hf(b+y+z)}{\Hf(2a+b+y+z+1)\Hf(b+z)}\notag\\
&\quad\times  \frac{\T(x+1,2a+b+2y+z,y)\V(2x+2a+3,b+2y+z-1,y)}{\T(1,2a+b+2y+z,y)\V(2a+3,b+2y+z-1,y)},
\end{align}

and that
\begin{align}\label{mainR1eqn}
\M(R^{(1)}_{x,y,z}(\textbf{a},\textbf{b}))&=2^{-y}\Q(a_1,\dotsc,a_{2\lceil \frac{m-1}{2}\rceil}, a_{2\lceil \frac{m-1}{2}\rceil+1}+x+y+b_{2\lfloor\frac{n+1}{2}\rfloor},b_{2\lfloor\frac{n+1}{2}\rfloor-1},\dotsc,b_1)\notag\\
&\quad\times \Q(0,a_1,\dotsc,a_{2\lfloor\frac{m+1}{2}\rfloor-1},a_{2\lfloor\frac{m+1}{2}\rfloor}+x+y+b_{2\lceil \frac{n-1}{2}\rceil+1},b_{2\lceil \frac{n-1}{2}\rceil},\dotsc, b_1,z) \notag\\
&\quad\times \frac{\Hf_2(2\e_a+2\od_b+1)\Hf_2(2\od_a+2\e_b+2z+1)}{\Hf_2(2\e_a+2\od_b+2y+1)\Hf_2(2\od_a+2\e_b+2y+2z+1)} \notag\\
&\quad\times \frac{\Hf(2a+b+2y+z+1)\Hf(b+y+z)}{\Hf(2a+b+y+z+1)\Hf(b+z)}\notag\\
&\quad\times  \frac{\T(x+1,2a+b+2y+z,y)\V(2x+2a+3,b+2y+z-1,y)}{\T(1,2a+b+2y+z,y)\V(2a+3,b+2y+z-1,y)}
\end{align}
by induction on  $y+z+b+\overline{n}$, where $\overline{n}$ is the number of positive terms in the sequence $\textbf{b}=(b_1,b_2,\dotsc,b_n)$.
The base cases are the situations when at least one of the parameters $x,y,$ and the sum $b+\overline{n}$ is equal to $0$.

We consider first the case when $x=0$. After removing several forced lozenges in the middle of the two fern,  we split the regions into two subregions along the line that the ferns are lying on (see Figure \ref{halvedhexbase2}).
The upper subregion, after recovering several forced vertical lozenges, is the region:

\begin{enumerate}
\item $\mathcal{Q}(0,a_1,\dotsc, a_m,y,b_n,\dots,b_1)$ if $m,n$ are even,
\item  $\mathcal{Q}(0,a_1,\dotsc, a_m,y+b_n,\dots,b_1)$ if $m$ is even and $n$ is odd,
\item $\mathcal{Q}(0,a_1,\dotsc, a_m+y,b_n,\dots,b_1)$ if $m$ is odd and $n$ is even,
\item $\mathcal{Q}(0,a_1,\dotsc, a_{m-1}, a_m+y+b_n,b_{n-1},\dots,b_1)$ if $m,n$ are odd.
   \end{enumerate}
       The lower subregion is a horizontal reflection of the  regions:
       \begin{enumerate}
\item $\mathcal{Q}(a_1,\dotsc, a_{m-1},a_m+y+b_n,b_{n-1},\dots,b_1,z )$ if $m,n$ are even,
\item  $\mathcal{Q}(a_1,\dotsc, a_m+y,b_n,\dots,b_1,z)$ if $m$ is even and $n$ is odd,
\item
$\mathcal{Q}(a_1,\dotsc, a_m,y+b_n,\dots,b_1,z)$ if $m$ is odd and $n$ is even,
\item $\mathcal{Q}(a_1,\dotsc, a_{m-1}, a_m,y,b_n,b_{n-1},\dots,b_1,z)$ if $m,n$ are odd.
 \end{enumerate}
 By Region-splitting Lemma \ref{RS},
the number of tilings of our region is given by the product of the tiling numbers of the two $\mathcal{Q}$-type regions corresponding to the upper and lower subregions. It means that (\ref{main1eqn}) follows from Lemma \ref{QAR}.
Similarly, we can also write the number of tilings of the region $R^{(1)}_{x,y,z}(\textbf{a},\textbf{b})$ as the product of that of the two $\mathcal{Q}$-type regions, and (\ref{mainR1eqn}) follows again  from Lemma \ref{QAR}.

If $y=0$, we also divide the region along the ferns into two parts corresponding to two $\mathcal{Q}$-type region as in the Figure \ref{halvedhexbase}. In particular, the two $\mathcal{Q}$-type regions (corresponding to the upper and lower parts respectively) are:
\begin{enumerate}
\item  $\mathcal{Q}(0,a_1,\dots,a_m+x+b_n,\dots,b_1)$ and $\mathcal{Q}(a_1,\dots,a_m,x,b_n,\dots,b_1,z)$ if $m,n$ are even,
\item  $\mathcal{Q}(0,a_1,\dots,a_m,x+b_n,\dots,b_1)$ and $\mathcal{Q}(a_1,\dots,a_m+x,b_n,\dots,b_1,z)$ if $m$ is odd and $n$ is even,
\item $\mathcal{Q}(0,a_1,\dots,a_m+x,b_n,\dots,b_1)$ and $\mathcal{Q}(a_1,\dots,a_m,x+b_n,\dots,b_1,z)$ if $m$ is even and $n$ is odd,
\item  $\mathcal{Q}(0,a_1,\dots,a_m,x,b_n,\dots,b_1)$ and $\mathcal{Q}(a_1,\dots,a_m+x+b_n,\dots,b_1,z)$ if $m,n$ are odd.
\end{enumerate}
This means that  (\ref{main1eqn}) is implied by Lemma \ref{QAR}. The verification for (\ref{mainR1eqn}) in this case can be treated in the same manner.

If $b+ \overline{n}=0$, then we have the right fern empty. This case was already considered in \cite{Halfhex2}.

\begin{figure}\centering
\setlength{\unitlength}{3947sp}%
\begingroup\makeatletter\ifx\SetFigFont\undefined%
\gdef\SetFigFont#1#2#3#4#5{%
  \reset@font\fontsize{#1}{#2pt}%
  \fontfamily{#3}\fontseries{#4}\fontshape{#5}%
  \selectfont}%
\fi\endgroup%
\resizebox{12cm}{!}{
\begin{picture}(0,0)%
\includegraphics{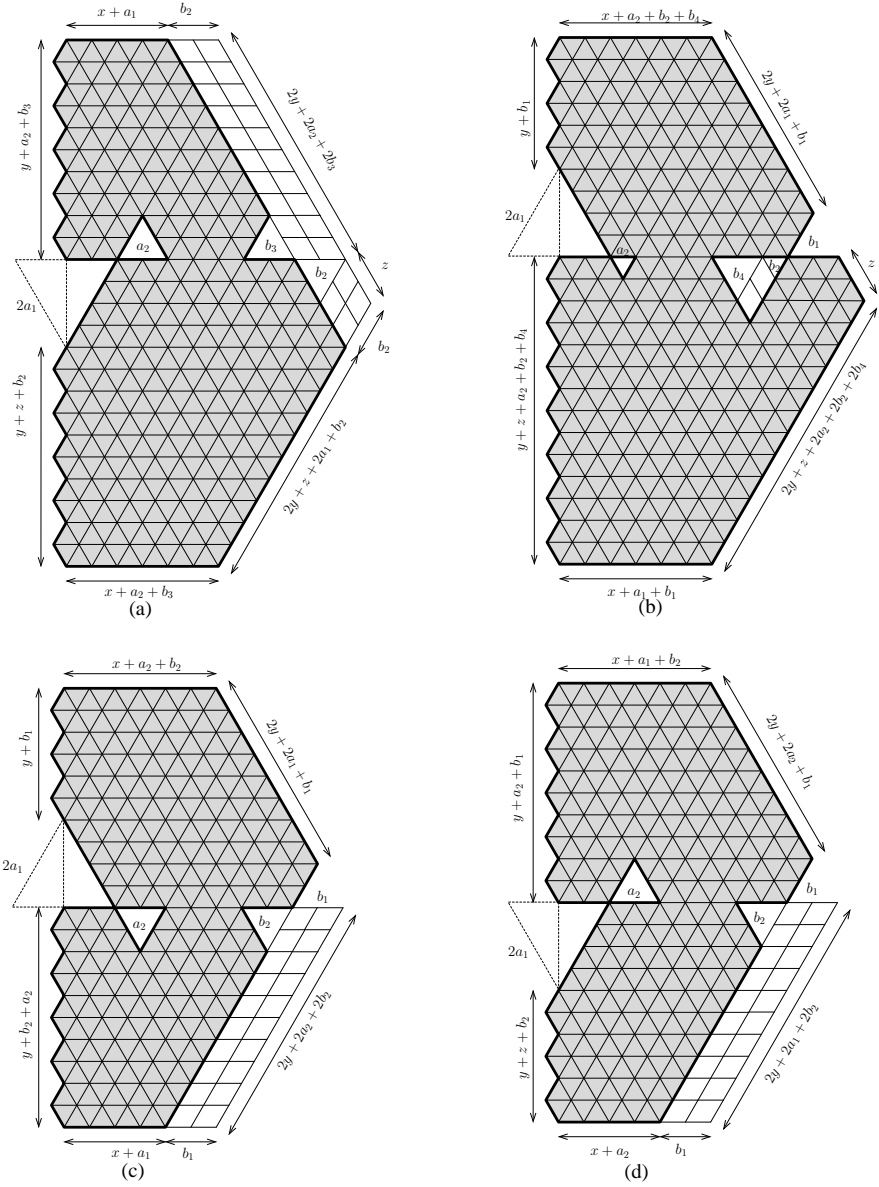}%
\end{picture}%
%
%

\begin{picture}(14502,19222)(1633,-18641)
\put(1877,-4501){\makebox(0,0)[lb]{\smash{{\SetFigFont{14}{16.8}{\rmdefault}{\mddefault}{\itdefault}{$2a_1$}%
}}}}
\put(3812,-3541){\makebox(0,0)[lb]{\smash{{\SetFigFont{14}{16.8}{\rmdefault}{\mddefault}{\itdefault}{$a_2$}%
}}}}
\put(5889,-3533){\makebox(0,0)[lb]{\smash{{\SetFigFont{14}{16.8}{\rmdefault}{\mddefault}{\itdefault}{$b_3$}%
}}}}
\put(6699,-3976){\makebox(0,0)[lb]{\smash{{\SetFigFont{14}{16.8}{\rmdefault}{\mddefault}{\itdefault}{$b_2$}%
}}}}
\put(7711,-5161){\makebox(0,0)[lb]{\smash{{\SetFigFont{14}{16.8}{\rmdefault}{\mddefault}{\itdefault}{$b_2$}%
}}}}
\put(7764,-3841){\makebox(0,0)[lb]{\smash{{\SetFigFont{14}{16.8}{\rmdefault}{\mddefault}{\itdefault}{$z$}%
}}}}
\put(3182,247){\makebox(0,0)[lb]{\smash{{\SetFigFont{14}{16.8}{\rmdefault}{\mddefault}{\itdefault}{$x+a_1$}%
}}}}
\put(4517,292){\makebox(0,0)[lb]{\smash{{\SetFigFont{14}{16.8}{\rmdefault}{\mddefault}{\itdefault}{$b_2$}%
}}}}
\put(6234,-991){\rotatebox{300.0}{\makebox(0,0)[lb]{\smash{{\SetFigFont{14}{16.8}{\rmdefault}{\mddefault}{\itdefault}{$2y+2a_2+2b_3$}%
}}}}}
\put(2065,-2267){\rotatebox{90.0}{\makebox(0,0)[lb]{\smash{{\SetFigFont{14}{16.8}{\rmdefault}{\mddefault}{\itdefault}{$y+a_2+b_3$}%
}}}}}
\put(1982,-6548){\rotatebox{90.0}{\makebox(0,0)[lb]{\smash{{\SetFigFont{14}{16.8}{\rmdefault}{\mddefault}{\itdefault}{$y+z+b_2$}%
}}}}}
\put(3292,-9118){\makebox(0,0)[lb]{\smash{{\SetFigFont{14}{16.8}{\rmdefault}{\mddefault}{\itdefault}{$x+a_2+b_3$}%
}}}}
\put(6301,-7816){\rotatebox{60.0}{\makebox(0,0)[lb]{\smash{{\SetFigFont{14}{16.8}{\rmdefault}{\mddefault}{\itdefault}{$2y+z+2a_1+b_2$}%
}}}}}
\put(11551,-3601){\makebox(0,0)[lb]{\smash{{\SetFigFont{14}{16.8}{\rmdefault}{\mddefault}{\itdefault}{$a_2$}%
}}}}
\put(9755,-3046){\makebox(0,0)[lb]{\smash{{\SetFigFont{14}{16.8}{\rmdefault}{\mddefault}{\itdefault}{$2a_1$}%
}}}}
\put(14011,-3856){\makebox(0,0)[lb]{\smash{{\SetFigFont{14}{16.8}{\rmdefault}{\mddefault}{\itdefault}{$b_2$}%
}}}}
\put(13434,-3984){\makebox(0,0)[lb]{\smash{{\SetFigFont{14}{16.8}{\rmdefault}{\mddefault}{\itdefault}{$b_4$}%
}}}}
\put(14649,-3466){\makebox(0,0)[lb]{\smash{{\SetFigFont{14}{16.8}{\rmdefault}{\mddefault}{\itdefault}{$b_1$}%
}}}}
\put(11311,171){\makebox(0,0)[lb]{\smash{{\SetFigFont{14}{16.8}{\rmdefault}{\mddefault}{\itdefault}{$x+a_2+b_2+b_4$}%
}}}}
\put(13891,-691){\rotatebox{300.0}{\makebox(0,0)[lb]{\smash{{\SetFigFont{14}{16.8}{\rmdefault}{\mddefault}{\itdefault}{$2y+2a_1+b_1$}%
}}}}}
\put(15609,-3804){\makebox(0,0)[lb]{\smash{{\SetFigFont{14}{16.8}{\rmdefault}{\mddefault}{\itdefault}{$z$}%
}}}}
\put(14334,-7509){\rotatebox{60.0}{\makebox(0,0)[lb]{\smash{{\SetFigFont{14}{16.8}{\rmdefault}{\mddefault}{\itdefault}{$2y+z+2a_2+2b_2+2b_4$}%
}}}}}
\put(11409,-9129){\makebox(0,0)[lb]{\smash{{\SetFigFont{14}{16.8}{\rmdefault}{\mddefault}{\itdefault}{$x+a_1+b_1$}%
}}}}
\put(10074,-6789){\rotatebox{90.0}{\makebox(0,0)[lb]{\smash{{\SetFigFont{14}{16.8}{\rmdefault}{\mddefault}{\itdefault}{$y+z+a_2+b_2+b_4$}%
}}}}}
\put(10119,-1711){\rotatebox{90.0}{\makebox(0,0)[lb]{\smash{{\SetFigFont{14}{16.8}{\rmdefault}{\mddefault}{\itdefault}{$y+b_1$}%
}}}}}
\put(1648,-13553){\makebox(0,0)[lb]{\smash{{\SetFigFont{14}{16.8}{\rmdefault}{\mddefault}{\itdefault}{$2a_1$}%
}}}}
\put(3751,-14476){\makebox(0,0)[lb]{\smash{{\SetFigFont{14}{16.8}{\rmdefault}{\mddefault}{\itdefault}{$a_2$}%
}}}}
\put(11731,-13956){\makebox(0,0)[lb]{\smash{{\SetFigFont{14}{16.8}{\rmdefault}{\mddefault}{\itdefault}{$a_2$}%
}}}}
\put(9808,-14931){\makebox(0,0)[lb]{\smash{{\SetFigFont{14}{16.8}{\rmdefault}{\mddefault}{\itdefault}{$2a_1$}%
}}}}
\put(14622,-13911){\makebox(0,0)[lb]{\smash{{\SetFigFont{14}{16.8}{\rmdefault}{\mddefault}{\itdefault}{$b_1$}%
}}}}
\put(6725,-14041){\makebox(0,0)[lb]{\smash{{\SetFigFont{14}{16.8}{\rmdefault}{\mddefault}{\itdefault}{$b_1$}%
}}}}
\put(13767,-14391){\makebox(0,0)[lb]{\smash{{\SetFigFont{14}{16.8}{\rmdefault}{\mddefault}{\itdefault}{$b_2$}%
}}}}
\put(5787,-14446){\makebox(0,0)[lb]{\smash{{\SetFigFont{14}{16.8}{\rmdefault}{\mddefault}{\itdefault}{$b_2$}%
}}}}
\put(4550,-18196){\makebox(0,0)[lb]{\smash{{\SetFigFont{14}{16.8}{\rmdefault}{\mddefault}{\itdefault}{$b_1$}%
}}}}
\put(12507,-18133){\makebox(0,0)[lb]{\smash{{\SetFigFont{14}{16.8}{\rmdefault}{\mddefault}{\itdefault}{$b_1$}%
}}}}
\put(3421,-10291){\makebox(0,0)[lb]{\smash{{\SetFigFont{14}{16.8}{\rmdefault}{\mddefault}{\itdefault}{$x+a_2+b_2$}%
}}}}
\put(5904,-11169){\rotatebox{300.0}{\makebox(0,0)[lb]{\smash{{\SetFigFont{14}{16.8}{\rmdefault}{\mddefault}{\itdefault}{$2y+2a_1+b_1$}%
}}}}}
\put(6204,-16794){\rotatebox{60.0}{\makebox(0,0)[lb]{\smash{{\SetFigFont{14}{16.8}{\rmdefault}{\mddefault}{\itdefault}{$2y+2a_2+2b_2$}%
}}}}}
\put(3406,-18174){\makebox(0,0)[lb]{\smash{{\SetFigFont{14}{16.8}{\rmdefault}{\mddefault}{\itdefault}{$x+a_1$}%
}}}}
\put(2086,-16599){\rotatebox{90.0}{\makebox(0,0)[lb]{\smash{{\SetFigFont{14}{16.8}{\rmdefault}{\mddefault}{\itdefault}{$y+b_2+a_2$}%
}}}}}
\put(2086,-11881){\rotatebox{90.0}{\makebox(0,0)[lb]{\smash{{\SetFigFont{14}{16.8}{\rmdefault}{\mddefault}{\itdefault}{$y+b_1$}%
}}}}}
\put(10103,-16923){\rotatebox{90.0}{\makebox(0,0)[lb]{\smash{{\SetFigFont{14}{16.8}{\rmdefault}{\mddefault}{\itdefault}{$y+z+b_2$}%
}}}}}
\put(10006,-12776){\rotatebox{90.0}{\makebox(0,0)[lb]{\smash{{\SetFigFont{14}{16.8}{\rmdefault}{\mddefault}{\itdefault}{$y+a_2+b_1$}%
}}}}}
\put(11480,-10221){\makebox(0,0)[lb]{\smash{{\SetFigFont{14}{16.8}{\rmdefault}{\mddefault}{\itdefault}{$x+a_1+b_2$}%
}}}}
\put(13968,-11091){\rotatebox{300.0}{\makebox(0,0)[lb]{\smash{{\SetFigFont{14}{16.8}{\rmdefault}{\mddefault}{\itdefault}{$2y+2a_2+b_1$}%
}}}}}
\put(14064,-17009){\rotatebox{60.0}{\makebox(0,0)[lb]{\smash{{\SetFigFont{14}{16.8}{\rmdefault}{\mddefault}{\itdefault}{$2y+2a_1+2b_2$}%
}}}}}
\put(11131,-18156){\makebox(0,0)[lb]{\smash{{\SetFigFont{14}{16.8}{\rmdefault}{\mddefault}{\itdefault}{$x+a_2$}%
}}}}
\end{picture}%
}
\caption{(a)--(b): Eliminating triangles of side length $0$ from the $b$-fern. (c)--(d): Region reduction when $z=0$.}\label{fig:halvedhexbase3}
\end{figure}

\begin{figure}\centering
\setlength{\unitlength}{3947sp}%
\begingroup\makeatletter\ifx\SetFigFont\undefined%
\gdef\SetFigFont#1#2#3#4#5{%
  \reset@font\fontsize{#1}{#2pt}%
  \fontfamily{#3}\fontseries{#4}\fontshape{#5}%
  \selectfont}%
\fi\endgroup%
\resizebox{13cm}{!}{
\begin{picture}(0,0)%
\includegraphics{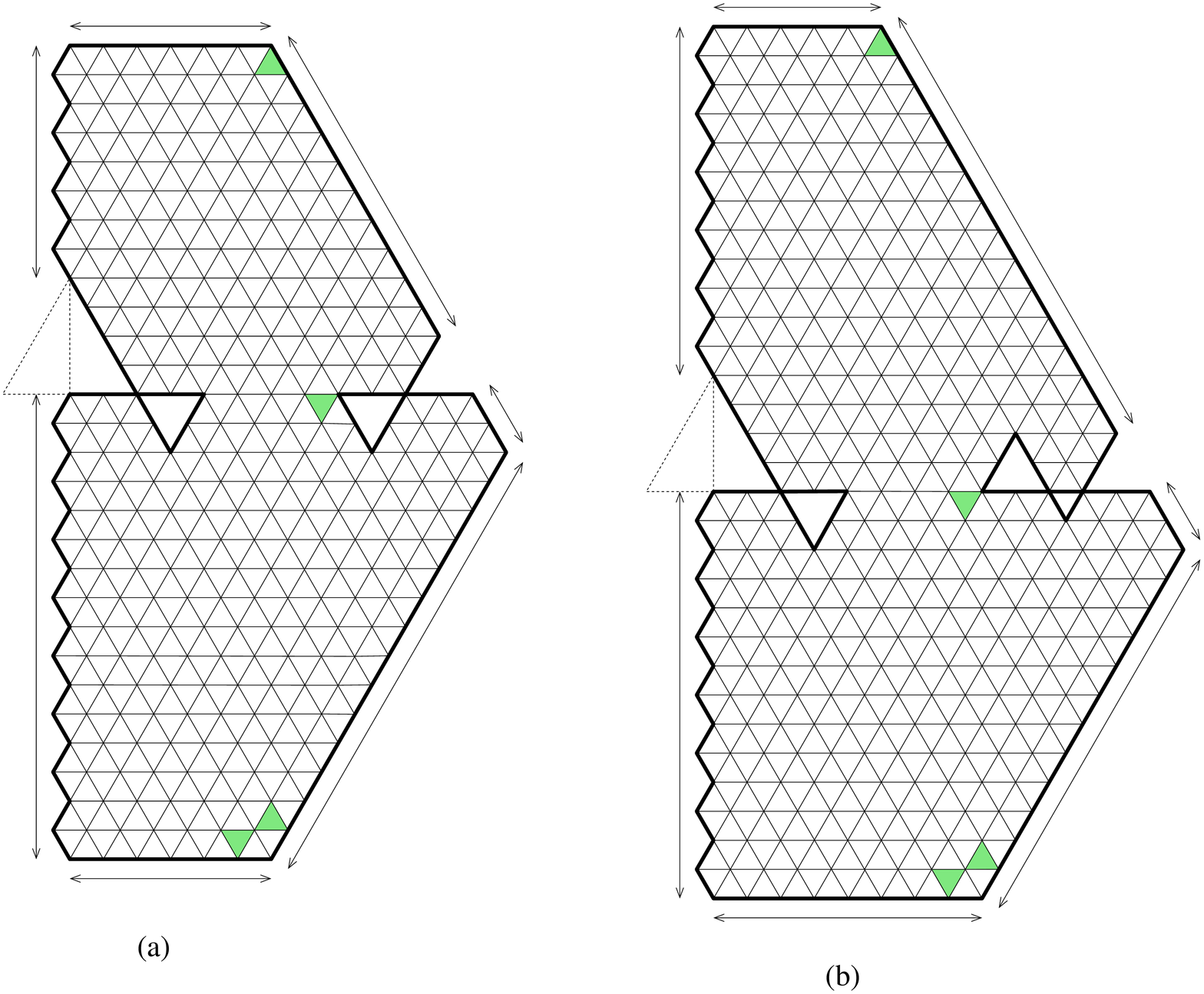}%
\end{picture}%
%
%

\begin{picture}(14672,12426)(1403,-14167)
\put(11229,-8378){\makebox(0,0)[lb]{\smash{{\SetFigFont{14}{16.8}{\rmdefault}{\mddefault}{\itdefault}{$a_2$}%
}}}}
\put(9477,-5038){\rotatebox{90.0}{\makebox(0,0)[lb]{\smash{{\SetFigFont{14}{16.8}{\rmdefault}{\mddefault}{\itdefault}{$y+b_1$}%
}}}}}
\put(1692,-6156){\rotatebox{60.0}{\makebox(0,0)[lb]{\smash{{\SetFigFont{14}{16.8}{\rmdefault}{\mddefault}{\itdefault}{$2a_1$}%
}}}}}
\put(3403,-7166){\makebox(0,0)[lb]{\smash{{\SetFigFont{14}{16.8}{\rmdefault}{\mddefault}{\itdefault}{$a_2$}%
}}}}
\put(6696,-6702){\makebox(0,0)[lb]{\smash{{\SetFigFont{14}{16.8}{\rmdefault}{\mddefault}{\itdefault}{$b_1$}%
}}}}
\put(5851,-7234){\makebox(0,0)[lb]{\smash{{\SetFigFont{14}{16.8}{\rmdefault}{\mddefault}{\itdefault}{$b_2$}%
}}}}
\put(13708,-7911){\makebox(0,0)[lb]{\smash{{\SetFigFont{14}{16.8}{\rmdefault}{\mddefault}{\itdefault}{$b_3$}%
}}}}
\put(2994,-2270){\makebox(0,0)[lb]{\smash{{\SetFigFont{14}{16.8}{\rmdefault}{\mddefault}{\itdefault}{$x+a_2+b_2$}%
}}}}
\put(7775,-6956){\makebox(0,0)[lb]{\smash{{\SetFigFont{14}{16.8}{\rmdefault}{\mddefault}{\itdefault}{$z$}%
}}}}
\put(5817,-3647){\rotatebox{300.0}{\makebox(0,0)[lb]{\smash{{\SetFigFont{14}{16.8}{\rmdefault}{\mddefault}{\itdefault}{$2y+2a_1+b_1$}%
}}}}}
\put(6138,-10972){\rotatebox{60.0}{\makebox(0,0)[lb]{\smash{{\SetFigFont{14}{16.8}{\rmdefault}{\mddefault}{\itdefault}{$2y+z+2a_2+2b_2$}%
}}}}}
\put(1692,-10343){\rotatebox{90.0}{\makebox(0,0)[lb]{\smash{{\SetFigFont{14}{16.8}{\rmdefault}{\mddefault}{\itdefault}{$y+z+a_2+b_2$}%
}}}}}
\put(1638,-4358){\rotatebox{90.0}{\makebox(0,0)[lb]{\smash{{\SetFigFont{14}{16.8}{\rmdefault}{\mddefault}{\itdefault}{$y+b_1$}%
}}}}}
\put(2974,-13036){\makebox(0,0)[lb]{\smash{{\SetFigFont{14}{16.8}{\rmdefault}{\mddefault}{\itdefault}{$x+a_1+b_1$}%
}}}}
\put(9477,-7490){\rotatebox{60.0}{\makebox(0,0)[lb]{\smash{{\SetFigFont{14}{16.8}{\rmdefault}{\mddefault}{\itdefault}{$2a_1$}%
}}}}}
\put(9518,-11345){\rotatebox{90.0}{\makebox(0,0)[lb]{\smash{{\SetFigFont{14}{16.8}{\rmdefault}{\mddefault}{\itdefault}{$y+z+a_2+b_2$}%
}}}}}
\put(11131,-13574){\makebox(0,0)[lb]{\smash{{\SetFigFont{14}{16.8}{\rmdefault}{\mddefault}{\itdefault}{$x+a_1+b_1$}%
}}}}
\put(14720,-11679){\rotatebox{60.0}{\makebox(0,0)[lb]{\smash{{\SetFigFont{14}{16.8}{\rmdefault}{\mddefault}{\itdefault}{$2y+z+2a_2+2b_2$}%
}}}}}
\put(15907,-8221){\makebox(0,0)[lb]{\smash{{\SetFigFont{14}{16.8}{\rmdefault}{\mddefault}{\itdefault}{$z$}%
}}}}
\put(14932,-7914){\makebox(0,0)[lb]{\smash{{\SetFigFont{14}{16.8}{\rmdefault}{\mddefault}{\itdefault}{$b_1$}%
}}}}
\put(14278,-8283){\makebox(0,0)[lb]{\smash{{\SetFigFont{14}{16.8}{\rmdefault}{\mddefault}{\itdefault}{$b_2$}%
}}}}
\put(13614,-4106){\rotatebox{300.0}{\makebox(0,0)[lb]{\smash{{\SetFigFont{14}{16.8}{\rmdefault}{\mddefault}{\itdefault}{$2y+2a_1+b_1$}%
}}}}}
\put(10633,-2030){\makebox(0,0)[lb]{\smash{{\SetFigFont{14}{16.8}{\rmdefault}{\mddefault}{\itdefault}{$x+a_2+b_2$}%
}}}}
\put(13224,-12620){\makebox(0,0)[lb]{\smash{{\SetFigFont{20}{24.0}{\rmdefault}{\mddefault}{\itdefault}{$w$}%
}}}}
\put(4583,-2904){\makebox(0,0)[lb]{\smash{{\SetFigFont{20}{24.0}{\rmdefault}{\mddefault}{\itdefault}{$u$}%
}}}}
\put(12021,-2685){\makebox(0,0)[lb]{\smash{{\SetFigFont{20}{24.0}{\rmdefault}{\mddefault}{\itdefault}{$u$}%
}}}}
\put(5217,-7063){\makebox(0,0)[lb]{\smash{{\SetFigFont{20}{24.0}{\rmdefault}{\mddefault}{\itdefault}{$v$}%
}}}}
\put(13054,-8242){\makebox(0,0)[lb]{\smash{{\SetFigFont{20}{24.0}{\rmdefault}{\mddefault}{\itdefault}{$v$}%
}}}}
\put(4562,-12122){\makebox(0,0)[lb]{\smash{{\SetFigFont{20}{24.0}{\rmdefault}{\mddefault}{\itdefault}{$w$}%
}}}}
\put(4194,-12368){\makebox(0,0)[lb]{\smash{{\SetFigFont{20}{24.0}{\rmdefault}{\mddefault}{\itdefault}{$s$}%
}}}}
\put(12866,-12845){\makebox(0,0)[lb]{\smash{{\SetFigFont{20}{24.0}{\rmdefault}{\mddefault}{\itdefault}{$s$}%
}}}}
\end{picture}%
}
\caption{How to apply Kuo condensation to a $H^{(1)}$-type region for the case $n$ is even (left) and for the case $n$ is odd (right).}\label{KuoH1}
\end{figure}

\begin{figure}\centering
\setlength{\unitlength}{3947sp}%
\begingroup\makeatletter\ifx\SetFigFont\undefined%
\gdef\SetFigFont#1#2#3#4#5{%
  \reset@font\fontsize{#1}{#2pt}%
  \fontfamily{#3}\fontseries{#4}\fontshape{#5}%
  \selectfont}%
\fi\endgroup%
\resizebox{10cm}{!}{
\begin{picture}(0,0)%
\includegraphics{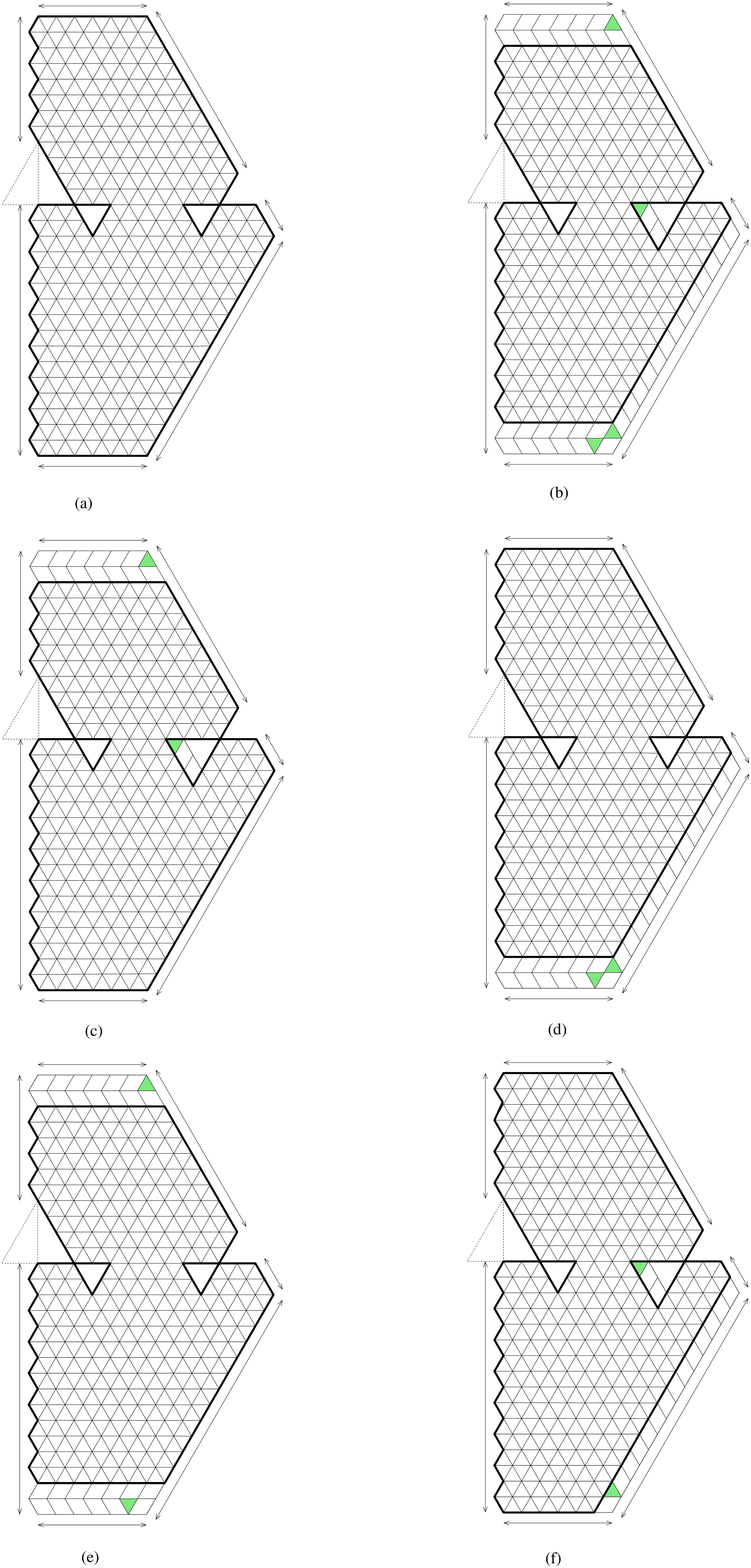}%
\end{picture}%
%
%

\begin{picture}(17101,35694)(1394,-37633)
\put(13473,-36892){\makebox(0,0)[lb]{\smash{{\SetFigFont{18}{16.8}{\rmdefault}{\mddefault}{\itdefault}{$x+a_1+b_1$}%
}}}}
\put(12137,-28214){\rotatebox{90.0}{\makebox(0,0)[lb]{\smash{{\SetFigFont{18}{16.8}{\rmdefault}{\mddefault}{\itdefault}{$y+b_1$}%
}}}}}
\put(12191,-34199){\rotatebox{90.0}{\makebox(0,0)[lb]{\smash{{\SetFigFont{18}{16.8}{\rmdefault}{\mddefault}{\itdefault}{$y+z+a_2+b_2$}%
}}}}}
\put(16637,-34828){\rotatebox{60.0}{\makebox(0,0)[lb]{\smash{{\SetFigFont{18}{16.8}{\rmdefault}{\mddefault}{\itdefault}{$2y+z+2a_2+2b_2$}%
}}}}}
\put(16316,-27503){\rotatebox{300.0}{\makebox(0,0)[lb]{\smash{{\SetFigFont{18}{16.8}{\rmdefault}{\mddefault}{\itdefault}{$2y+2a_1+b_1$}%
}}}}}
\put(18274,-30812){\makebox(0,0)[lb]{\smash{{\SetFigFont{18}{16.8}{\rmdefault}{\mddefault}{\itdefault}{$z$}%
}}}}
\put(13493,-26126){\makebox(0,0)[lb]{\smash{{\SetFigFont{18}{16.8}{\rmdefault}{\mddefault}{\itdefault}{$x+a_2+b_2$}%
}}}}
\put(16350,-31090){\makebox(0,0)[lb]{\smash{{\SetFigFont{18}{16.8}{\rmdefault}{\mddefault}{\itdefault}{$b_2$}%
}}}}
\put(17195,-30558){\makebox(0,0)[lb]{\smash{{\SetFigFont{18}{16.8}{\rmdefault}{\mddefault}{\itdefault}{$b_1$}%
}}}}
\put(13902,-31022){\makebox(0,0)[lb]{\smash{{\SetFigFont{18}{16.8}{\rmdefault}{\mddefault}{\itdefault}{$a_2$}%
}}}}
\put(12191,-30012){\rotatebox{60.0}{\makebox(0,0)[lb]{\smash{{\SetFigFont{18}{16.8}{\rmdefault}{\mddefault}{\itdefault}{$2a_1$}%
}}}}}
\put(2963,-36936){\makebox(0,0)[lb]{\smash{{\SetFigFont{18}{16.8}{\rmdefault}{\mddefault}{\itdefault}{$x+a_1+b_1$}%
}}}}
\put(1627,-28258){\rotatebox{90.0}{\makebox(0,0)[lb]{\smash{{\SetFigFont{18}{16.8}{\rmdefault}{\mddefault}{\itdefault}{$y+b_1$}%
}}}}}
\put(1692,-6156){\rotatebox{60.0}{\makebox(0,0)[lb]{\smash{{\SetFigFont{18}{16.8}{\rmdefault}{\mddefault}{\itdefault}{$2a_1$}%
}}}}}
\put(3403,-7166){\makebox(0,0)[lb]{\smash{{\SetFigFont{18}{16.8}{\rmdefault}{\mddefault}{\itdefault}{$a_2$}%
}}}}
\put(6696,-6702){\makebox(0,0)[lb]{\smash{{\SetFigFont{18}{16.8}{\rmdefault}{\mddefault}{\itdefault}{$b_1$}%
}}}}
\put(5851,-7234){\makebox(0,0)[lb]{\smash{{\SetFigFont{18}{16.8}{\rmdefault}{\mddefault}{\itdefault}{$b_2$}%
}}}}
\put(2994,-2270){\makebox(0,0)[lb]{\smash{{\SetFigFont{18}{16.8}{\rmdefault}{\mddefault}{\itdefault}{$x+a_2+b_2$}%
}}}}
\put(7775,-6956){\makebox(0,0)[lb]{\smash{{\SetFigFont{18}{16.8}{\rmdefault}{\mddefault}{\itdefault}{$z$}%
}}}}
\put(5817,-3647){\rotatebox{300.0}{\makebox(0,0)[lb]{\smash{{\SetFigFont{18}{16.8}{\rmdefault}{\mddefault}{\itdefault}{$2y+2a_1+b_1$}%
}}}}}
\put(6138,-10972){\rotatebox{60.0}{\makebox(0,0)[lb]{\smash{{\SetFigFont{18}{16.8}{\rmdefault}{\mddefault}{\itdefault}{$2y+z+2a_2+2b_2$}%
}}}}}
\put(1692,-10343){\rotatebox{90.0}{\makebox(0,0)[lb]{\smash{{\SetFigFont{18}{16.8}{\rmdefault}{\mddefault}{\itdefault}{$y+z+a_2+b_2$}%
}}}}}
\put(1638,-4358){\rotatebox{90.0}{\makebox(0,0)[lb]{\smash{{\SetFigFont{18}{16.8}{\rmdefault}{\mddefault}{\itdefault}{$y+b_1$}%
}}}}}
\put(2974,-13036){\makebox(0,0)[lb]{\smash{{\SetFigFont{18}{16.8}{\rmdefault}{\mddefault}{\itdefault}{$x+a_1+b_1$}%
}}}}
\put(1681,-34243){\rotatebox{90.0}{\makebox(0,0)[lb]{\smash{{\SetFigFont{18}{16.8}{\rmdefault}{\mddefault}{\itdefault}{$y+z+a_2+b_2$}%
}}}}}
\put(12202,-6112){\rotatebox{60.0}{\makebox(0,0)[lb]{\smash{{\SetFigFont{18}{16.8}{\rmdefault}{\mddefault}{\itdefault}{$2a_1$}%
}}}}}
\put(13913,-7122){\makebox(0,0)[lb]{\smash{{\SetFigFont{18}{16.8}{\rmdefault}{\mddefault}{\itdefault}{$a_2$}%
}}}}
\put(17206,-6658){\makebox(0,0)[lb]{\smash{{\SetFigFont{18}{16.8}{\rmdefault}{\mddefault}{\itdefault}{$b_1$}%
}}}}
\put(16361,-7190){\makebox(0,0)[lb]{\smash{{\SetFigFont{18}{16.8}{\rmdefault}{\mddefault}{\itdefault}{$b_2$}%
}}}}
\put(13504,-2226){\makebox(0,0)[lb]{\smash{{\SetFigFont{18}{16.8}{\rmdefault}{\mddefault}{\itdefault}{$x+a_2+b_2$}%
}}}}
\put(18285,-6912){\makebox(0,0)[lb]{\smash{{\SetFigFont{18}{16.8}{\rmdefault}{\mddefault}{\itdefault}{$z$}%
}}}}
\put(16327,-3603){\rotatebox{300.0}{\makebox(0,0)[lb]{\smash{{\SetFigFont{18}{16.8}{\rmdefault}{\mddefault}{\itdefault}{$2y+2a_1+b_1$}%
}}}}}
\put(16648,-10928){\rotatebox{60.0}{\makebox(0,0)[lb]{\smash{{\SetFigFont{18}{16.8}{\rmdefault}{\mddefault}{\itdefault}{$2y+z+2a_2+2b_2$}%
}}}}}
\put(12202,-10299){\rotatebox{90.0}{\makebox(0,0)[lb]{\smash{{\SetFigFont{18}{16.8}{\rmdefault}{\mddefault}{\itdefault}{$y+z+a_2+b_2$}%
}}}}}
\put(12148,-4314){\rotatebox{90.0}{\makebox(0,0)[lb]{\smash{{\SetFigFont{18}{16.8}{\rmdefault}{\mddefault}{\itdefault}{$y+b_1$}%
}}}}}
\put(13484,-12992){\makebox(0,0)[lb]{\smash{{\SetFigFont{18}{16.8}{\rmdefault}{\mddefault}{\itdefault}{$x+a_1+b_1$}%
}}}}
\put(6127,-34872){\rotatebox{60.0}{\makebox(0,0)[lb]{\smash{{\SetFigFont{18}{16.8}{\rmdefault}{\mddefault}{\itdefault}{$2y+z+2a_2+2b_2$}%
}}}}}
\put(5806,-27547){\rotatebox{300.0}{\makebox(0,0)[lb]{\smash{{\SetFigFont{18}{16.8}{\rmdefault}{\mddefault}{\itdefault}{$2y+2a_1+b_1$}%
}}}}}
\put(7764,-30856){\makebox(0,0)[lb]{\smash{{\SetFigFont{18}{16.8}{\rmdefault}{\mddefault}{\itdefault}{$z$}%
}}}}
\put(2983,-26170){\makebox(0,0)[lb]{\smash{{\SetFigFont{18}{16.8}{\rmdefault}{\mddefault}{\itdefault}{$x+a_2+b_2$}%
}}}}
\put(1703,-18221){\rotatebox{60.0}{\makebox(0,0)[lb]{\smash{{\SetFigFont{18}{16.8}{\rmdefault}{\mddefault}{\itdefault}{$2a_1$}%
}}}}}
\put(3414,-19231){\makebox(0,0)[lb]{\smash{{\SetFigFont{18}{16.8}{\rmdefault}{\mddefault}{\itdefault}{$a_2$}%
}}}}
\put(6707,-18767){\makebox(0,0)[lb]{\smash{{\SetFigFont{18}{16.8}{\rmdefault}{\mddefault}{\itdefault}{$b_1$}%
}}}}
\put(5862,-19299){\makebox(0,0)[lb]{\smash{{\SetFigFont{18}{16.8}{\rmdefault}{\mddefault}{\itdefault}{$b_2$}%
}}}}
\put(3005,-14335){\makebox(0,0)[lb]{\smash{{\SetFigFont{18}{16.8}{\rmdefault}{\mddefault}{\itdefault}{$x+a_2+b_2$}%
}}}}
\put(7786,-19021){\makebox(0,0)[lb]{\smash{{\SetFigFont{18}{16.8}{\rmdefault}{\mddefault}{\itdefault}{$z$}%
}}}}
\put(5828,-15712){\rotatebox{300.0}{\makebox(0,0)[lb]{\smash{{\SetFigFont{18}{16.8}{\rmdefault}{\mddefault}{\itdefault}{$2y+2a_1+b_1$}%
}}}}}
\put(6149,-23037){\rotatebox{60.0}{\makebox(0,0)[lb]{\smash{{\SetFigFont{18}{16.8}{\rmdefault}{\mddefault}{\itdefault}{$2y+z+2a_2+2b_2$}%
}}}}}
\put(1703,-22408){\rotatebox{90.0}{\makebox(0,0)[lb]{\smash{{\SetFigFont{18}{16.8}{\rmdefault}{\mddefault}{\itdefault}{$y+z+a_2+b_2$}%
}}}}}
\put(1649,-16423){\rotatebox{90.0}{\makebox(0,0)[lb]{\smash{{\SetFigFont{18}{16.8}{\rmdefault}{\mddefault}{\itdefault}{$y+b_1$}%
}}}}}
\put(2985,-25101){\makebox(0,0)[lb]{\smash{{\SetFigFont{18}{16.8}{\rmdefault}{\mddefault}{\itdefault}{$x+a_1+b_1$}%
}}}}
\put(5840,-31134){\makebox(0,0)[lb]{\smash{{\SetFigFont{18}{16.8}{\rmdefault}{\mddefault}{\itdefault}{$b_2$}%
}}}}
\put(6685,-30602){\makebox(0,0)[lb]{\smash{{\SetFigFont{18}{16.8}{\rmdefault}{\mddefault}{\itdefault}{$b_1$}%
}}}}
\put(12213,-18177){\rotatebox{60.0}{\makebox(0,0)[lb]{\smash{{\SetFigFont{18}{16.8}{\rmdefault}{\mddefault}{\itdefault}{$2a_1$}%
}}}}}
\put(13924,-19187){\makebox(0,0)[lb]{\smash{{\SetFigFont{18}{16.8}{\rmdefault}{\mddefault}{\itdefault}{$a_2$}%
}}}}
\put(17217,-18723){\makebox(0,0)[lb]{\smash{{\SetFigFont{18}{16.8}{\rmdefault}{\mddefault}{\itdefault}{$b_1$}%
}}}}
\put(16372,-19255){\makebox(0,0)[lb]{\smash{{\SetFigFont{18}{16.8}{\rmdefault}{\mddefault}{\itdefault}{$b_2$}%
}}}}
\put(13515,-14291){\makebox(0,0)[lb]{\smash{{\SetFigFont{18}{16.8}{\rmdefault}{\mddefault}{\itdefault}{$x+a_2+b_2$}%
}}}}
\put(18296,-18977){\makebox(0,0)[lb]{\smash{{\SetFigFont{18}{16.8}{\rmdefault}{\mddefault}{\itdefault}{$z$}%
}}}}
\put(16338,-15668){\rotatebox{300.0}{\makebox(0,0)[lb]{\smash{{\SetFigFont{18}{16.8}{\rmdefault}{\mddefault}{\itdefault}{$2y+2a_1+b_1$}%
}}}}}
\put(16659,-22993){\rotatebox{60.0}{\makebox(0,0)[lb]{\smash{{\SetFigFont{18}{16.8}{\rmdefault}{\mddefault}{\itdefault}{$2y+z+2a_2+2b_2$}%
}}}}}
\put(12213,-22364){\rotatebox{90.0}{\makebox(0,0)[lb]{\smash{{\SetFigFont{18}{16.8}{\rmdefault}{\mddefault}{\itdefault}{$y+z+a_2+b_2$}%
}}}}}
\put(12159,-16379){\rotatebox{90.0}{\makebox(0,0)[lb]{\smash{{\SetFigFont{18}{16.8}{\rmdefault}{\mddefault}{\itdefault}{$y+b_1$}%
}}}}}
\put(13495,-25057){\makebox(0,0)[lb]{\smash{{\SetFigFont{18}{16.8}{\rmdefault}{\mddefault}{\itdefault}{$x+a_1+b_1$}%
}}}}
\put(3392,-31066){\makebox(0,0)[lb]{\smash{{\SetFigFont{14}{16.8}{\rmdefault}{\mddefault}{\itdefault}{$a_2$}%
}}}}
\put(1681,-30056){\rotatebox{60.0}{\makebox(0,0)[lb]{\smash{{\SetFigFont{18}{16.8}{\rmdefault}{\mddefault}{\itdefault}{$2a_1$}%
}}}}}
\put(14715,-24389){\makebox(0,0)[lb]{\smash{{\SetFigFont{20}{24.0}{\rmdefault}{\mddefault}{\itdefault}{$s$}%
}}}}
\put(15083,-24143){\makebox(0,0)[lb]{\smash{{\SetFigFont{20}{24.0}{\rmdefault}{\mddefault}{\itdefault}{$w$}%
}}}}
\put(5228,-19128){\makebox(0,0)[lb]{\smash{{\SetFigFont{20}{24.0}{\rmdefault}{\mddefault}{\itdefault}{$v$}%
}}}}
\put(4594,-14969){\makebox(0,0)[lb]{\smash{{\SetFigFont{20}{24.0}{\rmdefault}{\mddefault}{\itdefault}{$u$}%
}}}}
\put(14704,-12324){\makebox(0,0)[lb]{\smash{{\SetFigFont{20}{24.0}{\rmdefault}{\mddefault}{\itdefault}{$s$}%
}}}}
\put(15072,-12078){\makebox(0,0)[lb]{\smash{{\SetFigFont{20}{24.0}{\rmdefault}{\mddefault}{\itdefault}{$w$}%
}}}}
\put(15727,-7019){\makebox(0,0)[lb]{\smash{{\SetFigFont{20}{24.0}{\rmdefault}{\mddefault}{\itdefault}{$v$}%
}}}}
\put(15093,-2860){\makebox(0,0)[lb]{\smash{{\SetFigFont{20}{24.0}{\rmdefault}{\mddefault}{\itdefault}{$u$}%
}}}}
\put(15061,-35978){\makebox(0,0)[lb]{\smash{{\SetFigFont{20}{24.0}{\rmdefault}{\mddefault}{\itdefault}{$w$}%
}}}}
\put(15716,-30919){\makebox(0,0)[lb]{\smash{{\SetFigFont{20}{24.0}{\rmdefault}{\mddefault}{\itdefault}{$v$}%
}}}}
\put(4572,-26804){\makebox(0,0)[lb]{\smash{{\SetFigFont{20}{24.0}{\rmdefault}{\mddefault}{\itdefault}{$u$}%
}}}}
\put(4183,-36268){\makebox(0,0)[lb]{\smash{{\SetFigFont{20}{24.0}{\rmdefault}{\mddefault}{\itdefault}{$s$}%
}}}}
\end{picture}%
}
\caption{Obtaining a recurrence for the number of tilings of the $H^{(1)}$-type regions.}\label{KuoH1b}
\end{figure}

\medskip

For induction step, we assume that $x,y,b+\overline{n}$ are all positive and that the our tiling formulas (\ref{main1eqn}) and (\ref{mainR1eqn}) hold respectively for any $H^{(1)}$-type and $R^{(1)}$-type regions
with the sum of the $y$-, $z$-, $b$- and $\overline{n}$-parameters strictly less than $y+z+b+\overline{n}$.
Before obtaining the recurrences for the $H^{(1)}$-type and $R^{(1)}$-type regions, we have two important notices as follows.

We can assume that all the terms in the sequence $\textbf{b}$ are positive. Indeed, if an even number of initial terms in the sequence $\textbf{b}$ are equal to $0$, say $b_1=b_2=\dotsc=b_{2l}$,
 then can simply eliminate this $0$ terms from the sequence $\textbf{b}$.  If  $b_1=0$ and $b_{2}>0$,
 by removing forced lozenges along the northeast side, we get a new $H^{(1)}$-type (resp., $R^{(1)}$-type) region with fewer holes
(see Figure \ref{fig:halvedhexbase3}(a) for the case of $H^{(1)}$-type regions, the case of $R^{(1)}$-type regions can be treated similarly).
 Finally, if we have $b_1>0$, and some middle terms
equal to $0$, say $b_{i}=\dots=b_{i+l}=0$ and $b_{i-1},b_{i+l+1}>0$, then we can remove several forced lozenges and combine the $b_{i-1}$- and the $b_{i+l+1}$-triangles in the $b$-fern into a triangle of side length $b_{i-1}+b_{i+l+1}$ as in Figure \ref{fig:halvedhexbase3}(b) (for a $H^{(1)}$-type region, the case of $R^{(1)}$-type regions can be treated in the same manner)
 to get a new region with the sum of the four parameter strictly less than $ y+z+b+\overline{n}$. In the rest of the proof, we assume that all $b_i>0$, i.e. $n=\overline{n}$.

We can assume further that $z>0$. Indeed, if $z=0$ (and $b_1>0$ by the above arguments), we remove forced lozenges along the southeast side of an $H^{(1)}$-type region
 to get a new $R^{(1)}$-type  region with the sum of the four parameters smaller (see Figure \ref{fig:halvedhexbase3}(c)). Similarly, when $z=0$, we can obtain a `smaller' $H^{(1)}$-type
region from the $R^{(1)}$-type region by removing forced lozenges as in Figure \ref{fig:halvedhexbase3}(d).

With the above two assumptions, we apply Kuo's  Condensation Theorem \ref{kuothm} to the dual graph $G$ of the region $H^{(1)}_{x,y,z}(\textbf{a};\textbf{b})$,
with the four vertices $u,v,w,s$ chosen as in Figure \ref{KuoH1}. In particular, the unit triangles corresponding the vertices $u,v,w,s$ are the shaded unit triangles with the same label.
The $u$-triangle is the up-pointing shaded unit triangle on  the upper-right corner of the region, and the $v$-triangle is the down-pointing shaded unit triangle adjacent to the left most of the $b$fern. The $w$- and $s$-triangles form a shaded bowtie on the lower-right corner.

First, we consider the region corresponding to the graph $G-\{u,v,w,s\}$. It is the region $H^{(1)}_{x,y,z}(a,b)$ with all the four $u$-, $v$-, $w$-, $s$-triangles removed.
The removal of the unit triangles yields several forced lozenges. By removing these forced lozenges, we get back the region $H^{(1)}_{x,y-1,z-1}(\textbf{a};\textbf{b}^{+1})$, where $\textbf{b}^{+1}$
denotes the sequence obtained from the sequence $\textbf{b}$ by adding $1$ to the last term if $\textbf{b}$ has an even number of terms, and by including a new term $1$ if $\textbf{b}$ has an odd number of terms
 (see the region restricted by the bold contour in Figure \ref{KuoH1b}(b) for the case when $\textbf{b}$ has an even number of terms; in the case $\textbf{b}$ has an odd number of terms the removal of the $v$-triangle forms a new down-pointing triangle of side-length $1$ at the end of the $b$-fern). Since the removal of these forced lozenges (with all weights 1) does not change the number of tilings, we have
\begin{equation}\label{proof1eq1}
\M(G-\{u,v,w,s\})=\M(H^{(1)}_{x,y-1,z-1}(\textbf{a};\textbf{b}^{+1})).
\end{equation}
Similarly, by considering forced lozenges yielded by the removal of the back unit triangles as in Figures \ref{KuoH1b}(c)--(f), respectively, we obtain:
\begin{equation}\label{proof1eq2}
\M(G-\{u,v\})=\M(H^{(1)}_{x,y-1,z}(\textbf{a};\textbf{b}^{+1})),
\end{equation}
\begin{equation}\label{proof1eq3}
\M(G-\{w,s\})=\M(H^{(1)}_{x,y,z-1}(\textbf{a};\textbf{b})),
\end{equation}
\begin{equation}\label{proof1eq4}
\M(G-\{u,v,w,s\})=\M(H^{(1)}_{x+1,y-1,z}(\textbf{a};\textbf{b})),
\end{equation}
and
\begin{equation}\label{proof1eq5}
\M(G-\{u,v,w,s\})=\M(H^{(1)}_{x-1,y,z-1}(\textbf{a};\textbf{b}^{+1})
\end{equation}
Plugging (\ref{proof1eq1})--(\ref{proof1eq5}) into the equation in Kuo's Theorem \ref{kuothm}, we have the following recurrence:
\begin{align}\label{recurrence1}
\M(H^{(1)}_{x,y,z}(\textbf{a};\textbf{b}))\M(H^{(1)}_{x,y-1,z-1}(\textbf{a};\textbf{b}^{+1}))=&\M(H^{(1)}_{x,y-1,z}(\textbf{a};\textbf{b}^{+1}))\M(H^{(1)}_{x,y,z-1}(\textbf{a};\textbf{b}))\notag\\
&+\M(H^{(1)}_{x+1,y-1,z}(\textbf{a};\textbf{b}))\M(H^{(1)}_{x-1,y,z-1}(\textbf{a};\textbf{b}^{+1})).
\end{align}
Working similarly on the region $R^{(1)}_{x,y,z}(\textbf{a}; \textbf{b})$, we get the same recurrence for $R^{(1)}$-type regions:
\begin{align}\label{recurrence2}
\M(R^{(1)}_{x,y,z}(\textbf{a};\textbf{b}))\M(R^{(1)}_{x,y-1,z-1}(\textbf{a};\textbf{b}^{+1}))=&\M(R^{(1)}_{x,y-1,z}(\textbf{a};\textbf{b}^{+1}))\M(R^{(1)}_{x,y,z-1}(\textbf{a};\textbf{b}))\notag\\
&+\M(R^{(1)}_{x+1,y-1,z}(\textbf{a};\textbf{b}))\M(R^{(1)}_{x-1,y,z-1}(\textbf{a};\textbf{b}^{+1})).
\end{align}

Next, we show that the formulas on the right-hand sides of  (\ref{main1eqn}) and  (\ref{mainR1eqn}), denoted by  $\phi_{x,y,z}(\textbf{a};\textbf{b})$ and $\psi_{x,y,z}(\textbf{a};\textbf{b})$, respectively,
satisfy the same recurrence above.
Equivalently, we would like to verify that
\begin{align}\label{recurrence1refined}
\frac{\phi_{x,y,z-1}(\textbf{a};\textbf{b})}{\phi_{x,y,z}(\textbf{a};\textbf{b})}\cdot \frac{\phi_{x,y-2,z}(\textbf{a};\textbf{b}^{+1})}{\phi_{x,y-2,z-1}(\textbf{a};\textbf{b}^{+1})}
+\frac{\phi_{x+1,y-2,z}(\textbf{a};\textbf{b})}{\phi_{x,y,z}(\textbf{a};\textbf{b})}\cdot \frac{\phi_{x-1,y,z-1}(\textbf{a};\textbf{b}^{+1})}{\phi_{x,y-2,z-1}(\textbf{a};\textbf{b}^{+1})}=1.
\end{align}
and that
 \begin{align}\label{recurrenceR1refined}
\frac{\psi_{x,y,z-1}(\textbf{a};\textbf{b})}{\psi_{x,y,z}(\textbf{a};\textbf{b})}\cdot \frac{\psi_{x,y-2,z}(\textbf{a};\textbf{b}^{+1})}{\psi_{x,y-2,z-1}(\textbf{a};\textbf{b}^{+1})}
+\frac{\psi_{x+1,y-2,z}(\textbf{a};\textbf{b})}{\psi_{x,y,z}(\textbf{a};\textbf{b})}\cdot \frac{\psi_{x-1,y,z-1}(\textbf{a};\textbf{b}^{+1})}{\psi_{x,y-2,z-1}(\textbf{a};\textbf{b}^{+1})}=1.
\end{align}
We only present here the verification for the case when $m$ and $n$ are both even, as the other cases can be handled in a completely analogous manner.

By Lemma \ref{TV}, we can simplify the first fraction on the left-hand side of (\ref{recurrence1refined}) as
\begin{align}
\frac{\phi_{x,y,z}(\textbf{a};\textbf{b})}{\phi_{x,y,z-1}(\textbf{a};\textbf{b})}&
=\frac{\Q(a_1,\dotsc,a_{m},x+y,b_{n},\dotsc,b_1,z)}{\Q(a_1,\dotsc,a_{m},x+y,b_{n},\dotsc,b_1,z-1)}\notag\\
&\times\frac{(2\e_a+2\e_b+2y+2z-1)!}{(2\e_a+2\e_b+2z-1)!}\frac{(2a+b+y+z)!}{(2a+b+2y+z)!}\frac{(b+z-1)!}{(b+y+z-1)!}\notag\\
&\times \frac{(2a+b+y+z+1)_{y}}{(x+2a+b+y+z+1)_{y}}
 \frac{[2a+2b+2y+2z+1]_{y}}{[2x+2a+2b+2y+2z+1]_{y}}.
\end{align}
Similarly, we get for the second fraction simplified as
\begin{align}
\frac{\phi_{x,y-2,z}(\textbf{a};\textbf{b}^{+1})}{\phi_{x,y-2,z-1}(\textbf{a};\textbf{b}^{+1})}&
=\frac{\Q(a_1,\dotsc,a_{m},x+y-1,b_{n}+1,\dotsc,b_1,z)}{\Q(a_1,\dotsc,a_{m},x+y-1,b_{n}+1,\dotsc,b_1,z-1)}\notag\\
&\times \frac{(2\e_a+2\e_b+2z+1)!}{(2\e_a+2\e_b+2y+2z-1)!}\frac{(2a+b+2y+z-1)!}{(2a+b+y+z)!}\frac{(b+y+z-1)!}{(b+z)!}\notag\\
&\times \frac{(x+2a+b+y+z+1)_{y-1}}{(2a+b+y+z+1)_{y-1}} \frac{[2x+2a+2b+2y+2z+1]_{y-1}}{[2a+2b+2y+2z+1]_{y-1}}.
\end{align}
From the above simplification and Lemma \ref{QK}, we have the first term on the left-hand side written by:

\begin{align}
&\frac{\phi_{x,y,z-1}(\textbf{a};\textbf{b})}{\phi_{x,y,z}(\textbf{a};\textbf{b})}\frac{\phi_{x,y-2,z}(\textbf{a};\textbf{b}^{+1})}{\phi_{x,y-2,z-1}(\textbf{a};\textbf{b}^{+1})}
=\frac{(2x+2a+b+2y+z)(2a+2b+4y+2z-1)}{(x+2a+b+2y+z)(2x+2a+2b+4y+2z-1)}.
\end{align}

Next, we consider the third and the fourth fractions on the left-hand side of (\ref{recurrence1refined}). We note that the $\Q$-factors cancel out here. By Lemma \ref{TV}, we can simplify
%
%
\begin{align}
&\frac{\phi_{x+1,y-2,z}(\textbf{a};\textbf{b})}{\phi_{x,y,z}(\textbf{a};\textbf{b})}\frac{\phi_{x-1,y,z-1}(\textbf{a};\textbf{b}^{+1})}{\phi_{x,y-2,z-1}(\textbf{a};\textbf{b}^{+1})}
=\frac{x(2x+2a+1)}{(x+2a+b+2y+z)(2x+2a+2b+4y+2z-1)}.
\end{align}
Therefore, (\ref{recurrence1refined}) is now equivalent to
\begin{align}
&\frac{(2x+2a+b+2y+z)(2a+2b+4y+2z-1)}{(x+2a+b+2y+z)(2x+2a+2b+4y+2z-1)}\notag\\&\qquad\qquad\qquad+\frac{x(2x+2a+1)}{(x+2a+b+2y+z)(2x+2a+2b+4y+2z-1)}=1,
\end{align}
which is a true identity.

Similarly, by using Lemmas \ref{TV} and \ref{QK}, we can simplify the terms on the left-hand side of (\ref{recurrenceR1refined}) as:
\begin{align}
\frac{\psi_{x,y,z-1} (\textbf{a};\textbf{b})}{\psi_{x,y,z}(\textbf{a};\textbf{b})}&\cdot \frac{\psi_{x,y-2,z}(\textbf{a};\textbf{b}^{+1})}{\psi_{x,y-2,z-1}(\textbf{a};\textbf{b}^{+1})}
=\frac{(2x+2a+b+2y+z)(2a+2b+4y+2z-1)}{(x+2a+b+2y+z)(2x+2a+2b+4y+2z-1)}
\end{align}
and
\begin{align}
\frac{\psi_{x+1,y-2,z}(\textbf{a};\textbf{b})}{\psi_{x,y,z}(\textbf{a};\textbf{b})}\cdot&\frac{\psi_{x-1,y,z-1}(\textbf{a};\textbf{b}^{+1})}{\psi_{x,y-2,z-1}(\textbf{a};\textbf{b}^{+1})}
=\frac{x(2x+2a+1)}{(x+2a+b+2y+z)(2x+2a+2b+4y+2z-1)}.
\end{align}
This means that (\ref{recurrenceR1refined}) now becomes the true identity:
\begin{align}
&\frac{(2x+2a+b+2y+z)(2a+2b+4y+2z-1)}{(x+2a+b+2y+z)(2x+2a+2b+4y+2z-1)}\notag\\ &\qquad\qquad+\frac{x(2x+2a+1)}{(x+2a+b+2y+z)(2x+2a+2b+4y+2z-1)}=1.
\end{align}

\medskip

To complete our proof we need to show that the second and the third expressions in each of (\ref{main1eq}) and (\ref{mainR1eq}) are equal. In particular, we need to show that
\begin{align}\label{main1eqx}
&\frac{\M(H^{(1)}_{x+y,0,z}(\textbf{a},\textbf{b}))\M(H^{(1)}_{0,2y,z}(\textbf{a},\textbf{b}))}{\M(H^{(1)}_{y,0,z}(\textbf{a},\textbf{b}))}\notag\\
&=2^{-y}\Q(0,a_1,\dotsc,a_{2\lfloor\frac{m+1}{2}\rfloor-1},a_{2\lfloor\frac{m+1}{2}\rfloor}+x+y+b_{2\lfloor\frac{n+1}{2}\rfloor},b_{2\lfloor\frac{n+1}{2}\rfloor-1},\dotsc,b_1)\notag\\
&\quad\times \Q(a_1,\dotsc,a_{\lceil \frac{m-1}{2}\rceil}, a_{\lceil \frac{m-1}{2}\rceil+1}+x+y+b_{\lceil \frac{n-1}{2}\rceil+1},b_{\lceil \frac{n-1}{2}\rceil},\dotsc, b_1,z) \notag\\
&\quad\times \frac{\Hf_2(2\od_a+2\od_b+1)\Hf_2(2\e_a+2\e_b+2z+1)}{\Hf_2(2\od_a+2\od_b+2y+1)\Hf_2(2\e_a+2\e_b+2y+2z+1)} \notag\\
&\quad\times \frac{\Hf(2a+b+2y+z+1)\Hf(b+y+z)}{\Hf(2a+b+y+z+1)\Hf(b+z)}.
\end{align}
and that
\begin{align}\label{mainR1eqx}
&\frac{\M(R^{(1)}_{x+y,0,z}(\textbf{a},\textbf{b}))\M(R^{(1)}_{0,2y,z}(\textbf{a},\textbf{b}))}{\M(R^{(1)}_{y,0,z}(\textbf{a},\textbf{b}))}\notag\\
&=2^{-y}\Q(a_1,\dotsc,a_{2\lceil \frac{m-1}{2}\rceil}, a_{2\lceil \frac{m-1}{2}\rceil+1}+x+y+b_{2\lfloor\frac{n+1}{2}\rfloor},b_{2\lfloor\frac{n+1}{2}\rfloor-1},\dotsc,b_1)\notag\\
&\quad\times \Q(0,a_1,\dotsc,a_{2\lfloor\frac{m+1}{2}\rfloor-1},a_{2\lfloor\frac{m+1}{2}\rfloor}+x+y+b_{2\lceil \frac{n-1}{2}\rceil+1},b_{2\lceil \frac{n-1}{2}\rceil},\dotsc, b_1,z) \notag\\
&\quad\times \frac{\Hf_2(2\e_a+2\od_b+1)\Hf_2(2\od_a+2\e_b+2z+1)}{\Hf_2(2\e_a+2\od_b+2y+1)\Hf_2(2\od_a+2\e_b+2y+2z+1)} \notag\\
&\quad\times \frac{\Hf(2a+b+2y+z+1)\Hf(b+y+z)}{\Hf(2a+b+y+z+1)\Hf(b+z)}.
\end{align}

Let us consider only (\ref{main1eqx}), as (\ref{mainR1eqx}) can be treated in the same way.
Similar to the base case $y=0$ treated above, by Region-splitting Lemma \ref{RS}, we have the product of the two $\Q$-factors on the
right-hand side is exactly $\M(H^{(1)}_{x+y,0,z}(\textbf{a},\textbf{b}))$. Moreover,
by  Region-splitting Lemma again, we can write each of $\M(H^{(1)}_{0,2y,z}(\textbf{a},\textbf{b}))$ and $\M(H^{(1)}_{y,0,z}(\textbf{a},\textbf{b}))$ as a product of the numbers of tilings of the two
$\mathcal{Q}$-type regions, and (\ref{main1eqx}) follows by performing a straightforward simplification using Lemmas \ref{QAR} and \ref{QK}.
\end{proof}

\begin{proof}[Combined proof of Theorems \ref{mainM1} and \ref{mainMR1}]
Similar to the combined proof of Theorems \ref{main1} and \ref{mainR1}, the second and the third expressions in (\ref{mainM1eq}) and  in  (\ref{mainMR1eq}) are equal by  Region-splitting Lemma \ref{RS} and by performing a straightforward simplification
using Lemmas \ref{QAR} and \ref{QK}.

We only need to show that
\begin{align}\label{mainM1eqn}
\M(N^{(1)}_{x,y,z}(\textbf{a},\textbf{b}))
&=2^{a_1-y}\K'(0,a_1+1,a_2,\dotsc,a_{2\lfloor\frac{m+1}{2}\rfloor-1},a_{2\lfloor\frac{m+1}{2}\rfloor}+x+y+b_{2\lfloor\frac{n+1}{2}\rfloor},b_{2\lfloor\frac{n+1}{2}\rfloor-1},\dotsc,b_1)\notag\\
&\quad\times \Q(a_1,\dotsc,a_{\lceil \frac{m-1}{2}\rceil}, a_{\lceil \frac{m-1}{2}\rceil+1}+x+y+b_{\lceil \frac{n-1}{2}\rceil+1},b_{\lceil \frac{n-1}{2}\rceil},\dotsc, b_1,z) \notag\\
&\quad\times \frac{(a+y)!}{a!} \frac{\Hf_2(2\od_a+2\od_b+2)\Hf_2(2\e_a+2\e_b+2z+1)}{\Hf_2(2\od_a+2\od_b+2y+2)
\Hf_2(2\e_a+2\e_b+2y+2z+1)} \notag\\
&\quad\times \frac{\Hf(2a+b+2y+z+1)\Hf(b+y+z)}{\Hf(2a+b+y+z+1)\Hf(b+z)}\notag\\
&\quad\times   \frac{\T(x+1,2a+b+2y+z,y)\T(x+a+1,b+2y+z,y)}{\T(1,2a+b+2y+z,y)\T(a+1,b+2y+z,y)}.
\end{align}
and that
\begin{align}\label{mainMR1eqn}
\M(NR^{(1)}_{x,y,z}(\textbf{a},\textbf{b}))
&=2^{a_1-y}\Q(a_1,\dotsc,a_{2\lceil \frac{m-1}{2}\rceil}, a_{2\lceil \frac{m-1}{2}\rceil+1}+x+y+b_{2\lfloor\frac{n+1}{2}\rfloor},b_{2\lfloor\frac{n+1}{2}\rfloor-1},\dotsc,b_1)\notag\\
&\quad\times \K'(0,a_1+1,a_2,\dotsc,a_{2\lfloor\frac{m+1}{2}\rfloor-1},a_{2\lfloor\frac{m+1}{2}\rfloor}+x+y+b_{2\lceil \frac{n-1}{2}\rceil+1},b_{2\lceil \frac{n-1}{2}\rceil},\dotsc, b_1,z) \notag\\
&\quad\times\frac{(a+y)!}{a!} \frac{\Hf_2(2\e_a+2\od_b+1)\Hf_2(2\od_a+2\e_b+2z+2)}{\Hf_2(2\e_a+2\od_b+2y+1)
\Hf_2(2\od_a+2\e_b+2y+2z+2)} \notag\\
&\quad\times \frac{\Hf(2a+b+2y+z+1)\Hf(b+y+z)}{\Hf(2a+b+y+z+1)\Hf(b+z)}\notag\\
&\quad\times   \frac{\T(x+1,2a+b+2y+z,y)\T(x+a+1,b+2y+z,y)}{\T(1,2a+b+2y+z,y)\T(a+1,b+2y+z,y)}.
\end{align}
by induction on $y+z+b+\overline{n}$.

The base cases are still the cases $x=0$, $y=0$, and $b+\overline{n}=0$. While cases $x=0$ and $y=0$ follow from Lemmas \ref{RS} and \ref{QAR}, the case $b+\overline{n}=0$ was again already treated in \cite{Halfhex2}.

The induction step follows the lines in the proof of Theorems \ref{main1} and \ref{mainR1}. Without loss of generality, we can assume that $n=\overline{n}$ and that $z>0$.
Next, we apply Kuo Condensation with the four unit triangles corresponding the four vertices $u,v,w,s$  chosen similarly to that in Figure \ref{KuoH1} for the $H^{(1)}$-type region.
By considering forced lozenges yielded from the removal of these unit triangles, we have the following recurrences for the $N^{(1)}$-type and the $RN^{(1)}$-type regions:
\begin{align}\label{recurrenceM1}
\M(N^{1}_{x,y,z}(\textbf{a},\textbf{b}))\M(N^{1}_{x,y-2,z-1}(\textbf{a},\textbf{b}^{+1}))=&\M(N^{1}_{x,y-2,z}(\textbf{a},\textbf{b}^{+1}))\M(N^{1}_{x,y,z-1}(\textbf{a},\textbf{b}))\notag\\
&+\M(N^{1}_{x+1,y-2,z}(\textbf{a},\textbf{b}))\M(N^{1}_{x-1,y,z-1}(\textbf{a},\textbf{b}^{+1})),
\end{align}
\begin{align}\label{recurrenceMR1}
\M(RN^{1}_{x,y,z}(\textbf{a},\textbf{b}))\M(RN^{1}_{x,y-2,z-}(\textbf{a},\textbf{b}^{+1}))=&\M(RN^{1}_{x,y-2,z}(\textbf{a},\textbf{b}^{+1}))\M(RN^{1}_{x,y,z-1}(\textbf{a},\textbf{b}))\notag\\
&+\M(RN^{1}_{x+1,y-2,z}(\textbf{a},\textbf{b}))\M(RN^{1}_{x-1,y,z-1}(\textbf{a},\textbf{b}^{+1})).
\end{align}
We now only need to show that the expressions on the right-hand sides of (\ref{mainM1eqn}) and (\ref{mainMR1eqn}), denoted by $f_{x,y,z}(\textbf{a},\textbf{b})$ and $g_{x,y,z}(\textbf{a},\textbf{b})$, respectively, satisfy the same recurrence.
It is equivalent to show that
\begin{align}\label{recurrenceM1refined}
\frac{f_{x,y,z-1}(\textbf{a},\textbf{b})}{f_{x,y,z}(\textbf{a},\textbf{b})}\cdot \frac{f_{x,y-2,z}(\textbf{a},\textbf{b}^{+1})}{f_{x,y-2,z-1}(\textbf{a},\textbf{b}^{+1})}+
\frac{f_{x+1,y-2,z}(\textbf{a},\textbf{b})}{f_{x,y,z}(\textbf{a},\textbf{b})}\cdot \frac{f_{x-1,y,z-1}(\textbf{a},\textbf{b}^{+1})}{f_{x,y-2,z-1}(\textbf{a},\textbf{b}^{+1})}=1.
\end{align}
and that
\begin{align}\label{recurrenceMR1refined}
\frac{g_{x,y,z-1}(\textbf{a},\textbf{b})}{g_{x,y,z}(\textbf{a},\textbf{b})}\cdot \frac{g_{x,y-2,z}(\textbf{a},\textbf{b}^{+1})}{g_{x,y-2,z-1}(\textbf{a},\textbf{b}^{+1})}+
\frac{g_{x+1,y-2,z}(\textbf{a},\textbf{b})}{g_{x,y,z}(\textbf{a},\textbf{b})}\cdot \frac{g_{x-1,y,z-1}(\textbf{a},\textbf{b}^{+1})}{g_{x,y-2,z-1}(\textbf{a},\textbf{b}^{+1})}=1.
\end{align}

Performing a simplification using Lemmas \ref{QK} and \ref{TV},  (\ref{recurrenceM1refined}) and (\ref{recurrenceM1refined}) are both reduced to the following true identity
\begin{align}
\frac{(2x+2a+b+2y+z)(a+b+2y+z)}{(x+2a+b+2y+z)(x+a+b+2y+z)}+\frac{x(x+a)}{(x+2a+b+2y+z)(x+a+b+2y+z)}=1.
\end{align}
This finishes the proof.
\end{proof}

The combined proofs of Theorems \ref{main2} and \ref{mainR2}, Theorems \ref{mainW1} and \ref{mainRW1}, Theorems \ref{mainW2} and \ref{mainRW2}, Theorems \ref{mainM2} and \ref{mainMR2},
Theorems \ref{mainM3} and \ref{mainMR3}, and Theorems \ref{mainM4} and \ref{mainMR4} are similar and omitted.

\begin{figure}\centering
\setlength{\unitlength}{3947sp}%
\begingroup\makeatletter\ifx\SetFigFont\undefined%
\gdef\SetFigFont#1#2#3#4#5{%
  \reset@font\fontsize{#1}{#2pt}%
  \fontfamily{#3}\fontseries{#4}\fontshape{#5}%
  \selectfont}%
\fi\endgroup%
\resizebox{15cm}{!}{
\begin{picture}(0,0)%
\includegraphics{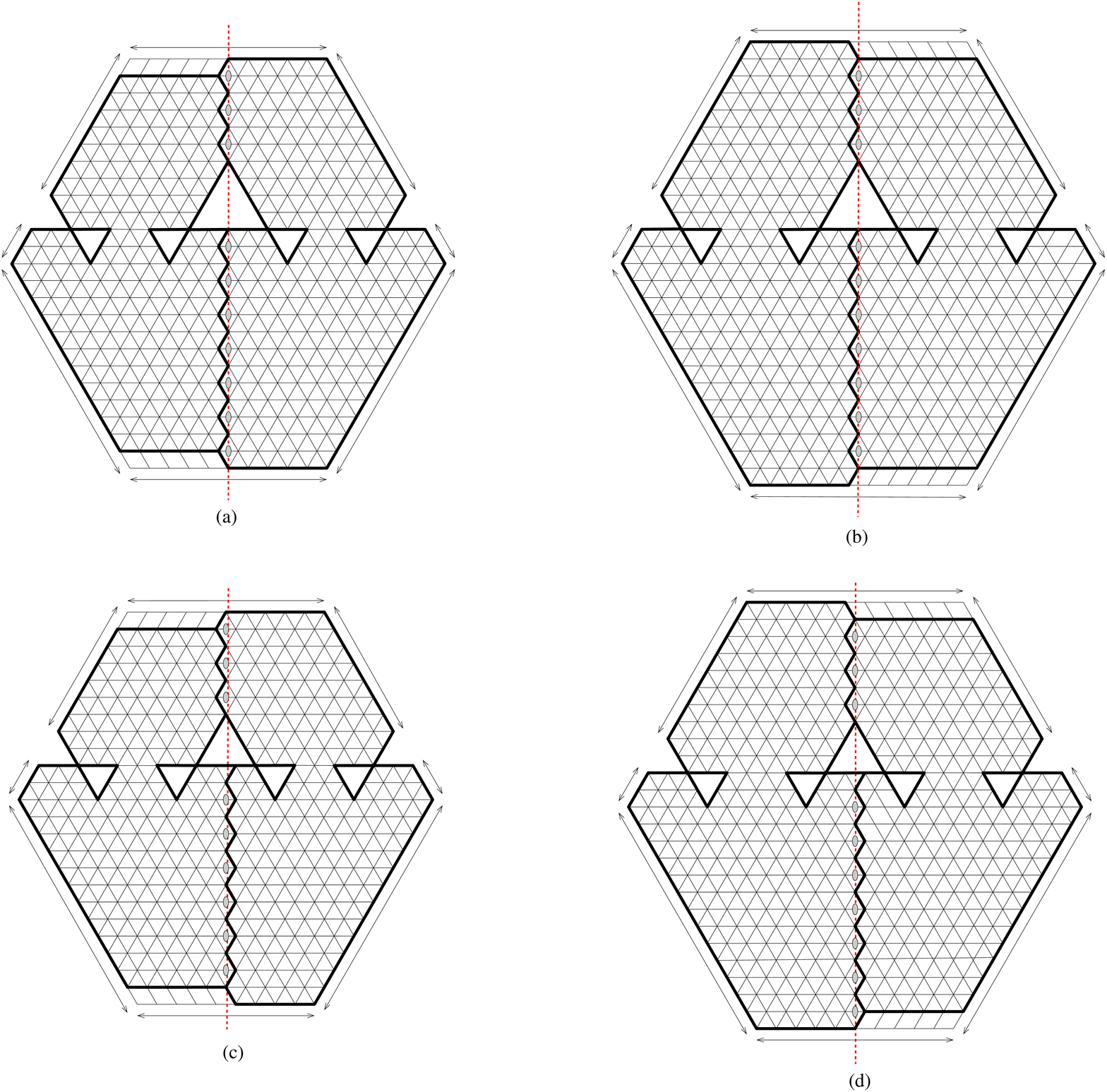}%
\end{picture}%
%
%

\begin{picture}(23176,22553)(2614,-24673)
\put(16607,-4828){\rotatebox{60.0}{\makebox(0,0)[lb]{\smash{{\SetFigFont{17}{20.4}{\rmdefault}{\mddefault}{\updefault}{$y+a_1+b_1$}%
}}}}}
\put(23809,-3897){\rotatebox{300.0}{\makebox(0,0)[lb]{\smash{{\SetFigFont{17}{20.4}{\rmdefault}{\mddefault}{\updefault}{$y+a_1+b_1$}%
}}}}}
\put(17087,-7066){\makebox(0,0)[lb]{\smash{{\SetFigFont{17}{20.4}{\rmdefault}{\mddefault}{\updefault}{$b_2$}%
}}}}
\put(23627,-7059){\makebox(0,0)[lb]{\smash{{\SetFigFont{17}{20.4}{\rmdefault}{\mddefault}{\updefault}{$b_2$}%
}}}}
\put(23348,-18378){\makebox(0,0)[lb]{\smash{{\SetFigFont{17}{20.4}{\rmdefault}{\mddefault}{\updefault}{$b_2$}%
}}}}
\put(17213,-18378){\makebox(0,0)[lb]{\smash{{\SetFigFont{17}{20.4}{\rmdefault}{\mddefault}{\updefault}{$b_2$}%
}}}}
\put(9883,-18232){\makebox(0,0)[lb]{\smash{{\SetFigFont{17}{20.4}{\rmdefault}{\mddefault}{\updefault}{$b_2$}%
}}}}
\put(4543,-18210){\makebox(0,0)[lb]{\smash{{\SetFigFont{17}{20.4}{\rmdefault}{\mddefault}{\updefault}{$b_2$}%
}}}}
\put(3568,-17692){\makebox(0,0)[lb]{\smash{{\SetFigFont{17}{20.4}{\rmdefault}{\mddefault}{\updefault}{$b_1$}%
}}}}
\put(10738,-17700){\makebox(0,0)[lb]{\smash{{\SetFigFont{17}{20.4}{\rmdefault}{\mddefault}{\updefault}{$b_1$}%
}}}}
\put(16276,-17883){\makebox(0,0)[lb]{\smash{{\SetFigFont{17}{20.4}{\rmdefault}{\mddefault}{\updefault}{$b_1$}%
}}}}
\put(24233,-17853){\makebox(0,0)[lb]{\smash{{\SetFigFont{17}{20.4}{\rmdefault}{\mddefault}{\updefault}{$b_1$}%
}}}}
\put(24534,-6564){\makebox(0,0)[lb]{\smash{{\SetFigFont{17}{20.4}{\rmdefault}{\mddefault}{\updefault}{$b_1$}%
}}}}
\put(16202,-6564){\makebox(0,0)[lb]{\smash{{\SetFigFont{17}{20.4}{\rmdefault}{\mddefault}{\updefault}{$b_1$}%
}}}}
\put(6189,-18195){\makebox(0,0)[lb]{\smash{{\SetFigFont{17}{20.4}{\rmdefault}{\mddefault}{\updefault}{$a_2$}%
}}}}
\put(8229,-18210){\makebox(0,0)[lb]{\smash{{\SetFigFont{17}{20.4}{\rmdefault}{\mddefault}{\updefault}{$a_2$}%
}}}}
\put(19249,-18348){\makebox(0,0)[lb]{\smash{{\SetFigFont{17}{20.4}{\rmdefault}{\mddefault}{\updefault}{$a_2$}%
}}}}
\put(21334,-18370){\makebox(0,0)[lb]{\smash{{\SetFigFont{17}{20.4}{\rmdefault}{\mddefault}{\updefault}{$a_2$}%
}}}}
\put(21560,-7066){\makebox(0,0)[lb]{\smash{{\SetFigFont{17}{20.4}{\rmdefault}{\mddefault}{\updefault}{$a_2$}%
}}}}
\put(19115,-7059){\makebox(0,0)[lb]{\smash{{\SetFigFont{17}{20.4}{\rmdefault}{\mddefault}{\updefault}{$a_2$}%
}}}}
\put(7171,-17685){\makebox(0,0)[lb]{\smash{{\SetFigFont{17}{20.4}{\rmdefault}{\mddefault}{\updefault}{$a_1$}%
}}}}
\put(20231,-17823){\makebox(0,0)[lb]{\smash{{\SetFigFont{17}{20.4}{\rmdefault}{\mddefault}{\updefault}{$a_1$}%
}}}}
\put(20337,-6399){\makebox(0,0)[lb]{\smash{{\SetFigFont{17}{20.4}{\rmdefault}{\mddefault}{\updefault}{$a_1$}%
}}}}
\put(23583,-15104){\rotatebox{300.0}{\makebox(0,0)[lb]{\smash{{\SetFigFont{17}{20.4}{\rmdefault}{\mddefault}{\updefault}{$y+a_1+b_1$}%
}}}}}
\put(16599,-16214){\rotatebox{60.0}{\makebox(0,0)[lb]{\smash{{\SetFigFont{17}{20.4}{\rmdefault}{\mddefault}{\updefault}{$y+a_1+b_1$}%
}}}}}
\put(3846,-16361){\rotatebox{60.0}{\makebox(0,0)[lb]{\smash{{\SetFigFont{17}{20.4}{\rmdefault}{\mddefault}{\updefault}{$y+a_1+b_1$}%
}}}}}
\put(3183,-20093){\rotatebox{300.0}{\makebox(0,0)[lb]{\smash{{\SetFigFont{17}{20.4}{\rmdefault}{\mddefault}{\updefault}{$y+z+2a_2+2b_2$}%
}}}}}
\put(16025,-20471){\rotatebox{300.0}{\makebox(0,0)[lb]{\smash{{\SetFigFont{17}{20.4}{\rmdefault}{\mddefault}{\updefault}{$y+z+2a_2+2b_2$}%
}}}}}
\put(15921,-9205){\rotatebox{300.0}{\makebox(0,0)[lb]{\smash{{\SetFigFont{17}{20.4}{\rmdefault}{\mddefault}{\updefault}{$y+z+2a_2+2b_2$}%
}}}}}
\put(2743,-18172){\makebox(0,0)[lb]{\smash{{\SetFigFont{17}{20.4}{\rmdefault}{\mddefault}{\updefault}{$z$}%
}}}}
\put(11766,-18090){\makebox(0,0)[lb]{\smash{{\SetFigFont{17}{20.4}{\rmdefault}{\mddefault}{\updefault}{$z$}%
}}}}
\put(15391,-18190){\makebox(0,0)[lb]{\smash{{\SetFigFont{17}{20.4}{\rmdefault}{\mddefault}{\updefault}{$z$}%
}}}}
\put(25336,-18220){\makebox(0,0)[lb]{\smash{{\SetFigFont{17}{20.4}{\rmdefault}{\mddefault}{\updefault}{$z$}%
}}}}
\put(25622,-6901){\makebox(0,0)[lb]{\smash{{\SetFigFont{17}{20.4}{\rmdefault}{\mddefault}{\updefault}{$z$}%
}}}}
\put(15197,-6924){\makebox(0,0)[lb]{\smash{{\SetFigFont{17}{20.4}{\rmdefault}{\mddefault}{\updefault}{$z$}%
}}}}
\put(10336,-21787){\rotatebox{60.0}{\makebox(0,0)[lb]{\smash{{\SetFigFont{17}{20.4}{\rmdefault}{\mddefault}{\updefault}{$y+z+2a_2+2b_2$}%
}}}}}
\put(23674,-22240){\rotatebox{60.0}{\makebox(0,0)[lb]{\smash{{\SetFigFont{17}{20.4}{\rmdefault}{\mddefault}{\updefault}{$y+z+2a_2+2b_2$}%
}}}}}
\put(10335,-15416){\rotatebox{300.0}{\makebox(0,0)[lb]{\smash{{\SetFigFont{17}{20.4}{\rmdefault}{\mddefault}{\updefault}{$y+a_1+b_1$}%
}}}}}
\put(24057,-10846){\rotatebox{60.0}{\makebox(0,0)[lb]{\smash{{\SetFigFont{17}{20.4}{\rmdefault}{\mddefault}{\updefault}{$y+z+2a_2+2b_2$}%
}}}}}
\put(6617,-23520){\makebox(0,0)[lb]{\smash{{\SetFigFont{17}{20.4}{\rmdefault}{\mddefault}{\updefault}{$x+a_1+2b_1$}%
}}}}
\put(19838,-23973){\makebox(0,0)[lb]{\smash{{\SetFigFont{17}{20.4}{\rmdefault}{\mddefault}{\updefault}{$x+a_1+2b_1$}%
}}}}
\put(19828,-12691){\makebox(0,0)[lb]{\smash{{\SetFigFont{17}{20.4}{\rmdefault}{\mddefault}{\updefault}{$x+a_1+2b_1$}%
}}}}
\put(6504,-14242){\makebox(0,0)[lb]{\smash{{\SetFigFont{17}{20.4}{\rmdefault}{\mddefault}{\updefault}{$x+2a_2+2b_2$}%
}}}}
\put(19706,-13998){\makebox(0,0)[lb]{\smash{{\SetFigFont{17}{20.4}{\rmdefault}{\mddefault}{\updefault}{$x+2a_2+2b_2$}%
}}}}
\put(7263,-6366){\makebox(0,0)[lb]{\smash{{\SetFigFont{17}{20.4}{\rmdefault}{\mddefault}{\updefault}{$a_1$}%
}}}}
\put(8501,-7086){\makebox(0,0)[lb]{\smash{{\SetFigFont{17}{20.4}{\rmdefault}{\mddefault}{\updefault}{$a_2$}%
}}}}
\put(6041,-7091){\makebox(0,0)[lb]{\smash{{\SetFigFont{17}{20.4}{\rmdefault}{\mddefault}{\updefault}{$a_2$}%
}}}}
\put(11036,-6546){\makebox(0,0)[lb]{\smash{{\SetFigFont{17}{20.4}{\rmdefault}{\mddefault}{\updefault}{$b_1$}%
}}}}
\put(10111,-7106){\makebox(0,0)[lb]{\smash{{\SetFigFont{17}{20.4}{\rmdefault}{\mddefault}{\updefault}{$b_2$}%
}}}}
\put(4374,-7061){\makebox(0,0)[lb]{\smash{{\SetFigFont{17}{20.4}{\rmdefault}{\mddefault}{\updefault}{$b_2$}%
}}}}
\put(3479,-6556){\makebox(0,0)[lb]{\smash{{\SetFigFont{17}{20.4}{\rmdefault}{\mddefault}{\updefault}{$b_1$}%
}}}}
\put(6696,-2876){\makebox(0,0)[lb]{\smash{{\SetFigFont{17}{20.4}{\rmdefault}{\mddefault}{\updefault}{$x+2a_2+2b_2$}%
}}}}
\put(10306,-3876){\rotatebox{300.0}{\makebox(0,0)[lb]{\smash{{\SetFigFont{17}{20.4}{\rmdefault}{\mddefault}{\updefault}{$y+a_1+b_1$}%
}}}}}
\put(12091,-6886){\makebox(0,0)[lb]{\smash{{\SetFigFont{17}{20.4}{\rmdefault}{\mddefault}{\updefault}{$z$}%
}}}}
\put(10576,-10716){\rotatebox{60.0}{\makebox(0,0)[lb]{\smash{{\SetFigFont{17}{20.4}{\rmdefault}{\mddefault}{\updefault}{$y+z+2a_2+2b_2$}%
}}}}}
\put(6691,-12256){\makebox(0,0)[lb]{\smash{{\SetFigFont{17}{20.4}{\rmdefault}{\mddefault}{\updefault}{$x+a_1+2b_1$}%
}}}}
\put(3062,-9006){\rotatebox{300.0}{\makebox(0,0)[lb]{\smash{{\SetFigFont{17}{20.4}{\rmdefault}{\mddefault}{\updefault}{$y+z+2a_2+2b_2$}%
}}}}}
\put(2629,-6941){\makebox(0,0)[lb]{\smash{{\SetFigFont{17}{20.4}{\rmdefault}{\mddefault}{\updefault}{$z$}%
}}}}
\put(3849,-4929){\rotatebox{60.0}{\makebox(0,0)[lb]{\smash{{\SetFigFont{17}{20.4}{\rmdefault}{\mddefault}{\updefault}{$y+a_1+b_1$}%
}}}}}
\put(19887,-2439){\makebox(0,0)[lb]{\smash{{\SetFigFont{17}{20.4}{\rmdefault}{\mddefault}{\updefault}{$x+2a_2+2b_2$}%
}}}}
\end{picture}%
}
\caption{Dividing the symmetric hexagon with three ferns removed $S^{(1)}_{x,y,z}(\textbf{a},\textbf{b})$ into two halved hexagons.}\label{fig:halvedhex5}
\end{figure}

We conclude this section by presenting the proof of Theorem \ref{mainthm1}. The proof of Theorem \ref{mainthm2} is essentially the same and is also  omitted.
\begin{proof}[Proof of Theorem \ref{mainthm1}]
Apply the cutting procedure in Ciucu's Factorization Theorem (Lemma \ref{ciucufactor}) to the dual graph $G$ of the symmetric hexagon with three ferns removed
$S^{(1)}_{x,y,z} (\textbf{a}; \textbf{b})$, we have
\begin{equation}
\M(S^{(1)}_{x,y,z} (\textbf{a}; \textbf{b}))=2^{y+z+a+b-a_1}\M(G^{+}) \M(G^{-}),
\end{equation}
as there are exactly $2y+2z+2a+2b-2a_1$ vertices of $G$ lying on the vertical symmetry axis.

We first consider the case when $a_1,x,y$ are all even. It is easy to see that the right component graph $G^-$ is congruent to the dual graph of the left subregion restricted by the bold contour in Figure \ref{fig:halvedhex5}(a)
(for the case $x=y=z=2$, $a_1=4, a_2=2$, $b_1=b_2=2$; the lozenges with shaded cores are weighted by $1/2$).  This region is exactly the weighted region $W^{(2)}_{\frac{x}{2},\frac{y}{2},z}(\frac{a_1}{2},a_2,\dots,a_m;\ \textbf{b})$.
The left component graph $G^+$ in turn corresponds to the dual graph of the right subregion. By removing several forced lozenges on the top and the bottom of this region, we get the region
 $H^{(2)}_{\frac{x}{2},\frac{y}{2},z}(\frac{a_1}{2},a_2,\dots,a_m;\ \textbf{b})$. This means that we get
 \begin{align}
\M(S^{(1)}_{x,y,z} (\textbf{a}; \textbf{b}))=2^{y+z+a+b-a_1}&\M\left(H^{(2)}_{\frac{x}{2},\frac{y}{2},z}\left(\frac{a_1}{2},a_2,\dots,a_m;\ \textbf{b}\right)\right)\notag\\
 &\times\M\left(W^{(2)}_{\frac{x}{2},\frac{y}{2},z}\left(\frac{a_1}{2},a_2,\dots,a_m;\ \textbf{b}\right)\right).
\end{align}
If $a_1$ is even, but $x$ and $y$ are odd, then as shown in Figure \ref{fig:halvedhex5}(b), we get
 \begin{align}
\M(S^{(1)}_{x,y,z} (\textbf{a}; \textbf{b}))=2^{y+z+a+b-a_1}&\M\left(H^{(2)}_{\frac{x-1}{2},\frac{y+1}{2},z}\left(\frac{a_1}{2},a_2,\dots,a_m;\ \textbf{b}\right)\right)\notag\\
 &\times\M\left(W^{(2)}_{\frac{x-1}{2},\frac{y+1}{2},z}\left(\frac{a_1}{2},a_2,\dots,a_m;\ \textbf{b}\right)\right).
\end{align}

Similarly, when $a_1$ is odd and $x,y$ are even, we get from Figure \ref{fig:halvedhex5}(c)
 \begin{align}
\M(S^{(1)}_{x,y,z} (\textbf{a}; \textbf{b}))=2^{y+z+a+b-a_1}&\M\left(N^{(4)}_{\frac{x}{2},\frac{y}{2},z}\left(\frac{a_1+1}{2},a_2,\dots,a_m;\ \textbf{b}\right)\right)\notag\\
& \times \M\left(N^{(1)}_{\frac{x}{2},\frac{y}{2},z}\left(\frac{a_1-1}{2},a_2,\dots,a_m;\ \textbf{b}\right)\right).
\end{align}
Finally, if $a_1,x,y$ are all odd, then, as shown in Figure \ref{fig:halvedhex5}(d), we obtain
 \begin{align}
\M(S^{(1)}_{x,y,z} (\textbf{a}; \textbf{b}))=2^{y+z+a+b-a_1}&\M\left(N^{(4)}_{\frac{x-1}{2},\frac{y+1}{2},z}(\frac{a_1+1}{2},a_2,\dots,a_m;\ \textbf{b})\right)\notag\\
 &\times \M\left(N^{(1)}_{\frac{x+1}{2},\frac{y-1}{2},z}\left(\frac{a_1-1}{2},a_2,\dots,a_m;\ \textbf{b}\right)\right).
\end{align}
\end{proof}

\section{Several open questions}\label{sec:Question}

The first equality signs in the formulas of Theorems \ref{main1}--\ref{mainMR4} show some factorizations for the numbers tilings of halved hexagons.
It is interesting to find a \emph{direct bijective} explanation for these factorizations.

As shown in Theorems \ref{mainthm1} and \ref{mainthm2}, only eight over sixteen halved hexagons
 (in particular, the $H^{(2)}$-, $W^{(1)}$-,  $R^{(2)}$, $RW^{(1)}$-, $N^{(1)}$-, $N^{(4)}$-, $NR^{(1)}$-, $NR^{(4)}$-types regions)
 are really halves of some symmetric hexagons with three ferns removed.
  How about the other eight?
Are there any regions whose halves are corresponding to these remaining regions?


\begin{thebibliography}{20}

\bibitem{Andrews}
  G. E. Andrews, Plane partitions (III): The weak Macdonald
conjecture, {\it Invent. Math.} {\bf 53} (1979), 193--225.


\bibitem{Ciucu3}
M. Ciucu, \emph{Enumeration of perfect matchings in graphs with reflective symmetry}, J. Combin. Theory Ser. A \textbf{77} (1997), 67--97.

\bibitem{Ciucu1}  M. Ciucu, \emph{Plane partition I: A generalization of MacMahon's formula}, Memoirs of Amer. Math. Soc., \textbf{178} (2005), no. 839, 107--144.

\bibitem{Ciucu2} M. Ciucu, \emph{Another dual of MacMahon's theorem on plane partitions}, arXiv:1509.06421.


\bibitem{Cutoff} M. Ciucu and C. Krattenthaler, \emph{Enumeration of lozenge tilings of hexagons with cut off corners}, J. Combin. Theory Ser. A, \textbf{100} (2002), 201--231.




\bibitem{CL}
M. Ciucu and T. Lai, \emph{Lozenge tilings of doubly-intruded hexagons}, Preprint \texttt{http://arxiv.org/abs/1712.08024}.








\bibitem{KGV}
C. Krattenthaler, A. J. Guttmann, and X. G. Viennot, \emph{Vicious walkers, friendly walkers and Young
tableaux II: with a wall}, J. Phys. A: Math. Gen. \textbf{33} (2000), 8835--8866.

\bibitem{KKZ}
  C. Koutschan, M. Kauers and D. Zeilberger, A proof of George
Andrews' and David Robbins' $q$-TSPP-conjecture, {\it Proc. Natl. Acad. Sci. USA} {\bf 108} (2011), 2196--2199.


\bibitem{Kuo}
E. H. Kuo,
\emph{Applications of Graphical Condensation for Enumerating Matchings and Tilings},
Theor. Comput. Sci. \textbf{319} (2004),
29--57.

\bibitem{Kup}
G. Kuperberg, \emph{Symmetries of plane partitions and the permanent-determinant method},  J. Combin. Theory Ser. A  \textbf{68} (1994), 115--151.

\bibitem{JP} W. Jockusch and J. Propp, \emph{Antisymmetric monotone triangles and domino tilings of
quartered Aztec diamonds}, Unpublished work.

\bibitem{Lai}
T. Lai, \emph{Enumeration of tilings of quartered Aztec rectangles},
Electron. J. Combin. \textbf{21} (4), \#P4.46.

\bibitem{Lai3}
T. Lai, \emph{A new proof for the number of lozenge tilings of quartered hexagons},
Discrete Math \textbf{338} (2015), 1866--1872.

\bibitem{Lai1q}
T. Lai,
\emph{A $q$-enumeration of lozenge tilings of a hexagon with three dents}, Adv. Applied Math \textbf{82} (2017), 23--57.

\bibitem{Lai2q}
T. Lai,
\emph{A $q$-enumeration of lozenge tilings of a hexagon with four adjacent triangles removed from the boundary}, European J. Combin. \textbf{64} (2017), 66--87.

\bibitem{Halfhex1}
T. Lai, \emph{Lozenge Tilings of a Halved Hexagon with an Array of Triangles Removed from the Boundary}, SIAM Discrete Math. \textbf{32}(1) (2018),  783--814.

\bibitem{Halfhex2}
T. Lai, \emph{Lozenge Tilings of a Halved Hexagon with an Array of Triangles Removed from the Boundary, Part II}, Electron. J. Combin., \textbf{25}(4) (2018), \# P.4.58 (34 pp).


\bibitem{Mac} P. A. MacMahon, Memoir on  the theory of the partition of numbers---Part V. Partition in two-dimensional space, \emph{Phil. Trans. R. S.}, 1911, A.



\bibitem{Proc}
R. Proctor, \emph{Odd symplectic groups}, Inven. Math. \textbf{92}(2) (1988), 307--332.

\bibitem{Propp}
J. Propp,
\emph{Enumeration of matchings: Problems and progress},
New Perspectives in Geometric Combinatorics,  Cambridge Univ. Press, 1999, 255--291.


\bibitem{Ranjan}
R. Rohatgi, \emph{Enumeration of lozenge tilings of halved hexagons with a boundary defect}, Electron. J. Combin.  \textbf{22}(4) (2015), \#P4.22.

\bibitem{Stem}
J. R. Stembridge, \emph{Nonintersecting paths, Pfaffians and plane partitions}, Adv. Math. \textbf{83} (1990), 96--131.

\bibitem{Stanley}
R. Stanley, \emph{Symmetries of plane partitions}, J. Combin. Theory Ser. A \textbf{43} (1986), 103--113.



\end{thebibliography}
\end{document}